\documentclass[12pt]{amsart}

\usepackage[letterpaper,margin=.6in]{geometry}
\usepackage[colorlinks,linkcolor=blue,citecolor=black!75!red]{hyperref}
\usepackage[matrix,arrow,ps,color,line,curve,frame,all]{xy}

\usepackage{url}
\usepackage{color}

\usepackage[all]{xypic}

\usepackage{latexsym}

\usepackage{amssymb}
\usepackage{amsfonts}
\usepackage{amscd}
\usepackage{amsmath,amsthm}
\usepackage{verbatim}

\usepackage{tikz}
\usetikzlibrary{matrix,arrows,decorations.pathmorphing,fit,cd}
\tikzset{myboxgroup/.style={draw, densely dotted}} 

\usepackage{cleveref}

\newtheorem{lemma}{Lemma}[section]
\newtheorem{proposition}[lemma]{Proposition}
\newtheorem{theorem}[lemma]{Theorem}
\newtheorem{corollary}[lemma]{Corollary}

{

}

\theoremstyle{definition}

\newtheorem{definition}[lemma]{Definition}

\theoremstyle{remark}

\makeatletter
\let\xx@thm\@thm
\AtBeginDocument{\let\@thm\xx@thm}
\makeatother

\crefname{section}{Section}{Sections}
\crefformat{section}{#2section~#1#3}
\Crefformat{section}{#2Section~#1#3}

\crefformat{subsection}{#2\S#1#3}
\Crefformat{subsection}{#2\S#1#3}
\crefrangeformat{subsection}{\S\S#3#1#4--#5#2#6}
\Crefrangeformat{subsection}{\S\S#3#1#4--#5#2#6}
\crefmultiformat{subsection}{\S\S#2#1#3}{ and~#2#1#3}{, #2#1#3}{ and~#2#1#3}
\Crefmultiformat{subsection}{\S\S#2#1#3}{ and~#2#1#3}{, #2#1#3}{ and~#2#1#3}
\crefrangemultiformat{subsection}{\S\S#3#1#4--#5#2#6}{ and~#3#1#4--#5#2#6}{, #3#1#4--#5#2#6}{ and~#3#1#4--#5#2#6}
\Crefrangemultiformat{subsection}{\S\S#3#1#4--#5#2#6}{ and~#3#1#4--#5#2#6}{, #3#1#4--#5#2#6}{ and~#3#1#4--#5#2#6}

\crefname{definition}{Definition}{Definitions}
\crefformat{definition}{#2Definition~#1#3}
\Crefformat{definition}{#2Definition~#1#3}

\crefname{definitionnodiamond}{Definition}{Definitions}
\crefformat{definitionnodiamond}{#2Definition~#1#3}
\Crefformat{definitionnodiamond}{#2Definition~#1#3}

\crefname{example}{Example}{Examples}
\crefformat{example}{#2Example~#1#3}
\Crefformat{example}{#2Example~#1#3}

\crefname{examplenodiamond}{Example}{Examples}
\crefformat{examplenodiamond}{#2Example~#1#3}
\Crefformat{examplenodiamond}{#2Example~#1#3}

\crefname{remark}{Remark}{Remarks}
\crefformat{remark}{#2Remark~#1#3}
\Crefformat{remark}{#2Remark~#1#3}

\crefname{remarknodiamond}{Remark}{Remarks}
\crefformat{remarknodiamond}{#2Remark~#1#3}
\Crefformat{remarknodiamond}{#2Remark~#1#3}

\crefname{convention}{Convention}{Conventions}
\crefformat{convention}{#2Convention~#1#3}
\Crefformat{convention}{#2Convention~#1#3}

\crefname{notation}{Notation}{Notations}
\crefformat{notation}{#2Notation~#1#3}
\Crefformat{notation}{#2Notation~#1#3}

\crefname{notationnodiamond}{Notation}{Notations}
\crefformat{notationnodiamond}{#2Notation~#1#3}
\Crefformat{notationnodiamond}{#2Notation~#1#3}

\crefname{lemma}{Lemma}{Lemmas}
\crefformat{lemma}{#2Lemma~#1#3}
\Crefformat{lemma}{#2Lemma~#1#3}

\crefname{proposition}{Proposition}{Propositions}
\crefformat{proposition}{#2Proposition~#1#3}
\Crefformat{proposition}{#2Proposition~#1#3}

\crefname{corollary}{Corollary}{Corollaries}
\crefformat{corollary}{#2Corollary~#1#3}
\Crefformat{corollary}{#2Corollary~#1#3}

\crefname{theorem}{Theorem}{Theorems}
\crefformat{theorem}{#2Theorem~#1#3}
\Crefformat{theorem}{#2Theorem~#1#3}

\crefname{assumption}{Assumption}{Assumptions}
\crefformat{assumption}{#2Assumption~#1#3}
\Crefformat{assumption}{#2Assumption~#1#3}

\crefname{enumi}{}{}
\crefformat{enumi}{(#2#1#3)}
\Crefformat{enumi}{(#2#1#3)}

\crefname{equation}{}{}
\crefformat{equation}{(#2#1#3)}
\Crefformat{equation}{(#2#1#3)}

\crefname{align}{}{}
\crefformat{align}{(#2#1#3)}
\Crefformat{align}{(#2#1#3)}

\crefname{proofstep}{Step}{Steps}
\crefformat{proofstep}{#2Step~#1#3}
\Crefformat{proofstep}{#2Step~#1#3}

\crefname{table}{Table}{Tables}
\crefformat{table}{#2Table~#1#3}
\Crefformat{table}{#2Table~#1#3}

\makeatletter
\renewcommand{\p@enumii}{}
\renewcommand{\p@enumiii}{}
\makeatother

\numberwithin{equation}{section}

\newenvironment{pf}{\proof}{\endproof}

\def\CC{{\mathbb C}}

\def\PP{{\mathbb P}}

\def\ZZ{{\mathbb Z}}

\newcommand{\bbZ}{\mathbb{Z}}

\newcommand{\bbC}{\mathbb{C}}
\newcommand{\bbP}{\mathbb{P}}

\def\0ol{{\bar 0}}
\def\1ol{{\bar 1}}
\def\2ol{{\bar 2}}
\def\ol2{{\bar 2}}
\def\3ol{{\bar 3}}
\def\4ol{{\bar 4}}
\def\5ol{{\bar 5}}
\def\6ol{{\bar 6}}
\def\7ol{{\bar 7}}
\def\8ol{{\bar 8}}
\def\9ol{{\bar 9}}

\def\bold0{{\bf 0}}
\def\bold1{{\bf 1}}
\def\bold2{{\bf 2}} 
\def\bold3{{\bf  3}}
\def\bold4{{\bf 4}}
\def\bold5{{\bf 5}}
\def\bold6{{\bf 6}}
\def\bold7{{\bf 7}}
\def\bold8{{\bf 8}}
\def\bold9{{\bf 9}}

\def\P2Skly{\PP^2_{Skly}}

\def\a{\alpha}
\def\b{\beta}

\def\d{\delta}

\def\l{\lambda}

\def\s{\sigma}

\def\D{\Delta}

\def\L{\Lambda}

\def\fa{{\mathfrak a}}

\def\sfi{{\sf i}}

\def\sfz{{\sf z}}

\def\cal{\mathcal}

\def\cL{{\cal L}}

\def\cO{{\cal O}}

\def\cS{{\cal S}}

\def\End{\operatorname {End}}

\def\Hom{\operatorname {Hom}}
\def\Aut{\operatorname{Aut}}

\def\id{\operatorname{id}}

\def\rank{\operatorname{rank}}
\def\Grass{\operatorname{Grass}}
\def\rel{\operatorname{rel}}
\def\Sym{\operatorname{Sym}}
\def\Alt{\operatorname{Alt}}

\def\im{\operatorname{im}}

\renewcommand{\Im}{\operatorname{Im}}

\renewcommand{\colon}{:}

\def\Gr{{\sf Gr}}

\def\Qcoh{{\sf Qcoh}}

\def\dirlim{\mathop{\vtop{\baselineskip -100pt\lineskip -1pt\lineskiplimit 0pt
\setbox0\hbox{lim}\copy0\hbox to \wd0{\rightarrowfill}}}\limits}
\def\invlim{\mathop{\vtop{\baselineskip -100pt\lineskip -1pt\lineskiplimit 0pt
\setbox0\hbox{lim}\copy0\hbox to \wd0{\leftarrowfill}}}\limits}

\def\I11{{1 \kern -0.8pt \! \mbox{l}}}
\def\mumu{{\mu\kern-4.2pt\mu}}
\def\bfmu{{\mu\kern-4.2pt\mu}}
\def\2slash{\backslash \! \backslash}

\makeatletter
\def\l@subsection{\@tocline{2}{0pt}{2.75pc}{5pc}{}}
\makeatother

\begin{document}

\title[Elliptic algebras]{Feigin and Odesskii's elliptic algebras}

\author[Alex Chirvasitu]{Alex Chirvasitu}
\author[Ryo Kanda]{Ryo Kanda}
\author[S. Paul Smith]{S. Paul Smith}

\address[Alex Chirvasitu]{Department of Mathematics, University at
  Buffalo, Buffalo, NY 14260-2900, USA.}
\email{achirvas@buffalo.edu}

\address[Ryo Kanda]{Department of Mathematics, Graduate School of Science, Osaka City University, 3-3-138, Sugimoto, Sumiyoshi, Osaka, 558-8585, Japan.}
\email{ryo.kanda.math@gmail.com}

\address[S. Paul Smith]{Department of Mathematics, Box 354350,
  University of Washington, Seattle, WA 98195, USA.}
\email{smith@math.washington.edu}



\begin{abstract}
We study the elliptic algebras $Q_{n,k}(E,\tau)$ introduced by Feigin and Odesskii as a generalization of Sklyanin algebras. They form a family of quadratic algebras parametrized by coprime integers $n>k\geq 1$, an elliptic curve $E$, and a point $\tau\in E$. We consider and compare several different definitions of the algebras and provide proofs of various statements about them made by Feigin and Odesskii. For example, we show that $Q_{n,k}(E,0)$, and $Q_{n,n-1}(E,\tau)$ are polynomial rings on $n$ variables. We also show that $Q_{n,k}(E,\tau+\zeta)$ is a twist of $Q_{n,k}(E,\tau)$ when $\zeta$ is an $n$-torsion point. This paper is the first of several we are writing about the algebras $Q_{n,k}(E,\tau)$.
\end{abstract}

\maketitle

\tableofcontents

\section{Introduction}
\label{sec.Intro}

\subsection{Notation and conventions}
\label{sect.notn.and.conventions}

Throughout this paper we use the notation $e(z)=e^{2\pi i z}$ for $z\in\bbC$.

We fix relatively prime integers, $n>k\ge 1$, and write $k'$ for the unique integer such that $n >k' \ge 1$ and $kk'=1$ in $\ZZ_n=\ZZ/n\ZZ$.

We fix a point $\eta \in \CC$ lying in the upper half-plane, 
the lattice $\L=\ZZ+\ZZ\eta$, and the elliptic curve $E=\CC/\L$. We write $E[n]$ for the
$n$-torsion subgroup, $\frac{1}{n}\Lambda/\Lambda$, of $E$.

We always work over the field $\CC$ of complex numbers unless otherwise specified. For a complex algebraic variety $X$, $x\in X$ means $x$ is a \emph{closed} point of $X$.

\subsection{The algebras $Q_{n,k}(E,\tau)$}

In 1989, Feigin and Odesskii defined a family of graded $\CC$-algebras $Q_{n,k}(E,\tau)$ depending on the data $(n,k,E)$ and 
a point $\tau \in \CC-\frac{1}{n}\L$. The algebras appear first in their manuscript  \cite{FO-Kiev} 
archived with the Academy of Science of the Ukrainian SSR (which we refer to as ``the Kiev preprint'') and, almost simultaneously, in  their published paper \cite{FO89}. 
They defined $Q_{n,k}(E,\tau)$ to be the free algebra $\CC\langle x_0,\ldots,x_{n-1}\rangle$ 
modulo the $n^2$ homogeneous quadratic relations\footnote{The original definition uses $x_{k(j-r)}x_{k(i+r)}$ instead of $x_{j-r}x_{i+r}$; see \cref{sssec.Rij.def.k}.}
\begin{equation}
\label{the-relns-1}
	r_{ij}\;=\; r_{ij}(\tau) \; =\; \sum_{r \in \ZZ_n} \frac{\theta_{j-i+(k-1)r}(0)}{\theta_{j-i-r}(-\tau)\theta_{kr}(\tau)} \,\, x_{j-r}x_{i+r}	
\end{equation}
where the indices $i$ and $j$ belong to $\ZZ_n$ and
$\theta_0,\ldots,\theta_{n-1}$ are certain theta functions of order $n$, also indexed by $\ZZ_n$, that are quasi-periodic with respect to the lattice $\L$.  
The quasi-periodicity properties imply that if $\l \in \Lambda$, then $r_{ij}(\tau+\l)$ is a non-zero scalar multiple of $r_{ij}(\tau)$ 
whence $Q_{n,k}(E,\tau)$ depends only on the image of $\tau$ in $E$;  thus, for fixed $(n,k,E)$ the algebras provide a family parametrized by $E-E[n]$. 

When $\tau \in \frac{1}{n}\L$, $\theta_{kr}(\tau)=0$ for some $r$ so the relations no longer make sense. In 
\cref{subsec.rel.tor.pts} we will show how to define  $Q_{n,k}(E,\tau)$ for all $\tau \in \CC$ (\cref{def.Qnk}). 
Using that definition, \cref{prop.qnk.poly} shows that  $Q_{n,k}(E,0)$ is a polynomial ring on $n$ variables for all $n$ and $k$.

A lot is known about the algebras $Q_{n,1}(E,\tau)$. In \cite{TvdB96}, Tate and Van den Bergh showed that $Q_{n,1}(E,\tau)$
is a noetherian domain having the same Hilbert series and the same homological properties as the polynomial ring 
on $n$ variables. The algebras $Q_{3,1}(E,\tau)$ and $Q_{4,1}(E,\tau)$  are well understood due to the work of Artin-Tate-Van den Bergh (\cite{ATV1,ATV2}), Smith-Stafford \cite{SS92}, Levasseur-Smith \cite{LS93}, and Smith-Tate \cite{ST94}. 
For the most part though, the representation theory of $Q_{n,1}(E,\tau)$ remains a mystery when $n \ge 5$. 

Although the algebras $Q_{n,k}(E,\tau)$ were defined thirty years ago they have not been studied much since then 
(with the exception of the case $k=1$). The algebras $Q_{4,1}(E,\tau)$ were discovered by Sklyanin \cite{Skl82}  
almost 40 years ago when he was studying questions arising from quantum physics. We endorse
 a sentiment he expressed in that paper:
\begin{quote}
During our investigation it turned out that it is necessary to bring into the picture new algebraic structures, namely, the quadratic algebras of Poisson brackets and the quadratic generalization of the universal enveloping algebra of a Lie algebra. The theory of these mathematical objects is surprisingly reminiscent of the theory of Lie algebras, the difference being that it is more complicated. 
In our opinion, it deserves the greatest attention of mathematicians.
\end{quote}

In investigating the algebras $Q_{n,k}(E,\tau)$ one encounters an interesting mix of topics. A few examples:
\begin{itemize}
\item The origin of these algebras in the study of elliptic solutions of the quantum Yang-Baxter equation is evident in the appearance and prevalence of $R$-matrices with spectral parameter defining the relations of $Q_{n,k}(E,\tau)$.
\item Theta functions and the sometimes mysterious identities they satisfy pervade the subject.
\item When regarded as parametrized by $\tau$, the family $Q_{n,k}(E,\tau)$ ``integrates'' a natural Poisson structure on a moduli space of bundles on $E$ of rank $k$ and degree $n$ \cite{FO98,pl98}.
\item Understanding the point scheme for $Q_{n,k}(E,\tau)$ is heavily reliant on the intricacies of the theory of holomorphic bundles on abelian varieties.
\end{itemize}
We believe that this wide array of topics speaks to the depth of the subject and its richness as a source of problems, questions and perhaps answers. For that reason, we echo Sklyanin's opinion that the algebras $Q_{n,k}(E,\tau)$ deserve considerable attention.

\subsection{The contents of subsequent papers}
\label{sect.contents.other.papers}

This is the first of several papers in which we examine the algebras $Q_{n,k}(E,\tau)$. 
For the most part they can be read independently of one another.
One of them examines the characteristic variety $X_{n/k}$ for $Q_{n,k}(E,\tau)$, which is a subvariety of $\PP^{n-1}$.
Another will show that a certain quotient category of  graded $Q_{n,k}(E,\tau)$-modules contains 
a ``closed subcategory'' that is equivalent to 
$\Qcoh(X_{n/k})$, the category of quasi-coherent sheaves on $X_{n/k}$. This is proved  by exhibiting a homomorphism from
$Q_{n,k}(E,\tau)$ to a ``twisted homogeneous coordinate ring'' of $X_{n/k}$ (defined in \cite{AV90}).  
In many cases, $X_{n/k}$ is the $g$-fold product, $E^g$, of copies of $E$ where $g$ is the length of a certain continued fraction expression for the rational number $n/k$. For example, if $f_0=f_1=1$ and $f_{i+1}=f_i+f_{i-1}$ and $(n,k)=(f_{2g+1},f_{2g-1})$, then $X_{n/k} \cong E^g$. If $k=1$, then $g=1$ and $X_{n/1}\cong E$. 
If $n\geq 5$ and $k=2$, then $g=2$ and $X_{n/k}\cong S^2E$ the $2^{\rm nd}$ symmetric power of $E$. If $(n,k)=(n,n-1)$, then $g=n-1$ and $X_{n/k}\cong \PP^{n-1}$. 

It is stated in \cite[\S 3]{Od-survey} that, for generic $\tau\in E$, the  dimensions of the homogeneous components of 
$Q_{n,k}(E,\tau)$ are the same as those of the polynomial ring on $n$ variables, and it is conjectured that this is true for
 \emph{all} $\tau$.  When $k=1$, this was proved by Tate and Van den Bergh \cite{TvdB96}. In \cite{CKS4}, 
 we will show this is true for all $Q_{n,k}(E,\tau)$ when $\tau+\Lambda$ is not a torsion point in $E$.

\subsection{The contents of this paper}

The present paper is a prerequisite for our later papers.

In \cref{sec.ThetaFuncs}  (see \cref{official.theta_alpha}) 
 we specify a particular basis $\theta_0,\ldots,\theta_{n-1}$ for a space $\Theta_n(\L)$ of order-$n$ theta functions 
that are quasi-periodic with respect to $\L$. 
We use this basis in the rest of this paper and in our subsequent papers. 
Theta functions are notorious for the fact that notation for them varies considerably from one source to another.\footnote{Regarding the various notations for theta functions, the final paragraph of \cite[\S16.27]{handbook} provides this warning: 
``There is a bewildering variety of notations $\ldots$ so that in consulting books caution should be used''.}
Even when the same symbol appears in two different sources 
the reader must be alert to the possibility that the functions they denote are not the same. That is the case
in Feigin and Odesskii's various papers. For that reason, \cref{subsec.std.basis} makes a careful comparison 
of their various definitions and describes exactly how our $\theta_0,\ldots,\theta_{n-1}$  relate to their functions labeled by the same symbols. 
We then discuss the action of the Heisenberg group $H_n$ of order $n^3$ on $\Theta_n(\L)$ and the canonical morphism $E=\CC/\L \to \PP(\Theta_n(\L)^*)$ to the projective space of 1-dimensional subspaces of the dual space $\Theta_n(\L)^*$. 

In \cref{sect.Skl.defn.alg} we examine various definitions of $Q_{n,k}(E,\tau)$ and explain why they produce the same algebra.  
In \Cref{subsec.def.FO.alg} we compare different definitions given in terms of generators and relations.

In \cref{subsec.k=1.defns}, we focus on the case $k=1$. We use results of Feigin and Odesskii to give three alternative definitions of  
$Q_{n,1}(E,\tau)$ for {\it all}  $\tau \in \CC$. The first is based on their 
elliptic analogue of the usual shuffle product  for the symmetric algebra. The second, based on the theta function identity
$\text{\cref{42905872}}=\text{\cref{42905873}}$ in the proof of
 \Cref{prop.FO}, declares that $Q_{n,1}(E,\tau)$ is the algebra whose defining (quadratic) relations are the image 
an  explicit injective linear map $\Alt^2\Theta_{n}(\Lambda) \to \Theta_{n}(\Lambda)^{\otimes 2}$ where  $\Alt^2\Theta_{n}(\Lambda)$
denotes the space of anti-symmetric functions in $\Theta_{n}(\Lambda)^{\otimes 2}$. 
This is essentially the way Tate and Van den Bergh defined $Q_{n,1}(E,\tau)$ in \cite[(4.1)]{TvdB96} (see \cref{ssect.cf.TVdB}). The third, in \cref{ssect.geom.defn}, is of a geometric nature: the relations are defined as the subspace of 
$H^0(E \times E,\cL \boxtimes \cL)$, where $\cL$ is a certain invertible $\cO_E$-module of degree $n$, 
consisting of those sections $g$ such that 
$(g)_0$, its divisor of zeros, has certain symmetry properties. This definition allows one to define $Q_{n,1}(E,\tau)$ for arbitrary base fields
(see \cite{TvdB96}).

In \cref{subsec.rel.tor.pts}, we  define $Q_{n,k}(E,\tau)$ for all $\tau \in \CC$ and show that different definitions 
produce the same algebra under reasonable hypotheses.  To discuss this we define, for $\tau \in \CC- \frac{1}{n}\Lambda$, 
$$
\rel_{n,k}(E,\tau) \; := \; \operatorname{span}\{r_{ij}(\tau) \; | \; i,j \in \ZZ_n\},
$$
We examine three ways of defining $Q_{n,k}(E,\tau)$ for all $\tau \in \CC$.
\begin{enumerate}
\item 
(\cref{ssect.Lij})
If $r_{ij}(\tau) \ne 0$, let $L_{ij}(\tau)$ denote the point $\CC. r_{ij}(\tau)$ in $\PP\big( V^{\otimes 2} \big)$ and extend the holomorphic map 
$\CC-\frac{1}{n}\Lambda \to \PP\big( V^{\otimes 2} \big)$, $\tau \mapsto L_{ij}(\tau)$, 
 to $\CC \to \PP\big( V^{\otimes 2} \big)$  and define $L_{ij}(\tau)$ for all $\tau$ to be the image of $\tau$ under the extension; then define 
$$
\rel_{n,k}(E,\tau) \; :=\; \text{the linear span of all the $L_{ij}(\tau)$'s}.
$$
\item 
(\cref{ssect.R.tau.tau})
In \cref{eq:odr} we introduce, for $\tau \in \CC- \frac{1}{n}\Lambda$, a linear operator $R_\tau(\tau): V^{\otimes 2}  \to V^{\otimes 2}$ whose
image is $\operatorname{span}\{r_{ij}(\tau) \; | \; i,j \in \ZZ_n\}$; we then show that the holomorphic map 
$\tau \mapsto R_\tau(\tau)$ extends in a unique way to a holomorphic map $\CC \to \End_\CC (V^{\otimes 2})$;
 \cref{prop.relns.im.R.tau.tau} shows for all $\tau \in \CC$ that 
 $$
\text{the image of $R_\tau(\tau)$} \;=\;  \rel_{n,k}(E,\tau).
$$
\item 
(\cref{sect.subspaces})
In \cite{CKS4}, we will show  that $\dim\rel_{n,k}(E,\tau)=\binom{n}{2}$ for all $\tau \in \CC-\frac{1}{2n}\Lambda$;
the morphism $E-E[2n] \to \Grass  \big( \binom{n}{2},V^{\otimes 2} \big)$, $\tau \mapsto \rel_{n,k}(E,\tau)$, extends uniquely to a 
morphism $E\to \Grass\big( \binom{n}{2},V^{\otimes 2} \big)$;
one might then define $\rel_{n,k}(E,\tau)$ to be the image of $\tau+\Lambda$ under this extension.  
\end{enumerate}

In \cref{subsec.first.props}, we show that $Q_{n,k}(E,\tau) \cong Q_{n,k'}(E,\tau)$ where $k'$ is the unique integer such that $n >k' \ge 1$ and $kk'=1$ in $\bbZ_n$. Feigin and Odesskii state this but leave its proof to the reader. 
Feigin and Odesskii state several results without indicating how they might be proved. Some, like this isomorphism, 
are straightforward but we have had difficulty proving others. For that reason, and because the definition of the 
$\theta_\a$'s in one of their papers is not always the same as in others, we often provide more detail than strictly necessary.
The extra detail will provide a solid foundation for the future study of $Q_{n,k}(E,\tau)$.

For example, the statement that the only isomorphisms among the $Q_{n,k}(E,\tau)$'s
are those in the first sentence of the previous paragraph, \cite[\S 1, Rmk.~3]{FO89}, 
requires more precision because, for example, \cref{prop.n.minus.one} shows that 
 $Q_{n,n-1}(E,\tau)$ is a polynomial ring for all $\tau$. Furthermore,  
\cref{prop.anti.isom}  provides another isomorphism when $\tau$ is replaced by $-\tau$; indeed,
$Q_{n,k}(E,\tau) \cong Q_{n,k}(E,-\tau)=Q_{n,k}(E,\tau)^{\rm op}$. More isomorphisms appear in \cref{ssect.TVdB,ssect.more.isoms}.
We do not have a complete understanding
of all isomorphisms among the $Q_{n,k}(E,\tau)$'s.

In \cref{sec.twist.FO.alg} we show that $Q_{n,k}(E,\tau+\zeta)$  is isomorphic to a ``twist'' of $Q_{n,k}(E,\tau)$ for all $\zeta \in E[n]$.\footnote{The ``twist'' construction is quite general. Given any $\ZZ$-graded ring $A$ and a degree-preserving automorphism 
$\phi\colon A\to A$ the \textsf{twist} $A^{\phi}$ is the  graded vector space $A$ endowed with multiplication $a*b:= \phi^m(a)b$ when $b\in A_m$.
There is an equivalence $\Gr(A)\equiv \Gr(A^\phi)$ between their categories of graded left modules.}
The Heisenberg group $H_n$ acts as degree-preserving algebra automorphisms of $Q_{n,k}(E,\tau)$. There is a 
surjective homomorphism $H_n \to E[n]\cong\ZZ_n \times \ZZ_n$ and the twist just referred to is induced by any one of the 
automorphisms in $H_n$ that is a preimage of $\zeta$. Since $Q_{n,k}(E,0)$ is a polynomial ring on $n$ variables (\cref{prop.qnk.poly})
this confirms Feigin and Odesskii's statement \cite[\S1.2, Rmk.~1]{FO89}
that $Q_{n,k}(E,\zeta)$ is isomorphic to an algebra of ``skew polynomials'' though they don't define that term.

In \cref{sec.FO.alg.tors.pts}, we provide a proof of the assertion in \cite[\S1.2, Rmk.~1]{FO89} and \cite[\S3]{Od-survey} that 
$Q_{n,k}(E,0)$ is a polynomial ring on $n$ variables. 

In \Cref{sect.appx} we state and prove a lemma (a``standard'' result in complex analysis) that 
allows us to define what we mean by a theta function (in one variable) and establishes two fundamental results about such a function,
the number of its zeros in a fundamental parallelogram and the sum of those zeros. 
This lemma will also be used in our subsequent papers.

\subsection{Acknowledgements}

The authors are particularly grateful to Kevin De Laet for several useful conversations and for allowing us to include his result in
\cref{prop.new.rel}. \cref{prop.n.minus.one} and the observation in \cref{subsec.alt.basis} are also based on his work.

A.C. was partially supported through NSF grant DMS-1801011. 

R.K. was a JSPS Overseas Research Fellow, and supported by JSPS KAKENHI Grant Numbers JP16H06337, JP17K14164, and JP20K14288, Leading Initiative for Excellent Young Researchers, MEXT, Japan, and Osaka City University Advanced Mathematical Institute (MEXT Joint Usage/Research Center on Mathematics and Theoretical Physics JPMXP0619217849). R.K. would like to express his deep gratitude to Paul Smith for his hospitality as a host researcher during R.K.'s visit to the University of Washington.

\section{Theta functions in one variable}
\label{sec.ThetaFuncs}

In this section we collect some results on theta functions. 

The results are ``standard'' but we could not find a single source that states them in the way we need them;  for that reason
we have included them here. Proofs are given in more detail than strictly necessary because the calculations are often 
prone to error and the material will be new for some readers.

\subsection{The spaces $\Theta_n(\L)$ and the functions $\vartheta(z \, | \, \eta)$ and $\theta(z)$}
\label{ssect.theta.fns.one.var}

We fix an integer  $n \ge 1$ and a point $c \in \CC$.\footnote{Usually $n$ is  the integer fixed in \cref{sect.notn.and.conventions}
but we also allow $n=1$ here.}
We adopt the notation in Odesskii's survey article \cite[Appendix~A]{Od-survey},  and at \cite[p.~1025]{HP1}, and
write $\Theta_{n,c}(\L)$ for the set of holomorphic functions $f$ on $\CC$ satisfying the quasi-periodicity conditions
\begin{align*}
f(z+1) & \; = \; f(z),
\\
f(z+\eta) & \; = \; e\big(-nz+c+\tfrac{n}{2}\big) f(z).
\end{align*}
Functions in $\Theta_{n,c}(\L)$ are called {\sf theta functions of order $n$} with respect to the lattice $\L$.
They have $n$ zeros (always counted with multiplicity) in each fundamental parallelogram for $\Lambda$ and the sum
of those zeros is equal to $c$ modulo $\Lambda$ (see \Cref{sect.appx}).

\begin{proposition}
$\Theta_{n,c}(\L)$ is a vector space of dimension $n$.
\end{proposition}
\begin{proof}
This follows from the Fourier expansions for elements in $\Theta_{n,c}(\L)$. See \cite[I.\S 1]{Mum07}, for example.
\end{proof}

In keeping with the notation in the Kiev preprint \cite[p.~32]{FO-Kiev}, and in the first Odesskii-Feigin paper \cite[\S 1.1]{FO89}, we will always 
use the notation
\begin{equation*}
	\Theta_{n}(\L) \;:=\; \Theta_{n,\frac{n-1}{2}}(\L).
\end{equation*}
When $c=\frac{n-1}{2}$ the second quasi-periodicity condition becomes $f(z+\eta)=-e(-nz)f(z)$.

\subsubsection{}
\label{sssec.riem.theta}

All theta functions in this paper will be defined in terms of the holomorphic functions
$$
\vartheta(z\, | \, \eta)  \;:=\;  \sum_{n \in \ZZ}  e\big(nz + \tfrac{1}{2}n^2\eta\big)
$$
and $\theta(z) =\vartheta(z-\frac{1}{2}-\frac{1}{2}\eta \, | \, \eta)$ in \cref{eq.basic.theta.fn}. Both $\vartheta$ and $\theta$ 
have order one, meaning they have a single zero in each 
fundamental parallelogram. The Fourier expansion for $\theta(z)$ is given by \cref{eq.basic.theta.fn}. 

\begin{lemma}\label{lem.riemann.theta.prop}
The function 
\begin{equation}
\label{eq.basic.theta.fn}
\theta(z )  \; :=\;  \sum_{n \in \ZZ} (-1)^n e\big(nz + \tfrac{1}{2}n(n-1)\eta\big)
\end{equation}
has the following properties:
\begin{enumerate}
	\item\label{item.lem.riemann.theta.basis}
	it is a basis for $\Theta_{1,0}(\L)$;
	\item\label{item.lem.riemann.theta.per}
	$\theta(z+1)=\theta(z)$ and $\theta(z+\eta)=-e(-z)\theta(z)$;
	\item\label{item.lem.riemann.theta.minus}
	$\theta(-z)=-e(-z)\theta(z)$;
	\item\label{item.lem.riemann.theta.zeros}
	$\theta(z)=0$ if and only if $z\in\Lambda$. Each zero has order $1$.
\end{enumerate}
\end{lemma}
\begin{proof}
Statement \cref{item.lem.riemann.theta.basis}, and hence \cref{item.lem.riemann.theta.per}, follows from the fact that 
$\vartheta(z \, | \,\eta)$ is a basis for $\Theta_{1,-\frac{1}{2}-\frac{1}{2}\eta}(\Lambda)$, which can be found in \cite[\S I.1]{Mum07}.

It follows from the definition of $\theta$ that 
\begin{align*}
\theta(-z) & \;=\; \sum_{n \in \ZZ} (-1)^n e(-nz + \tfrac{1}{2}n(n-1)\eta)
\\
& \;=\;  \sum_n (-1)^{-n} e(-nz + \tfrac{1}{2}(-n)(-n+1)\eta)
\\
& \;=\;  \sum_m (-1)^{m-1} e(mz -z + \tfrac{1}{2}(m-1)m\eta)  \qquad \text{(after setting $m=-n+1$)}
\\
& \;=\;  -\, e(-z)\sum_m (-1)^{m} e(mz + \tfrac{1}{2}m(m-1)\eta) 
\\
& \;=\;  -\, e(-z)\theta(z)
\end{align*}
as claimed in \cref{item.lem.riemann.theta.minus}.

Statement \cref{item.lem.riemann.theta.zeros} follows from \cite[Lem.~4.1]{Mum07}: it is shown there that the zeros of 
$\vartheta_{00}(z)=\vartheta(z \, | \, \eta)$ are the points in $\frac{1}{2}+\frac{1}{2}\eta+\Lambda$ and those zeros have order $1$. 
Thus the zeros of $\theta(z)=\vartheta(z-\frac{1}{2}-\frac{1}{2}\eta \, | \,\eta)$ are the points in $\Lambda$ and they too have order $1$.
\end{proof}

\subsubsection{Remarks}
\label{rmk.prod.theta.fns}
Assume $c,d \in \CC$, $r\in\ZZ$, $f \in \Theta_{n,c}(\Lambda)$, and $f_i \in \Theta_{n_i,c_i}(\Lambda)$ for $i=1,2$.
\begin{enumerate}
\item
$\Theta_{n,c+r}(\Lambda)=\Theta_{n,c}(\Lambda)$. $\phantom{\big\vert}$
\item
The function $z\mapsto f_1(z)f_2(z)$ belongs to $\Theta_{n_1+n_2,c_1+c_2}(\Lambda)$. $\phantom{\big\vert}$
\item
The function $z \mapsto f(z+d)$ belongs to $\Theta_{n,c-nd}(\Lambda)$. $\phantom{\big\vert}$
\item
$\vartheta(z \, | \, \eta) \in \Theta_{1,\frac{1}{2}(1+\eta)}(z)$. $\phantom{\big\vert}$
\item
The function $z \mapsto f(rz)$ belongs to $\Theta_{r^{2}n,\,rc+\frac{r(1-r)n}{2}\eta}(\Lambda)$. $\phantom{\big\vert}$
\end{enumerate}

\subsection{The standard basis for $\Theta_n(\L)$}
\label{subsec.std.basis}

In their various papers Feigin and Odesskii use a basis for $\Theta_{n,c}(\L)$ that is labeled $\theta_0,\ldots,\theta_{n-1}$.
The functions they call $\theta_\a$ in one paper are not always the same as those called  $\theta_\a$ in
another paper.
Nevertheless, in \cite{FO-Kiev,FO89,Od-survey} the zeros of $\theta_\a$ always belong to
\begin{equation*}
	\big\{  \tfrac{1}{n}(-\a\eta +m) \; | \; 0 \le m \le n-1\big \} \, + \, \L \;=\; -\, \tfrac{\a}{n}\eta + \tfrac{1}{n}\ZZ + \ZZ\eta.
\end{equation*} 
In particular, $\theta_\a$ has $n$ distinct zeros in the fundamental parallelogram
\begin{equation*}
	[0,1)+(-1,0]\eta\;=\;\{a+b\eta\;|\;0\leq a<1,\ -1<b\leq 0\},
\end{equation*}
each zero having multiplicity 1. 
Furthermore, their $\theta_\a$'s, $\alpha\in\ZZ_{n}$, always have the properties
\begin{align*}
\theta_\a(z+\tfrac{1}{n}) & \; = \; e\big( \tfrac{\a}{n}\big)  \theta_\a(z),
\\
\theta_\a(z+\tfrac{1}{n}\eta)  &  \; = \; C e(-z) \theta_{\a+1}(z), 
\end{align*}
where $C$ is a non-zero constant independent of $\a$. 

Since $\theta$ has a unique zero  in the fundamental parallelogram, namely a simple zero at $z=0$,
the function 
\begin{equation}
\label{theta.alpha}
\theta\!\left(z +\tfrac{\a}{n}\eta\right)  \theta\!\left(z +\tfrac{1}{n}+ \tfrac{\a}{n}\eta\right)
\cdots \theta\!\left(z +\tfrac{n-1}{n}+ \tfrac{\a}{n}\eta\right)
\end{equation}
has exactly $n$ zeros in the fundamental parallelogram, namely $\{\frac{1}{n}(-\a\eta +m) \; | \; 0 \le m \le n-1\}$, 
each of which has order one. 
Thus, Feigin and Odesskii's functions $\theta_\a$, $\a \in \ZZ_n$, are multiples of the functions in \cref{theta.alpha}
by nowhere vanishing holomorphic functions.

\begin{lemma}\label{defn.R.theta.prop}
For each $\a \in \ZZ$ let $[\a] \in \CC$ be an arbitrary complex number.\footnote{Later we will make a judicious choice of $[\a]$. See \cref{official.theta_alpha} for the ``standard'' definition.}
The functions
\begin{equation}
\label{defn.R.theta.fn}
\theta_\a(z) \; :=\; e(\a z +[\a])  \prod_{m=0}^{n-1}  \theta\!\left(z +\tfrac{m}{n}+ \tfrac{\a}{n}\eta\right),
\end{equation}
indexed by $\alpha\in\ZZ$, have the following properties:
\begin{enumerate}
\item\label{item.defn.R.theta.fn.sp}
$\theta_\a \in \Theta_{n}(\L)$,
\item\label{item.defn.R.theta.fn.qp}
$\theta_\a(z+1) \; = \;  \theta_\a(z)$ and $\theta_\a(z+\eta) \; = \;  - e(-nz) \theta_\a(z)$,
\item\label{item.defn.R.theta.fn.S}
$\theta_\a(z+\tfrac{1}{n})  \; = \; e\big(\tfrac{\a}{n}\big) \theta_\a(z)$, 
\item\label{defn.R.theta.fn.per.sec}
$\theta_\a(z+\tfrac{1}{n}\eta)  \; = \;  e\big(\tfrac{\a}{n}\eta  +[\a]   -[\a+1]\big) e(-z) \theta_{\a+1}(z)$, 
\item\label{defn.R.theta.fn.neg}
$\theta_\a(-z)  \; = \;  -  e\big(-nz+\a\eta+[\a]-[-\a]\big) \theta_{-\a}(z)$, and 
\item\label{item.defn.R.theta.fn.wd}
$\theta_{\a+n}(z)  \; = \;    -e([\a+n]-[\a]-\a\eta)\theta_\a(z)$.
\end{enumerate}
\end{lemma}
\begin{pf}
\cref{item.defn.R.theta.fn.sp}
It follows from \cref{rmk.prod.theta.fns} that the function in \cref{theta.alpha} belongs to
$\Theta_{n,-\frac{n-1}{2}-\a\eta}(\Lambda)$ and hence $\theta_\a$ belongs to $\Theta_{n,-\frac{n-1}{2}}(\L)=\Theta_{n,\frac{n-1}{2}}(\L)$.

\cref{item.defn.R.theta.fn.qp}
This is a restatement of \cref{item.defn.R.theta.fn.sp}.

\cref{item.defn.R.theta.fn.S}
Since $\theta(z+1)=\theta(z)$, 
\begin{align*}
\theta_\a(z+\tfrac{1}{n}) & \;=\;  e(\a (z+\tfrac{1}{n})  +[\a])  \prod_{m=0}^{n-1}  \theta\!\left(z +\tfrac{1+m}{n}+ \tfrac{\a}{n}\eta\right)
\\
& \;=\;  e(\tfrac{\a}{n})   e(\a z  +[\a])   \theta\!\left(z +\tfrac{1}{n}+ \tfrac{\a}{n}\eta\right)
\ldots
 \theta\!\left(z +\tfrac{n-1}{n}+ \tfrac{\a}{n}\eta\right)
  \theta\!\left(z +\tfrac{n}{n}+ \tfrac{\a}{n}\eta\right)
  \\
& \;=\;  e(\tfrac{\a}{n})   e(\a z  +[\a])  
\prod_{m=0}^{n-1}  \theta\!\left(z +\tfrac{m}{n}+ \tfrac{\a}{n}\eta\right)
  \\
& \;=\;  e(\tfrac{\a}{n}) \theta_\a(z),
\end{align*}
as claimed.

\cref{defn.R.theta.fn.per.sec}
Similarly,
\begin{align*}
\theta_\a(z+\tfrac{1}{n}\eta) 
& \;=\;  e(\a (z+\tfrac{1}{n}\eta)  +[\a])  \prod_{m=0}^{n-1}  \theta\!\left(z +\tfrac{m}{n}+\tfrac{1+\a}{n}\eta\right)
\\
& \;=\;  e(\a (z+\tfrac{1}{n}\eta)  +[\a]) e(-(\a+1) z -[\a+1])  \theta_{\a+1}(z)
  \\
& \;=\;  e(\tfrac{\a}{n}\eta  +[\a]   -[\a+1]) e(-z) \theta_{\a+1}(z),
\end{align*}
as claimed.

\cref{defn.R.theta.fn.neg}
Since $\theta(-z)=-e(-z)\theta(z)$, 
\begin{align*}
\theta_\a(-z) & \; = \;  e(-\a z+[\a])
\prod_{m=0}^{n-1}  \theta\!\left(-z +\tfrac{m}{n}+ \tfrac{\a}{n}\eta\right)
\\
& \; = \; 
e(-\a z+[\a])
\prod_{m=0}^{n-1}   (-1) e( -z + \tfrac{m}{n}+ \tfrac{\a}{n}\eta)   \theta\!\left(z -\tfrac{m}{n}-  \tfrac{\a}{n}\eta\right)
\\
& \; = \; 
 (-1)^n e(-\a z+[\a]) e(-nz+\a\eta) e( \tfrac{1}{n} + \cdots +  \tfrac{n-1}{n})
\prod_{m=0}^{n-1}    \theta\!\left(z -\tfrac{m}{n}-  \tfrac{\a}{n}\eta\right).
\end{align*}
The expression before the product symbol in the last formula is
\begin{align*}
p(z)
& \; = \; 
 (-1)^n e(-\a z+[\a]) e(-nz+\a\eta) e( \tfrac{1}{2}(n-1))
\\
& \; = \; 
 (-1)^n e(-\a z+[\a]) e(-nz+\a\eta)  (-1)^{n-1}
\\
& \; = \; 
-  e(-nz+\a\eta +[\a]) e(-\a z).
\end{align*}
Since
\begin{align*}
\prod_{m=0}^{n-1}    \theta\!\left(z -\tfrac{m}{n}-  \tfrac{\a}{n}\eta\right)
& \; = \;
\prod_{m=0}^{n-1}    \theta\!\left(z + \tfrac{n-m}{n}-  \tfrac{\a}{n}\eta\right)
\\
& \; = \;
\theta\!\left(z + \tfrac{n}{n}-\tfrac{\a}{n}\eta\right) 
\theta\!\left(z + \tfrac{n-1}{n}-\tfrac{\a}{n}\eta\right)
\cdots
\theta\!\left(z + \tfrac{1}{n}-\tfrac{\a}{n}\eta\right)
\\
& \; = \; 
e(\a z-[-\a]) \theta_{-\a}(z),
\\
\theta_\a(-z)
& \; = \;
p(z)\prod_{m=0}^{n-1}\theta\!\left(z-\tfrac{m}{n}-\tfrac{\a}{n}\eta\right)
\\
& \; = \;
-e(-nz+\a\eta+[\a])e(-\a z)e(\a z-[-\a]) \theta_{-\a}(z)
\\
& \; = \;
-e(-nz+\a\eta+[\a]-[-\a])\theta_{-\a}(z),
\end{align*}
as claimed.

\cref{item.defn.R.theta.fn.wd}
Since $\theta(z+\eta)=-e(-z)\theta(z)$, 
\begin{align*}
\theta_{\a+n}(z) & \; = \;  e((\a+n) z+[\a+n]) 
\prod_{m=0}^{n-1}  \theta\!\left(z +\tfrac{m}{n}+ \tfrac{\a+n}{n}\eta\right)
\\
& \; = \; 
 e(nz+[\a+n]-[\a]) e(\a z+[\a]) 
\prod_{m=0}^{n-1} (-1) e\big(-z -\tfrac{m}{n}- \tfrac{\a}{n}\eta\big)   \theta\!\left(z +\tfrac{m}{n}+ \tfrac{\a}{n}\eta\right)
\\
& \; = \; 
(-1)^n e(nz+[\a+n]-[\a]) e(\a z+[\a])   
 e\big(-nz -\tfrac{1}{n}-\tfrac{2}{n} \cdots -\tfrac{n-1}{n} - \a\eta\big) 
\prod_{m=0}^{n-1}   \theta\!\left(z +\tfrac{m}{n}+ \tfrac{\a}{n}\eta\right)
\\
& \; = \; 
(-1)^n e([\a+n]-[\a]) (-1)^{n-1}    e(- \a\eta)    \theta_\a(z) 
\\
& \; = \; 
- e([\a+n]-[\a]  - \a\eta)    \theta_\a(z), 
\end{align*}
as claimed.
\end{pf}

\begin{lemma}
\label{prop.basis.1}
The set $\{\theta_{0},\ldots,\theta_{n-1}\}$ in \cref{defn.R.theta.fn} is a basis for $\Theta_{n}(\L)$.  
\end{lemma}
\begin{proof}
Since $\theta_{\a}$ is an eigenvector with eigenvalue $e\big(\frac{\a}{n}\big)$ for the linear transformation $f(z)\mapsto f(z+\frac{1}{n})$, 
the functions $\theta_{0},\ldots,\theta_{n-1}$ are linearly independent. 
But  $\dim \Theta_{n}(\L)=n$, so they form a basis for it.
\end{proof}

In \cref{sect.general.theta-alpha} we consider how to choose $[\a]$ and hence $\theta_\a$. 
We then devote a single subsection to the definition of the functions $\theta_\a$ in each of the following papers
of Feigin and Odesskii: the Kiev preprint \cite{FO-Kiev};  their first published paper \cite{FO89}; Odesskii's survey \cite{Od-survey}. 
Finally, in \cref{official.defn.theta-alpha}, we fix particular $[\a]$'s and define the $\theta_\a$'s that will be used in the rest of this paper and in our subsequent 
papers.

We advise the reader to jump to \cref{official.defn.theta-alpha} on a first reading. 

\subsubsection{}
\label{sect.general.theta-alpha}
We now consider the choice of $[\a]$. 
First, we want the coefficient $e(\frac{\a}{n}\eta+[\a]-[\a+1])$ in \cref{defn.R.theta.prop}\cref{defn.R.theta.fn.per.sec}  
to be a constant $C$ independent of $\alpha$. 
Second, we want equalities $\theta_{\alpha+n}=\theta_{\alpha}$ for all $\a\in \ZZ$.
Third, since adding a constant to $[\a]$ corresponds to multiplying all the $\theta_{\alpha}$'s by a common scalar, 
we normalize the function $\a \mapsto [\a]$ by requiring $[0]=0$. 
In summary, we will choose the $[\a]$'s so the following three conditions hold:
\begin{equation*}
	\begin{cases}
		\, e(\frac{\a}{n}\eta+[\a]-[\a+1]) \;=\; C,\\
		\, -e([\a+n]-[\a]-\a\eta) \; = \; 1,\\
		\, [0]\; = \; 0.
	\end{cases}	
\end{equation*}
Taken together, the first and the third of these conditions imply that 
\begin{align*}
	C^{\a} &\;=\;  \prod_{i=0}^{\a-1} e\left(\tfrac{i}{n}\eta+[i]-[i+1]\right) \;=\; 	e\left(\tfrac{\a(\a-1)}{2n}\eta-[\a]\right)\qquad\text{and}\\
	C^{\a} &\;=\;  \prod_{i=-\alpha}^{-1} e\left(\tfrac{i}{n}\eta+[i]-[i+1]\right) \;=\; 	e\left(-\tfrac{(-\a)(-\a-1)}{2n}\eta+[-\a]\right)
\end{align*}
for all integers $\a>0$.
Hence
\begin{equation*}
	e([\a]) \;=\; C^{-\a}e\left(\tfrac{\a(\a-1)}{2n}\eta\right)
\end{equation*}
for all $\a \in \ZZ$.
Substituting this into the second condition, implies that
\begin{equation*}
	1\;=\; -\, e([\a+n]-[\a]-\a\eta)\;=\; -\, C^{-n}e\left(\tfrac{(\a+n)(\a+n-1)}{2n}\eta-\tfrac{\a(\a-1)}{2n}\eta-\a\eta\right) \;=\; 
	C^{-n}e\left(-\tfrac{1}{2}+\tfrac{n-1}{2}\eta\right).
\end{equation*}
Therefore $C=e\left(\frac{r}{n}-\frac{1}{2n}+\frac{n-1}{2n}\eta\right)$ for some integer $r$. It follows that
\begin{equation*}
	e([\alpha]) \;=\; e\left(\tfrac{\a(1-2r)}{2n}+\tfrac{\a(\a-n)}{2n}\eta\right).
\end{equation*}
Parts  \cref{defn.R.theta.fn.per.sec} and \cref{defn.R.theta.fn.neg} of \cref{defn.R.theta.prop} now become
\begin{align*}
	\theta_\a(z+\tfrac{1}{n}\eta) &  \; = \;  e\left(-z+\tfrac{2r-1}{2n}+\tfrac{n-1}{2n}\eta\right)\theta_{\a+1}(z),\\ 
	\theta_\a(-z) & \; = \;  -  e\left(-nz+\tfrac{\a(1-2r)}{n}\right)\theta_{-\a}(z).
\end{align*}
The next result summarizes these discussions.\footnote{We note that $e([\a])$ depends only on the image of $r$ in $\ZZ_n$.}

\begin{lemma}\label{std.basis.theta.func}
	Let $r\in\bbZ$ be any integer. The functions
\begin{equation}
\label{defn.R.theta.fn.cond}
\theta_\a(z) \; :=\; e\left(\a z +\tfrac{\a(1-2r)}{2n}+\tfrac{\a(\a-n)}{2n}\eta\right)  \prod_{m=0}^{n-1}  \theta\!\left(z +\tfrac{m}{n}+ \tfrac{\a}{n}\eta\right),
\end{equation}
indexed by $\alpha\in\bbZ$, have the following properties:
	\begin{enumerate}
	\item
	$\theta_{\a+n}=\theta_{\a}$,
	\item
	$\{\theta_{0},\ldots,\theta_{n-1}\}$ is a basis for $\Theta_{n}(\L)$,
	\item 
	$\theta_\a(z+\tfrac{1}{n})  \; = \; e\big(\tfrac{\a}{n}\big) \theta_\a(z)$, 
	\item\label{defn.R.theta.fn.cond.per.sec}
	$\theta_\a(z+\tfrac{1}{n}\eta)  \; = \;  e\big(-z+\tfrac{2r-1}{2n}+\tfrac{n-1}{2n}\eta\big)\theta_{\a+1}(z)$, and
	\item\label{defn.R.theta.fn.cond.neg}
	$\theta_\a(-z)  \; = \;  -  e\big(-nz+\tfrac{\a(1-2r)}{n}\big)\theta_{-\a}(z)$.
	\end{enumerate}
\end{lemma}

The key point in each of the next three subsections  is how to choose the integer $r$ (modulo $n$) so the functions $\theta_\a$
have the properties that Feigin and Odesskii ask of them.

\subsubsection{}\label{subsubsec.Kiev.charact.theta.basis}
	The appendix of the Kiev preprint \cite{FO-Kiev} says that when $n$ is odd 
	$\Theta_n(\L)$ has a basis $\{\theta_\a \; | \; \a \in \ZZ_n\}$ 
	such that
\begin{enumerate}
  \item\label{kiev.theta.prop.first.per} 
  $\theta_\a(z+\tfrac{1}{n}) \;=\; e\big(\frac{\a}{n}\big) \theta_\a(z)$,
  \item\label{kiev.theta.prop.sec.per} 
  $\theta_\a(z+\tfrac{1}{n}\eta) \;=\; - e\big(-z  + \frac{n-1}{2n}\eta \big) \theta_{\a+1}(z)$, 
  \item\label{kiev.theta.wrong.prop} 
  $\theta_\a(-z)= e(- nz) \theta_{-\a}(z)$, and
  \item\label{kiev.theta.prop.zeros} 
$\theta_\a(z)$ is zero exactly at the points in  $-\frac{\a}{n}\eta +\frac{1}{n}\ZZ + \ZZ \eta$.
\end{enumerate}
This is false. 
There is no integer $r$ such that the functions $\theta_{\alpha}$ defined by \cref{defn.R.theta.fn.cond} have these four properties:  if there were, then \cref{kiev.theta.wrong.prop} together with \cref{std.basis.theta.func}\cref{defn.R.theta.fn.cond.neg}
would imply that $\frac{\a(1-2r)}{n}+\frac{1}{2}$ is an integer, which is not the case when $\alpha=0$, for example.

If \cref{kiev.theta.prop.sec.per} held, then \cref{std.basis.theta.func}\cref{defn.R.theta.fn.cond.per.sec} would imply that 
the number $s:=\tfrac{2r-1}{2n}-\tfrac{1}{2}$ is an integer so $r=\frac{n+1}{2}+ns=\frac{n+1}{2}$ ($\mathrm{mod}$ $n$) which implies that $n$ is odd.
If $n$ is odd and $r=\frac{n+1}{2}$, then the functions $\theta_{\alpha}$ in \cref{defn.R.theta.fn.cond} 
satisfy \cref{kiev.theta.prop.first.per}, \cref{kiev.theta.prop.sec.per}, \cref{kiev.theta.prop.zeros},  and $\theta_\a(-z)  \; = \;  -  e(-nz)\theta_{-\a}(z)$;
in \cref{subsec.alt.basis} we denote these functions by $\psi_\a$ (only when $n$ is odd).
It is likely that the $\theta_\a$'s in the Kiev preprint are the $\psi_\a$'s and the statement that $\theta_a(-z)=e(-nz)\theta_{-\a}(z)$
in its appendix is a typo.  
In \cref{ssect.Kiev.relns.n.odd}, we make some additional comments about the $\theta_\a$'s in the Kiev preprint.

\subsubsection{}\label{subsubsec.FO.charact.theta.basis}
	Let $2^p$ be the largest power of $2$ dividing $n$. The paper \cite[\S1.1]{FO89} says that 
	$\Theta_n(\L)$ has a basis $\{\theta_\a 	\; | \; \a \in \ZZ_n\}$  	such that
\begin{enumerate}
  \item\label{of.theta.prop.first.per}
  $\theta_\a(z+\tfrac{1}{n}) \;=\; e\big(\frac{\a}{n}\big) \theta_\a(z)$,
  \item\label{of.theta.prop.sec.per}  
  $\theta_\a(z+\frac{1}{n}\eta) \;=\;  e\big(-z -2^{-p-1} + \frac{n-1}{2n}\eta \big) \theta_{\a+1}(z)$, 
  \item\label{of.theta.prop.neg1}  
  $\theta_\a(-z)= - e(- nz) \theta_{-\a}(z)$ if $n$ is odd, 
    \item\label{of.theta.prop.neg2}  
  $\theta_\a(-z)= - e(- nz+2^{-p}\a) \theta_{-\a}(z)$ if $n$ is even, and
  \item\label{of.theta.prop.zeros}
$\theta_\a(z)$ is zero exactly at the points in $-\frac{\a}{n}\eta +\frac{1}{n}\ZZ + \ZZ \eta$.
\end{enumerate}
	If the functions $\theta_{\alpha}$ defined by \cref{defn.R.theta.fn.cond} satisfy these five properties, then \cref{of.theta.prop.sec.per} implies that the number $s:=\frac{2r-1}{2n}-(-2^{-p-1})$ is an integer and $r=\frac{1}{2}(-\frac{n}{2^{p}}+1)+ns=\frac{1}{2}(-\frac{n}{2^{p}}+1)$.
	
	Conversely, set $r=\frac{1}{2}(-\frac{n}{2^{p}}+1)$, which is always an integer.\footnote{If we write $n=2^{p}(2l+1)$ as in \cite{FO89}, then $r=-l$ modulo $n$.} Since $\frac{1-2r}{n}=2^{-p}$, the function $\theta_\a$ in  \cref{defn.R.theta.fn.cond} now has the property that
	\begin{equation}\label{eq.theta.minus.FO}
		\theta_{\a}(-z) \; =\; -\,e(-nz+2^{-p}\alpha)\theta_{-\a}(z).
	\end{equation}
Hence conditions \cref{of.theta.prop.neg1} and \cref{of.theta.prop.neg2} are satisfied, and so are \cref{of.theta.prop.first.per}, \cref{of.theta.prop.sec.per}, and \cref{of.theta.prop.zeros}. We also note that 
$$
e([\a]) \;=\; e\left(2^{-p-1}\a+\tfrac{\a(\a-n)}{2n}\eta\right)
$$
in this case.

\subsubsection{}
\label{sssect.od.survey} 
	Odesskii's survey \cite[Appendix~A]{Od-survey} says  $\Theta_{n}(\L)$ has a basis $\{\theta_\a \; | \; \a \in \ZZ_n\}$ such that
	\begin{enumerate}
  \item\label{od.surv.theta.prop.first.per} 
  $\theta_\a(z+\tfrac{1}{n}) \;=\; e\big(\frac{\a}{n}\big) \theta_\a(z)$,
  \item\label{od.surv.theta.prop.sec.per} 
  $\theta_\a(z+\tfrac{1}{n}\eta) \;=\; e\big(-z -\frac{1}{2n} + \frac{n-1}{2n}\eta \big) \theta_{\a+1}(z)$, and
  \item\label{od.surv.theta.prop.prod} 
  $\theta_\a(z) \; =\; e(\a z +\tfrac{\a}{2n}+\tfrac{\a(\a-n)}{2n}\eta)  \prod_{m=0}^{n-1}  \theta\!\left(z +\frac{m}{n}+ \frac{\a}{n}\eta\right)$.
\end{enumerate}
If the functions $\theta_{\alpha}$ in \cref{defn.R.theta.fn.cond} satisfy \cref{od.surv.theta.prop.sec.per}, then $r$ is divisible by $n$. If $r$ {\it is} divisible by $n$, then the functions $\theta_{\alpha}$ in \cref{defn.R.theta.fn.cond} have properties \cref{od.surv.theta.prop.first.per}, \cref{od.surv.theta.prop.sec.per},
and \cref{od.surv.theta.prop.prod}.

\subsubsection{The ``standard'' definition of $\theta_\a$}
\label{official.defn.theta-alpha}
From now on, unless otherwise stated, $\theta_{\alpha}$ denotes the function in \cref{defn.R.theta.fn.cond} with $r=0$ modulo $n$.\footnote{The function in \cref{defn.R.theta.fn.cond} only depends on $r$ modulo $n$.} We repeat this definition in \cref{official.theta_alpha}
below.  
As remarked in \cref{sssect.od.survey}, the function $\theta_\a$ in  \cref{official.theta_alpha} is the same as the function 
$\theta_\a$ defined in Odesskii's survey \cite[Appendix~A]{Od-survey}. 

\begin{proposition}\label{prop.official.theta.basis}
  The functions
  \begin{equation}
    \label{official.theta_alpha}
    \theta_\a(z) \; :=\; e\left(\a z +\tfrac{\a}{2n}+\tfrac{\a(\a-n)}{2n}\eta\right)  \prod_{m=0}^{n-1}  \theta\!\left(z +\tfrac{m}{n}+ \tfrac{\a}{n}\eta\right),
  \end{equation}
  indexed by $\alpha\in\bbZ$, have the following properties:
  \begin{enumerate}
  \item
    $\theta_{\a+n}=\theta_{\a}$.
  \item
    $\{\theta_{0},\ldots,\theta_{n-1}\}$ is a basis for $\Theta_{n}(\L)$.
  \item 
    $\theta_\a(z+\tfrac{1}{n})  \; = \; e\big(\tfrac{\a}{n}\big) \theta_\a(z)$.
  \item\label{item.official.theta.basis.act.eta}
    $\theta_\a(z+\tfrac{1}{n}\eta)  \; = \;  e\big(-z-\tfrac{1}{2n}+\tfrac{n-1}{2n}\eta\big)\theta_{\a+1}(z)$.
  \item\label{item.official.theta.basis.minus}
    $\theta_\a(-z)\;=\;-e\big(-nz+\tfrac{\a}{n}\big)\theta_{-\a}(z)$.
  \item\label{item.official.theta.basis.zeros}
    The zeros of $\theta_\a$ are the points in $-\tfrac{\a}{n}\eta +\tfrac{1}{n}\ZZ + \ZZ\eta$
    and all of them have multiplicity one.
  \item
    \label{eq.torsion.translate}
    For all $r\in\ZZ$,
    $		\theta_{\a}(z+\tfrac{r}{n}\eta)     
    \,=\,                e\big(-rz-\tfrac{r}{2n}+\tfrac{rn-r^2}{2n}\eta\big)   \theta_{\a+r}(z)$.
  \end{enumerate}
\end{proposition}
\begin{proof}
All of this, with the exception of part \cref{eq.torsion.translate} has been proved before. The formula in \cref{eq.torsion.translate} is first proved by induction for all $r \ge  0$, then, 
by replacing $\a$ by $\a-r$ and $z$ by $z-\frac{r}{n}\eta$ in the formula, one sees that it holds for all $r\in \ZZ$.
\end{proof}

The basis $\theta_{0}$ for $\Theta_{0}(\Lambda)$ is the function $\theta$ defined in \cref{eq.basic.theta.fn}.

\subsubsection{}
A basis for $\Theta_{n,c}(\Lambda)$ can be constructed from the basis $\theta_\a$ for $\Theta_n(\L)=\Theta_{n,\frac{n-1}{2}}(\L)$.

\begin{proposition}\label{theta.basis.c}
	For $\a \in \ZZ$, let $\theta_\alpha$ be the function defined in  \cref{official.theta_alpha}. The functions
	\begin{equation*}
		\theta_{\a,c}(z) \; := \; \theta_{\a}(z-\tfrac{1}{n}c+\tfrac{n-1}{2n})
	\end{equation*}
	have the following properties:
	\begin{enumerate}
		\item\label{item.theta.basis.c.n} $\theta_{\alpha+n,c}=\theta_{\alpha,c}$.
		\item\label{item.theta.basis.c.basis} $\{\theta_{0,c},\ldots,\theta_{n-1,c}\}$ is a basis of $\Theta_{n,c}(\L)$.
		\item\label{item.theta.basis.c.S} $\theta_{\a,c}(z+\frac{1}{n})=e(\frac{\a}{n})\theta_{\a,c}(z)$.
		\item\label{item.theta.basis.c.T} $\theta_{\a,c}(z+\frac{1}{n}\eta)=-e(-z+\frac{1}{n}c+\frac{n-1}{2n}\eta)\theta_{\a+1,c}(z)$.
		\item\label{item.theta.basis.c.unad} $\theta_{\alpha,\frac{n-1}{2}}(z)=\theta_{\alpha}(z)$.
	\end{enumerate}
\end{proposition}

\begin{proof}
	It is clear that \cref{item.theta.basis.c.unad} holds.

	The properties \cref{item.theta.basis.c.n}, $\theta_{\a,c}(z+1)=\theta_{\a,c}(z)$, and $\theta_{\a,c}(z+\frac{1}{n})=e(\frac{\a}{n})\theta_{\a,c}(z)$ follow from 
	the same properties of $\theta_{\a}$. 
	Let $d:=\frac{1}{n}c-\frac{n-1}{2n}$. Then
	\begin{align*}
		\theta_{\a,c}(z+\eta)
		&\;=\; \theta_{\a}(z+\eta-d)\\
		&\;=\; -e(-nz+nd)\theta_{\a}(z-d)\\
		&\;=\; -e(-nz+c-\tfrac{n-1}{2})\theta_{\a,c}(z)\\
		&\;=\; (-1)^{n}e(-nz+c)\theta_{\a,c}(z).
	\end{align*}
	Hence $\theta_{\a,c}\in\Theta_{n,c}(\L)$. Since the $\theta_{0,c}, \ldots, \theta_{n-1,c}$ are eigenvectors for the linear operator 
	$f(z)\mapsto f(z+\frac{1}{n})$ with different eigenvalues and the dimension of $\Theta_{n,c}(\L)$ is $n$, they are a basis for 
	$\Theta_{n,c}(\L)$.
	
Statement \cref{item.theta.basis.c.T} holds because 
	\begin{align*}
		\theta_{\a,c}(z+\tfrac{1}{n}\eta)
		&\;=\; \theta_{\a}(z+\tfrac{1}{n}\eta-d)\\
		&\;=\; e(-z+d-\tfrac{1}{2n}+\tfrac{n-1}{2n}\eta)\theta_{\a+1}(z-d)\\
		&\;=\; e(-z+\tfrac{1}{n}c-\tfrac{1}{2}+\tfrac{n-1}{2n}\eta)\theta_{\a+1,c}(z)\\
		&\;=\; -e(-z+\tfrac{1}{n}c+\tfrac{n-1}{2n}\eta)\theta_{\a+1,c}(z).
	\end{align*}
The proof is now complete.	
\end{proof}

In \cite[Appendix~A]{Od-survey}, Odesskii considered another basis $\{\theta_{\alpha}(z-\frac{1}{n}c-\frac{n-1}{2n})\;|\;\a\in\ZZ_{n}\}$ for $\Theta_{n,c}(\Lambda)$. It {\it is} a basis because
\begin{align*}
	\theta_{\alpha}(z-\tfrac{1}{n}c-\tfrac{n-1}{2n})
	&\;=\;
	\theta_{\alpha}(z-\tfrac{1}{n}c+\tfrac{n-1}{2n}-1+\tfrac{1}{n})\\
	&\;=\;
	e(\tfrac{\alpha}{n})\theta_{\alpha}(z-\tfrac{1}{n}c+\tfrac{n-1}{2n})\\
	&\;=\;
	e(\tfrac{\alpha}{n})\theta_{\a,c}(z).
\end{align*}

\subsection{$\Theta_n(\L)$ as a representation of the Heisenberg group}

Fix $d \in \CC$. Let $S$ and $T$ be the operators on the space of meromorphic functions on $\CC$ defined by
\begin{align*}
(S\cdot f)(z) & \;=\; f\big(z+\tfrac{1}{n}\big),
\\
(T\cdot f)(z) & \;=\; e(z+d)f\big(z+\tfrac{1}{n}\eta\big).
\end{align*}
Both $S$ and $T$ are invertible and satisfy $ST=e(\frac{1}{n})TS$.

It is clear that $\Theta_n(\L)$ is stable under the action of $S$ and $T$ and that $S^n$ acts as the identity on $\Theta_n(\L)$. 
When $d=\frac{1}{2n}- \frac{n-1}{2n}\eta$ the 
operator $T^n$ also acts as the identity on $\Theta_n(\L)$ because
\begin{align*}
(T^n\cdot f)(z)  & \;=\; e(z+d)(T^{n-1}\cdot f)\big(z+\tfrac{1}{n}\eta\big)
\\
& \;=\; e(z+d)e(z+\tfrac{1}{n}\eta+d) (T^{n-2}\cdot f)\big(z+\tfrac{2}{n}\eta\big)
\\
& \;=\; \cdots
\\
& \;=\; e(z+d)e(z+\tfrac{1}{n}\eta+d) \cdots  e(z+\tfrac{n-1}{n}\eta+d)  f\big(z+\tfrac{n}{n}\eta\big)
\\
& \;=\; e(nz+nd +\tfrac{n-1}{2}\eta)  f\big(z+\eta\big)
\\
& \;=\; - \, e(nd +\tfrac{n-1}{2}\eta)  f(z)
\\
& \;=\;   f(z).
\end{align*}

This leads to a representation of the {\sf Heisenberg group of order $n^3$} on $\Theta_n(\L)$.  
This group is
\begin{equation}
\label{defn.Hn}
H_n \; : =\; \langle S,T, \epsilon \; | \; S^n=T^n=\epsilon^n=1,  \, \epsilon =[S,T], \, [S,\epsilon]=[T,\epsilon]=1\rangle.
\end{equation}

\begin{lemma}
\label{lem.Hn.1}
The space $\Theta_n(\L)$ is an irreducible representation of $H_n$ via the actions
\begin{align*}
(S\cdot f)(z) & \;=\; f\big(z+\tfrac{1}{n}\big),
\\
(T\cdot f)(z) & \;=\; e\big(z+\tfrac{1}{2n}-\tfrac{n-1}{2n}\eta \big)f\big(z+\tfrac{1}{n}\eta\big).
\end{align*}
The action on the $\theta_\a$'s in \cref{official.theta_alpha} is 
\begin{align*}
S\cdot \theta_\a& \;=\;   e\big( \tfrac{\a}{n} \big) \theta_{\a},
\\
T\cdot \theta_\a& \;=\; \theta_{\a+1}.
\end{align*}
\end{lemma}
\begin{proof}
The action of $S$ and $T$ on the $\theta_\a$'s is as claimed because $\theta_{\a}\big(z+\frac{1}{n}\big)=e\big(\frac{\a}{n}\big)\theta_{\a}(z)$
and 
\begin{align*} 
(T\cdot \theta_\a)(z) & \;=\; e\big(z+\tfrac{1}{2n}-\tfrac{n-1}{2n}\eta \big)\theta_\a \big(z+\tfrac{1}{n}\eta\big)
\\
& \;=\; e\big(z+\tfrac{1}{2n}-\tfrac{n-1}{2n}\eta \big)e\big(-z-\tfrac{1}{2n}  + \tfrac{n-1}{2n}\eta\big)     \theta_{\a+1} (z) 
\\
& \;=\;  \theta_{\a+1} (z).
\end{align*}
Because the $\theta_\a$'s are $S$-eigenvectors with different eigenvalues, every
subspace of $\Theta_n(\L)$ that is stable under the action of $S$ is spanned by some of the $\theta_\a$'s. 
Since  $T \cdot \theta_\a = \theta_{\a+1}$ the only non-zero subrepresentation of $\Theta_n(\L)$ 
is $\Theta_n(\L)$ itself. Hence $\Theta_n(\L)$ is an irreducible representation of $H_n$. 
\end{proof}

\subsection{Embedding $E$ in $\PP^{n-1}$ via $\Theta_n(\L)$}
\label{ssec.emb.E}

Evaluation at a point $z\in \CC$ provides a surjective linear map $\Theta_n(\L) \to \CC$. 
The kernel of this evaluation map depends only on the coset $z+\L$ so there is a well-defined map from $\CC/\L$ to the set of codimension-one subspaces of $\Theta_n(\L)$ or, what is essentially the same thing, a holomorphic map
\begin{equation}
\label{map.E.to.P.n-1}
\iota\colon E \; \longrightarrow \; \PP(\Theta_n(\L)^*)
\end{equation}
to the projective space of 1-dimensional subspaces of $\Theta_n(\L)^*$. Since $E$ and $\PP(\Theta_n(\L)^*)$ are smooth projective varieties, $\iota$ is a morphism of algebraic varieties \cite[p.~170]{GH78}.

Since the $\theta_\a$'s are a basis for $\Theta_n(\L)$ they form
a system of homogeneous  coordinate functions on $\PP(\Theta_n(\L)^*)$. With respect to this system of homogeneous coordinates the 
map in \cref{map.E.to.P.n-1} is 
$$
z \; \mapsto \; (\theta_0(z),\ldots,\theta_{n-1}(z)).
$$

Suppose $n\geq 3$. Since the pullback $\iota^{*}\cO(1)$ of the twisting sheaf $\cO(1)$ on $\PP(\Theta_n(\L)^*)$ has degree $n$, \cite[Cor.~IV.3.2]{Hart} implies that $\iota^{*}\cO(1)$ is very ample. Hence $\iota$ is a closed immersion. We will often identify $E$ with its image under $\iota$.
Each linear form on $\PP(\Theta_n(\L)^*)$ vanishes at exactly $n$ points of $E$ counted with multiplicity and the sum of those points is
the image of $\frac{n-1}{2}$ in $E$. Conversely, if $p_1,\ldots,p_n$ are points on $E$ whose sum is the image of $\frac{n-1}{2}$ 
there is a function $f \in \Theta_n(\L)$, unique up to non-zero scalar multiples, that vanishes exactly at $p_{1},\ldots,p_{n}$ modulo $\Lambda$, counted with multiplicity.

Since $\Theta_n(\L)$ is a representation of $H_n$, its dual  $\Theta_n(\L)^*$ becomes a representation of $H_n$ with respect
to the contragredient action $(g\cdot \varphi)(f) = \varphi(g^{-1}\cdot f)$ for $g \in H_n$, $\varphi \in \Theta_n(\L)^*$, and $f \in \Theta_n(\L)$. 
Thus $H_n$ acts as linear automorphisms of $\PP\big(\Theta_n(\L)^*\big)$. For example, if $z \in E$, then 
$$
S\cdot (\theta_0(z),\ldots, \theta_{n-1}(z)) \;=\; \big(\theta_0\big(z-\tfrac{1}{n}\big),\ldots, \theta_{n-1}\big(z-\tfrac{1}{n}\big)\big). 
$$
 Since the commutator $[S,T]$ acts on $\Theta_n(\L)$ as multiplication by $e\big(\frac{1}{n}\big)$, it acts trivially on $\PP\big(\Theta_n(\L)^*\big)$. Thus, the action of $H_n$
factors through the quotient of $H_n$ by the subgroup generated by $[S,T]$. This quotient is isomorphic to $\ZZ_n \times \ZZ_n$.  

\subsection{Another basis for $\Theta_n(\L)$ when $n$ is odd}
\label{subsec.alt.basis}

As we explained in \cref{subsubsec.Kiev.charact.theta.basis}, the characterization of the basis for 
$\Theta_n(\L)$ in the Kiev preprint \cite{FO-Kiev} is not compatible with \cref{defn.R.theta.fn.cond}
and, even after removing condition \cref{kiev.theta.wrong.prop} in  \cref{subsubsec.Kiev.charact.theta.basis}, 
it is only compatible when $n$ is odd, and in that case, the integer  $r$ (modulo $n$), and hence the definition of the basis, coincides with that of \cite{FO89} described in \cref{subsubsec.FO.charact.theta.basis}.

We denote that basis by $\psi_{0}, \ldots, \psi_{n-1}$. Explicitly, we assume that $n$ is odd, and the $\psi_{\alpha}$'s are the functions
in \cref{defn.R.theta.fn.cond} with $r=\frac{n+1}{2}$ (modulo $n$); i.e., 
\begin{equation*}
	\psi_{\a}(z)\;=\; e\left(-\tfrac{\a(n+1)}{2n}\right)\theta_{\a}(z).
\end{equation*}
The bases $\{\theta_\a\}$ and  $\{\psi_\a\}$ coincide if and only if $n=1$.

The transformation properties of the $\{\psi_\a\}$'s are given by \cref{std.basis.theta.func}.

For some purposes the $\psi_\alpha$'s are a ``better''  basis than the $\theta_\alpha$'s. Define $\nu \in \Aut(\bbP^{n-1})$ by
\begin{equation*}
	\nu(x_{0},x_{1},\ldots,x_{n-1}) \; : = \; (x_{0},x_{n-1},\ldots,x_{1})
\end{equation*}
as in \cite[Lem.~3.5]{Fisher}. By property \cref{of.theta.prop.neg1} in \cref{subsubsec.FO.charact.theta.basis}, 
$
\psi_\alpha(-z) = -\,e(-nz)\psi_{-\alpha}(z).
$
The closed immersion $\psi\colon E\to\bbP^{n-1}$ given by $\psi(z)=(\psi_0(z),\ldots,\psi_{n-1}(z))$ therefore fits into the commutative diagram
\begin{equation*}
	\begin{tikzcd}
		E\ar[r,"\psi"]\ar[d,"{[-1]}"'] & \bbP^{n-1}\ar[d,"\nu"] \\
		E\ar[r,"\psi"'] & \bbP^{n-1}
	\end{tikzcd}
\end{equation*}
where $[-1]\colon E\to E$ is the automorphism that sends $z$ to $-z$; i.e., if $\psi(z)=(x_{0},x_{1},\ldots,x_{n-1})$, then 
$\psi(-z)=(x_{0},x_{n-1},\ldots,x_{1})$

The only other places in this paper where the functions $\psi_\a$ appear are  \cref{subsubsec.Kiev.charact.theta.basis,ssect.Kiev.relns.n.odd}.

\section{Definitions and basic properties of $Q_{n,k}(E,\tau)$}
\label{sect.Skl.defn.alg}

From now on, $n>k \ge 1$ are relatively prime integers.

For the remainder of this paper the $\theta_\a$'s are the functions defined in \cref{official.theta_alpha}.

\subsection{The definition of $Q_{n,k}(E,\tau)$ and $\rel_{n,k}(E,\tau)$ when $\tau \notin \frac{1}{n}\L$}
\label{subsec.def.FO.alg}
\leavevmode

Fix $\tau\in\bbC-\frac{1}{n}\Lambda$, and let $V$ be a $\bbC$-vector space with basis $\{x_{i}\;|\;i\in\ZZ_{n}\}$.

\begin{definition}\label{}
$Q_{n,k}(E,\tau)$ is the quotient of the free algebra 
$TV=\CC \langle x_0,\ldots,x_{n-1}\rangle$ by the $n^2$ relations
\begin{equation}\label{the-relns}
r_{ij} \;=\; r_{ij}(\tau)  \; :=\; \sum_{r \in \ZZ_n} \frac{\theta_{j-i+(k-1)r}(0)}{\theta_{j-i-r}(-\tau)\theta_{kr}(\tau)} \,\,  x_{j-r}x_{i+r}, \qquad i,j \in \ZZ_n.	
\end{equation}
The space of quadratic relations is denoted 
\begin{equation*}
\rel_{n,k}(E,\tau) \; :=\; \operatorname{span} \{r_{ij}(\tau) \; | \; i,j \in \ZZ_n\} \;\subseteq\; V \otimes V.
\end{equation*}
\end{definition}

For $\tau\in\frac{1}{n}\Lambda$, $\rel_{n,k}(E,\tau)$ and $Q_{n,k}(E,\tau)$ will be defined in \cref{def.Qnk}.

\subsubsection{}\label{sssec.Rij.def.k}
Although our definition of $Q_{n,k}(E,\tau)$ differs from that in \cite{FO89, OF93, OF95, FO98, OF92}, 
our $Q_{n,k}(E,\tau) =\text{their }Q_{n,k}(E,\tau)$. 

In \cite{FO89, OF93, OF95, FO98, OF92} the term $x_{j-r}x_{i+r}$ in 
\cref{the-relns} is replaced by $x_{k(j-r)}x_{k(i+r)}$. This is just a change of variables: our $x_{i}$ is their $x_{ki}$. 
\cref{prop.isom.k.k'} below shows there is an isomorphism $Q_{n,k}(E,\tau) \to Q_{n,k'}(E,\tau)$ given by $x_i \mapsto x_{ki}$. 
Thus, the algebra we call $Q_{n,k}(E,\tau)$ with ordered basis $x_0,\ldots,x_{n-1}$ is the same as the algebra 
$Q_{n,k'}(E,\tau)$ with ordered basis $x_0,\ldots,x_{n-1}$ in  {\it loc.\ cit.} 

\subsubsection{}
When $k=1$, our relation $r_{ij}=0$ is identical to that at  \cite[(18), p.~1143]{Od-survey}; 
that definition of $Q_{n,1}(E,\tau)$ is used in Odesskii's subsequent papers \cite{OR08, ORTP11a, ORTP11b}. 

\subsubsection{}
\label{ssect.k=1.relns}
Suppose $k=1$. Then $r_{ii}=0$ for all $i$ because $\theta_0(0)=0$. Thus, 
whenever we speak of $r_{ij}$ when $k=1$ we will assume that $i\ne j$. (When $k \ne 1$,
$r_{ij}(\tau)$ is non-zero for all $i,j$ and all $\tau \in \CC-\frac{1}{n}\Lambda$.) 
When $i \ne j$ all the structure constants in $r_{ij}$ have the same numerator so $r_{ij}$ can be replaced by the relation
\begin{equation}
\label{the-relns.2}
\sum_{r \in \ZZ_n} \frac{x_{j-r}x_{i+r} }{\theta_{j-i-r}(-\tau)\theta_{r}(\tau)}  \; = \; 0.
\end{equation}

\subsubsection{Relations for $Q_{n,1}(E,\tau)$ when $n$ is odd}
\label{ssect.Kiev.relns.n.odd}
 In the Kiev preprint \cite[\S3]{FO-Kiev}, $Q_{n,1}(E,\tau)$ is defined for odd $n \ge 3$ as the free algebra
 $\CC\langle x_0,\ldots,x_{n-1}\rangle$ modulo the $n(n-1)$ relations
\begin{equation}
\label{FO.relns}
\frac{x_i^2}{\theta_j(\tau)\theta_{-j}(\tau)}  \, + \,   \frac{x_{i-1}x_{i+1}}{\theta_{1+j}(\tau)\theta_{1-j}(\tau)}  \, + \,  \cdots \, + \,  
\frac{x_{i-(n-1)}x_{i+n-1}}{\theta_{n-1+j}(\tau)\theta_{n-1-j}(\tau)}  \;=\; 0
 \end{equation}
indexed by $(i,j) \in \ZZ_n \times (\ZZ_n-\{0\})$. These relations do not hold in our $Q_{n,1}(E,\tau)$ because our $\theta_\a$'s are not the same as those in \cite{FO-Kiev}. If $n$ is odd and $\omega=e(\frac{1}{n})$, then the relations
\begin{equation}
\label{FO.relns.2}
\frac{x_i^2}{\theta_j(\tau)\theta_{-j}(\tau)}  \, + \, \omega \, \frac{x_{i-1}x_{i+1}}{\theta_{1+j}(\tau)\theta_{1-j}(\tau)}  \, +   \, \cdots\, + \,  
\omega^{n-1} \, \frac{x_{i-(n-1)}x_{i+n-1}}{\theta_{n-1+j}(\tau)\theta_{n-1-j}(\tau)}  \,=\, 0, 
\; \;
(i,j) \in \ZZ_n \times (\ZZ_n-\{0\}),
 \end{equation}
 hold in $Q_{n,1}(E,\tau)$. If  $n$ is odd and $\psi_a(z)=e\big(-\tfrac{\a(n+1)}{2n}\big)\theta_{\a}(z)$, as in  \cref{subsec.alt.basis}, then 
\begin{equation}
\label{FO.relns.3}
\frac{x_i^2}{\psi_j(\tau)\psi_{-j}(\tau)}  \, + \,   \frac{x_{i-1}x_{i+1}}{\psi_{1+j}(\tau)\psi_{1-j}(\tau)}  \, + \,  \cdots \, + \,  
\frac{x_{i-(n-1)}x_{i+n-1}}{\psi_{n-1+j}(\tau)\psi_{n-1-j}(\tau)}  \;=\; 0
 \end{equation}
 in $Q_{n,1}(E,\tau)$. It is likely that the $\theta_\a$'s in the Kiev preprint (for $n$ odd) are the $\psi_\a$'s.

\cref{prop.new.rel} provides relations for $Q_{n,k}(E,\tau)$ that are similar to those in \cref{FO.relns.2} in the sense that the indices on the 
$x$'s involve $i$ but not $j$ and the indices on the $\theta_\a$'s involve $j$ but not $i$.

\subsection{Extending the definition of $Q_{n,k}(E,\tau)$ to all $\tau \in \CC$ when $k=1$}
\label{subsec.k=1.defns}
Feigin and Odesskii provide three ways to extend the definition of $Q_{n,1}(E,\tau)$ to all $\tau \in \CC$.
The results in this section are theirs: we make some of their implicit statements explicit, 
fill in some details, and explain some incompatibilities between their conventions.

\subsubsection{Conventions}
\label{ssect.symalt.conv}

If $W$ is a finite dimensional $\CC$-vector space we will write $\Sym^{d}W$ and $\Alt^d W$ for the subspaces of $W^{\otimes d}$ 
on which the symmetric group of order $d!$ acts via the trivial and sign representations, respectively.

Let $W$ be a finite dimensional $\CC$-vector space of $\CC$-valued functions on a set $X$. We adopt the convention that 
$W^{\otimes d}$ acts as functions on $X^d$ by $(w_1 \otimes \cdots \otimes w_d)(x_1,\ldots,x_d)=w_1(x_1)\cdots w_d(x_d)$.
Thus, $\Sym^{d}W$ (resp., $\Alt^d W$) consists of symmetric (resp., anti- or skew-symmetric) functions $X^d \to \CC$.

\begin{lemma}
\label{lem.theta.tensor}
Fix $c\in\CC$. By the above convention, $\Theta_{n,c}(\Lambda)^{\otimes d}$ is identified with the space of holomorphic functions $f$ on $\CC^{d}$ such that $f$ is a function in $\Theta_{n,c}(\Lambda)$ in each variable. $\Sym^{d}\Theta_{n,c}(\Lambda)$ and $\Alt^{d}\Theta_{n,c}(\Lambda)$
are identified with those functions that are symmetric and anti-symmetric, respectively.
\end{lemma}
\begin{proof}
Let $f:\CC^d \to \CC$ be a holomorphic function that is a function in $\Theta_{n,c}(\Lambda)$ in each variable. Since the function $z_{d}\mapsto f(z_{1},\ldots,z_{d})$ belongs to $\Theta_{n,c}(\Lambda)$ for each $(z_{1},\ldots,z_{d-1}) \in\CC^{d-1}$, there are unique functions 
$\rho_{\a}\colon\CC^{d-1}\to\CC$, $\a\in\ZZ_{n}$, such that
\begin{equation*}
	f(z_{1},\ldots,z_{d})\;=\;\sum_{\a\in\ZZ_{n}}\rho_{\a}(z_{1},\ldots,z_{d-1})\theta_{\a,c}(z_{d})
\end{equation*}
as functions of $z_{d}$, where $\{\theta_{\alpha,c}\mid\alpha\in\ZZ_{n}\}$ is the basis for $\Theta_{n,c}(\Lambda)$ defined in \cref{theta.basis.c}. The quasi-periodicity of $f$ with respect to $z_{1},\ldots,z_{d-1}$ implies that the $\rho_{\a}$'s have the same quasi-periodicity properties.

We will now show that $\rho_{\a}$ is a holomorphic function. Since $\{\theta_{0,c}, \ldots, \theta_{n-1,c}\}$ is linearly independent,
\begin{equation*}
	\operatorname{span}\{(\theta_{0,c}(z), \ldots , \theta_{n-1,c}(z) )  \mid z\in\CC\}\;=\;\CC^{n}.
\end{equation*}
Thus, for fixed $\b\in\ZZ_{n}$, there is a finite set of points $(t_{i},x_i)\in\CC^2$ such that $\sum_{i}t_{i}\theta_{\a,c}(x_{i})=\delta_{\a,\b}$, the Kronecker delta. Hence
\begin{align*}
	\rho_{\b}(z_{1},\ldots,z_{d-1})
	&\;=\;\sum_{\a\in\ZZ_{n}}\rho_{\a}(z_{1},\ldots,z_{d-1})\delta_{\a,\b}\\
	&\;=\;\sum_{\a\in\ZZ_{n}}\sum_{i}\rho_{\a}(z_{1},\ldots,z_{d-1})t_{i}\theta_{\a,c}(x_{i})\\
	&\;=\;\sum_{i}t_{i}f(z_{1},\ldots,z_{d-1},x_{i})
\end{align*}
is a holomorphic function on $\CC^{d-1}$. Therefore $\rho_{\a}$ is a function in $\Theta_{n,c}(\Lambda)$ in each variable.

Applying this procedure to $\rho_{\a}$ inductively, we deduce that $f$ is a linear combination of the functions of the form $\theta_{\a_{1}}\otimes\cdots\otimes\theta_{\a_{d}}$, and the uniqueness of $\rho_{\a}$ in each step implies that the coefficients of $\theta_{\a_{1}}\otimes\cdots\otimes\theta_{\a_{d}}$'s are unique. This proves the first statement. The second statement follows.
\end{proof}

\subsubsection{Definition of $Q_{n,k}(E,\tau)$ via an ``elliptic'' shuffle product}
\label{ssect.*.product}

The symmetric algebra $SV:=TV/(\Alt^2 V)$ is naturally isomorphic as a graded $\CC$-algebra to 
$$
\Sym V \; := \; \bigoplus_{d=0}^\infty \Sym^d V  \; \subseteq \; TV
$$ 
when $\Sym V$ is endowed with the shuffle product. Feigin and Odesskii proved an ``elliptic'' analogue of this result.
We now follow \cite[\S2]{Od-survey} and \cite[\S1]{FO2001} with some small changes that we will comment on later.

Let $c=\frac{n-1}{2}$. For $d \ge 1$, by \cref{lem.theta.tensor}, the space $\Sym^d  \Theta_{n,c+(1-d)n\tau}(\Lambda)$ is identified with the space of symmetric holomorphic 
functions $f(z_1,\ldots,z_d)$ on $\CC^d$ such that
 \begin{align*}
 f(z_1+1,z_2,\ldots,z_d) & \;=\;  f(z_1,z_2,\ldots,z_d)
 \\
  f(z_1+\eta,z_2,\ldots,z_d) & \;=\;  e(-nz_1+ c+(1-d)n\tau+\tfrac{n}{2}) \, f(z_1,z_2,\ldots,z_d).
\end{align*}
with the convention $\Sym^0=\CC$.
We now define the graded vector space
$$
F\;=\;F_{n}(E,\tau) \; :=\; \bigoplus_{d= 0}^\infty  \Sym^d  \Theta_{n,c+(1-d)n\tau}(\Lambda).
$$
Note that $F_0=\CC$ and $F_1=\Theta_{n,c}(\Lambda)=\Theta_{n}(\Lambda)$.
Since $\dim \Theta_{m,c}(\Lambda)=m$ for all $c \in \CC$ and all $m\geq 1$, 
$\dim F_d= \binom{n+d-1}{d}$, 
which is the same as the dimension of the degree-$d$ component of the polynomial ring on $n$ variables.

For $\a,\b \in \ZZ_{\geq 0}$, let $S_{\a+\b}$ denote the group of permutations of $\{1,\ldots,\a+\b\}$ and define
$$
S_{\a|\b}    \; :=\; \{ \s \in S_{\a+\b} \; | \; \s(1)<\cdots <\s(\a) \text{ and }\s(\a+1)<\cdots <\s(\a+\b)\}.
$$ 

\begin{proposition}
\cite[p.~1137 and Prop.~10, p.~1142]{Od-survey}\footnote{See also \cite[Prop., p.~37]{FO98}.}
\label{prop.10.survey}
\label{prop.Q_n.maps.to.F_n} 
The space $F_{n}(E,\tau)$ is a graded $\CC$-algebra with respect to the multiplication $*$ defined as follows:
if $f \in F_\a$ and $g \in F_\b$, then 
\begin{align}
&(f*g)(z_1,\ldots,z_{\a+\b})\\
\label{eq.def.star.first}&\qquad\; :=\; \frac{1}{\a!\b!}\sum_{\s \in S_{\a+\b}}c_{\a,\b,\s}(\sfz) 
f(z_{\s(1)},\ldots,z_{\s(\a)}) g(z_{\s(\a+1)}+2\a\tau,\ldots,z_{\s(\a+\b)}+2\a\tau)\\
&\qquad\; \phantom{:}=\; \sum_{\s \in S_{\a|\b}}c_{\a,\b,\s}(\sfz) 
f(z_{\s(1)},\ldots,z_{\s(\a)}) g(z_{\s(\a+1)}+2\a\tau,\ldots,z_{\s(\a+\b)}+2\a\tau)
\end{align}
where
$$
c_{\a,\b,\s}(\sfz)\;=\; c_{\a,\b,\s}(z_1,\ldots,z_{\a+\b}) \;=\; \prod_{\substack{1 \le i \le \a  \\ \a+1 \le   j \le \a+\b   }} \frac{\theta(z_{\s(i)} - z_{\s(j)} +n \tau)}{\theta(z_{\s(i)} - z_{\s(j)} )}.
$$
If $\tau \notin \frac{1}{n}\Lambda$,  the map $x_i \mapsto \theta_i$ extends to a homomorphism of graded $\CC$-algebras,  
$Q_{n,1}(E,\tau) \to F_n(E,\tau) $. 
\end{proposition}
\begin{proof}
It is proved in \cite[Prop.~5, p.~1137]{Od-survey} that $f*g$ is holomorphic on $\CC^{\a+\b}$.
A straightforward computation shows that $*$ is associative. 

To prove that  the map $x_i \mapsto \theta_i$ extends to a homomorphism we must show that 
\begin{equation}
\label{eq:reln.maps.to.0.in.F_n}
\sum_{r\in\ZZ_n}\frac{1}{\theta_{j-i-r}(-\tau)\theta_{r}(\tau)} \, \big( \theta_{j-r} * \theta_{i+r}\big)(x,y) \;=\; 0
\end{equation}
for all $i,j\in \ZZ_n$ and all $(x,y) \in \CC^2$. If $f,g \in F_1= \Theta_n(\Lambda)$, then
\begin{equation}
\label{new.defn.f*g}
(f*g)(x,y) \;=\; f(x)g(y+2\tau) \, \frac{\theta(x-y+n\tau)}{\theta(x-y)} \, + \, f(y)g(x+2\tau) \, \frac{\theta(y-x+n\tau)}{\theta(y-x)}\, 
\end{equation}
so we must show that
\begin{align*}
\sum_{r \in \ZZ_n} &\, \frac{1}{\theta_{j-i-r}(-\tau)\theta_r(\tau)} \; \times \;
\\
&
\Bigg(  \frac{\theta(x-y+n\tau)}{\theta(x-y)} \, \theta_{j-r} (x)\theta_{i+r} (y+2\tau) \; + \;
 \frac{  \theta(y-x+n\tau)}{\theta(y-x)}\  \, \theta_{j-r} (y)\theta_{i+r} (x+2\tau) 
\Bigg)
\; = \; 0.
\end{align*}
After changing notation, equation (30) in \cite{Od-survey} (see also \cite[Cor.~5.10]{CKS2}) says that
\begin{align}
	&\frac{\theta(-n\tau+x-y)}{\theta(x-y)} \, \Big(\theta_{i}(x+\tau)\theta_{j}(y+\tau)-\theta_{i}(y+\tau) \theta_{j}(x+\tau)\Big)
	\label{42905872}\\
	& \qquad \; =\;  d \, \sum_{r\in\ZZ_n}\frac{1}{\theta_{j-i-r}(-\tau)\theta_{r}(\tau)} \, \theta_{j-r}(x)\theta_{i+r}(y+2\tau)
	\label{42905873}
\end{align}
where
$
d = \tfrac{1}{n}\theta(\tfrac{1}{n})\cdots\theta(\tfrac{n-1}{n})\, \theta(-n\tau) \theta_{j-i}(0).
$
So we must show that
$\text{\cref{eq:first.half}}+\text{\cref{eq:second.half}}=0$ where 
\begin{equation}
\label{eq:first.half}
 \frac{\theta(x-y+n\tau)}{\theta(x-y)} \, d^{-1} \,  \frac{\theta(-n\tau+x-y)}{\theta(x-y)} \, \big(\theta_{i}(x+\tau)\theta_{j}(y+\tau)-\theta_{i}(y+\tau) \theta_{j}(x+\tau)\big)
 \end{equation}
and 
\begin{equation}
\label{eq:second.half}
 \frac{\theta(y-x+n\tau)}{\theta(y-x)} \, d^{-1} \,   \frac{\theta(-n\tau+y-x)}{\theta(y-x)} \, \big(\theta_{i}(y+\tau)\theta_{j}(x+\tau)-\theta_{i}(x+\tau) \theta_{j}(y+\tau)\big).
 \end{equation}
However, $\theta(-z)=-e(-z)\theta(z)$, so
$$
 \frac{\theta(x-y+n\tau)}{\theta(x-y)} \, d^{-1} \,    \frac{\theta(-n\tau+x-y)}{\theta(x-y)}  \;=\; 
 \frac{\theta(y-x+n\tau)}{\theta(y-x)} \, d^{-1} \,    \frac{\theta(-n\tau+y-x)}{\theta(y-x)}.
$$
It follows that $\text{\cref{eq:first.half}}+\text{\cref{eq:second.half}}=0$.
 \end{proof}
 
\cref{prop.10.survey} should be compared to Proposition 10 in Odesskii's survey \cite[p.~1142]{Od-survey}
which says that the map $x_i \mapsto \theta_i$ extends to an algebra isomorphism
$Q_{n,1}(E,\tau)  \to   F_n(E,-\tau)$; this is not correct---the last sentence on p.~1142 is not true. Indeed, when $z_1-z_2=n\tau$, that sentence (with $\eta$ replaced by $\tau$) and $\text{\cref{42905872}}=\text{\cref{42905873}}$ imply that
$\theta_i(z_1-n\tau+\tau) \theta_j(z_1-3\tau)-\theta_i(z_1-3\tau) \theta_j(z_1-n\tau+\tau)=0$ for all $i,j,z_1,\tau$.
However, 
if $z_1=3\tau-\frac{i}{n}\eta$, then $ \theta_i(z_1-3\tau)=0$ so that we obtain $ \theta_i(3\tau-\frac{i}{n}\eta-n\tau+\tau) \theta_j(-\frac{i}{n}\eta) =0$ for all $i,j,\tau$; this is clearly false. A corrected version of Proposition 10 would say that the map $x_{-i} \mapsto \theta_i$ 
extends to an algebra homomorphism $Q_{n,1}(E,\tau) \to F_n(E,-\tau)$ (by \cref{prop.anti.isom}). We have not been able to verify whether this is an 
isomorphism when $\tau \notin \frac{1}{n}\Lambda$. For example, to show this map is surjective one would have to show that 
$F_n(E,\tau)$ is generated by its degree-one component and we have not been able to verify that.

Since $c_{\a,\b,\sigma}(\sfz)=1$ when $\tau=0$, the multiplication on $F_n(E,0)=\Sym\Theta_{n}(\Lambda)$ is the usual shuffle product.

\subsubsection{A definition of $\rel_{n,1}(E,\tau)$ as a space of the holomorphic functions on $\CC^2$}
\label{ssect.defn.using.Os.identity}
This subsection makes no use of the material in \cref{ssect.*.product}. 

In this subsection we identify the degree-one component of  $Q_{n,1}(E,\tau)$ with $\Theta_{n}(\Lambda)$ 
via $x_i \leftrightarrow \theta_i$. With this convention, $ \rel_{n,1}(E,\tau)$ is a subspace of $\Theta_{n}(\Lambda)^{\otimes 2}$ and 
elements of $\rel_{n,1}(E,\tau)$ are holomorphic functions $\CC^2 \to \CC$ (see the convention in \cref{ssect.symalt.conv}).

\begin{proposition}
\label{prop.FO}
Assume $\tau \notin \frac{1}{n}\Lambda$. 
The map 
$$
\psi:\Alt^2\Theta_{n}(\Lambda)\; \longrightarrow \;  \rel_{n,1}(E,\tau)
$$
given by
\begin{equation}\label{eq.def.psi}
	\psi(f)(x,y)\;:=\;\frac{\theta(x-y+(2-n)\tau)}{\theta(x-y+2\tau)}f(x+\tau,y-\tau)
\end{equation}
is an isomorphism of vector spaces.  Therefore 
$$
\dim \rel_{n,1}(E,\tau) \;=\; \tbinom{n}{2}.
$$
\end{proposition}
\begin{proof}
Since $\dim \Theta_{n}(\Lambda)=n$, the dimension of $\Alt^2\Theta_{n}(\Lambda)$ 
is $n \choose 2$. Thus the final conclusion of the proposition follows from the first.

Let $\Psi$ be the automorphism of the field of $\CC$-valued meromorphic functions on $\CC^2$ defined by the same formula as \cref{eq.def.psi}. Since $\psi$ is a restriction of $\Psi$, it suffices to show that $\Psi(\Alt^{2}\Theta_{n}(\Lambda))=\rel_{n,1}(E,\tau)$.

Since the $\theta_\a$'s form a basis for $\Theta_n(\Lambda)$, the domain $\Alt^{2}\Theta_{n}(\Lambda)$ of $\psi$ is the linear span of the functions
\begin{equation}
\label{eq:holom.fn}
f_{i,j}(x,y) \; :=\; \theta_{i}(x)\theta_{j}(y)-\theta_{i}(y) \theta_{j}(x),
\qquad i,j\in\ZZ_{n}.
\end{equation}
Define 
\begin{equation}
\label{half.complicated.identity}
h_{i,j}(x,y) \; := \; 
d \, 
 \sum_{r\in\ZZ_n}\frac{1}{\theta_{j-i-r}(-\tau)\theta_{r}(\tau)}
 \theta_{j-r}(x)\theta_{i+r}(y)
\end{equation}
where $d =  \tfrac{1}{n}\theta(\tfrac{1}{n})\cdots\theta(\tfrac{n-1}{n}) \theta(-n\tau)\theta_{j-i}(0)$.
Because we are identifying $V$ with $\Theta_n(\Lambda)$ via $x_i \leftrightarrow \theta_i$, 
$$
V^{\otimes 2} \ni r_{i,j} \, \longleftrightarrow \,  d^{-1}\theta_{j-i}(0) h_{i,j} \in   \Theta_n(\Lambda)^{\otimes 2}.
$$ 
Thus, $\rel_{n,1}(E,\tau)=\operatorname{span}\{h_{i,j} \; | \; (i,j) \in \ZZ_{n}^2\}$. 

The identity $\text{\cref{42905872}}=\text{\cref{42905873}}$ says that $\Psi(f_{i,j})(x,y+2\tau)=h_{i,j}(x,y+2\tau)$ so $\Psi(f_{i,j})=h_{i,j}$. Therefore $\Psi(\Alt^{2}\Theta_{n}(\Lambda))=\rel_{n,1}(E,\tau)$.
\end{proof}

In \cref{ssect.cf.TVdB}, we describe the relation between \cref{prop.FO} and the description of $\rel_{n,1}(E,\tau)$ that is used in Tate and 
Van den Bergh's paper \cite{TvdB96}.

If we view $h_{i,j}(x,y)$  as a meromorphic function of $(x,y,\tau)\in\CC^3$, then the singularities at $\CC^2 \times \frac{1}{n}\Lambda$
are removable.

\begin{lemma}\label{lem.hij.ext}
Fix $(i,j) \in \ZZ_n^2$. The function $h_{i,j}(x,y)$, viewed as a function of $(x,y,\tau)$ defined 
on $\CC^2 \times (\CC-\frac{1}{n}\Lambda)$, extends uniquely to a holomorphic function on $\CC^3$.
\end{lemma}
\begin{proof}
Assume $\tau \in \frac{1}{n}\Lambda$. Since $n\tau \in \Lambda$, $\theta(-n\tau)=0$. If only one of $\theta_{j-i-r}(-\tau)$ and
$\theta_r(\tau)$ is zero, the potential pole at $\tau$ is canceled by the vanishing of $\theta(-n\tau)$.

If $\theta_{j-i-r}(-\tau)=\theta_r(\tau)=0$, then $-\tau \in -\frac{j-i-r}{n}\eta + \frac{1}{n}\ZZ + \ZZ\eta$ and 
$\tau \in -\frac{r}{n}\eta + \frac{1}{n}\ZZ + \ZZ\eta$.
It follows that $0\in-\frac{j-i}{n}\eta+\frac{1}{n}\ZZ + \ZZ\eta$, whence $\theta_{j-i}(0)=0$; thus $h_{i,j}(x,y)$ is identically zero.
\end{proof}
 
We also write  $h_{i,j}(x,y)$ for the holomorphic extension of $h_{i,j}(x,y)$ to $\CC^3$ and define, for all $\tau \in \CC$,
\begin{equation}
\label{defn.rel.n1}
\rel_{n,1}(E,\tau) \; :=\;  \text{span}\{h_{i,j} \;   | \; (i,j) \in \ZZ_n^2\}
\end{equation}
and 
\begin{equation}
\label{defn.Q.n1}
Q_{n,1}(E,\tau) \; :=\; \frac{T(\Theta_n(\Lambda))}{(\rel_{n,1}(E,\tau))} \, .
\end{equation}
The isomorphism $\psi$ in \cref{prop.FO} makes sense for all $\tau \in \CC$ so, for all $\tau \in \CC$,
$$
\dim \rel_{n,1}(E,\tau) \; =\;  \tbinom{n}{2}.
$$

\subsubsection{Comparing conventions and results in \cite{FO89} with those in this paper}
The next result ``disagrees'' with the implicit assertion in \cite[\S2]{FO89} that the
quadratic relations for $Q_{n,1}(E,\tau)$ are the functions in $\Theta_n(\Lambda)^{\otimes 2}$ that satisfy the properties 
(a) and (b) at \cite[pp.~210--211]{FO89} (with $s=2$); 
condition (a) says that the the quadratic relations for $Q_{n,1}(E,\tau)$ vanish on $\{(x,x+(n-2)\tau) \; | \; x \in \CC\} $. 

\begin{lemma}
\label{lem.shifted.diag}
The function $h_{ij}$ vanishes on the line
$\{(x,x+(2-n)\tau) \; | \; x \in \CC\}$ in $\CC^2$.
\end{lemma}
\begin{proof}
By \cref{eq.def.psi}, $\psi(f_{ij})$ vanishes on this line. The conclusion follows because $\psi(f_{ij})=h_{ij}$.
\end{proof}

The disagreement is apparent rather than real because Odesskii and Feigin are using a different (unstated) 
convention than the one we adopted just before \cref{prop.FO}. 
In  \cite[\S2]{FO89}  they use the convention that $(f \otimes g)(x,y) =f(y)g(x)$. That is appropriate because
if $U$ and $V$ are finite dimensional vector spaces one should identify $(U \otimes V)^*$ with $V^* \otimes U^*$, not with
$U^* \otimes V^*$.\footnote{That this is the ``right'' convention is apparent when $U$ and $V$ are finite dimensional modules over
a $\CC$-algebra $A$: if $U$ is a right $A$-module and $V$ a left $A$-module, then $U^*$ becomes a left $A$-module
and $V^*$  becomes a right $A$-module and there is a natural map $V^* \otimes_A U^* \to (U \otimes_A V)^*$ 
(one can not reverse the order of the tensorands in this situation).} Nevertheless, we will use the 
convention stated  just before \cref{prop.FO}.

\subsubsection{A geometric definition of $\rel_{n,1}(E,\tau)$}
\label{ssect.geom.defn}
Fix arbitrary points $p_{i}=q_{i}+\Lambda\in E=\CC/\Lambda$, $1\leq i\leq n$, such that $p_{1}+\cdots+p_{n}=\frac{n-1}{2}+\Lambda$ and define $\cL:=\cO_E(D)$, where $D=(p_{1})+\cdots+(p_{n})$. As mentioned in \cref{ssec.emb.E}, there is $s\in\Theta_{n}(\Lambda)$, unique up to non-zero scalar multiples, that vanishes exactly at $q_{1},\ldots,q_{n}$ modulo $\Lambda$, counted with multiplicity.
There is an isomorphism of vector spaces $\Theta_{n}(\Lambda)\to H^0(E,\cL)$, $g\mapsto g/s$, and hence 
an identification between $\Theta_{n}(\Lambda)^{\otimes 2}$ and $H^0(E\times E,\cL \boxtimes \cL)$.
Each $g \in \rel_{n,1}(E,\tau) \subseteq \Theta_{n}(\Lambda)^{\otimes 2}$ can therefore be considered as a 
global section of $\cL \boxtimes \cL$ and as such it has a divisor of zeros  that we denote by $(g)_0$
(when $g \ne 0$). By \Cref{lem.shifted.diag}, $(g)_0$ contains the shifted diagonal 
$ \Delta_{(2-n)\tau}:= \{(x,x+(2-n)\tau)\}  \subseteq E^2$. 

The fixed locus of the involution $(x,y) \mapsto (y-2\tau,x+2\tau)$ on $E^2$ is $\Delta_{2\tau}:=\{(x,x+2\tau)\}$. 

For non-zero $g \in H^{0}(E\times E,\cL\boxtimes\cL)$, we define the following conditions:
\begin{enumerate}
	\item[(a$'$)] $(g)_{0}-\Delta_{(2-n)\tau}$ is an effective divisor; i.e., $g$ vanishes along $\Delta_{(2-n)\tau}$.
	\item[(b1$'$)] $(g)_{0}-\Delta_{(2-n)\tau}$ is stable under the involution $(x,y)\mapsto (y-2\tau,x+2\tau)$ on $E^{2}$.
	\item[(b2$'$)] $(g)_{0}-\Delta_{(2-n)\tau}$ contains $\Delta_{2\tau}$ with even, possibly zero, multiplicity.
\end{enumerate}
Condition (a$'$) is the analogue of (a) at \cite[pp.~210--211]{FO89} for $s=2$. Conditions (b1$'$) and (b2$'$) are the analogues of the first and the second assertions of (b), respectively, when $s=2$. \Cref{lem.shifted.diag} says that the quadratic relations for $Q_{n,1}(E,\tau)$ satisfy condition (a$'$).

Let 
\begin{align*}
D_2'   &  \; :=\;  \{\text{functions in $ \Theta_{n}(\Lambda)^{\otimes 2}$ that satisfy (a$'$), (b1$'$), and (b2$'$)}\} \, \cup \, \{0\}.
\end{align*}
Thus, $D_2'$  is the analogue of Odesskii and Feigin's space $D_2$ defined at \cite[p.~210]{FO89}.

\begin{lemma}
\label{lem.another.shifted.diag}
Let $0 \ne g \in \rel_{n,1}(E,\tau) \subseteq \Theta_{n}(\Lambda)^{\otimes 2}$. Then $g$ satisfies (b1\,$'$) and (b2\,$'$).
\end{lemma}
\begin{proof}
(b1$'$) 
If $\psi$ is the isomorphism in \cref{prop.FO}, then $g=\psi(f)$ for some $f\in\Alt^{2}\Theta_{n}(\Lambda)$. 
Let $p(x,y)$ be the numerator of the fraction in \cref{eq.def.psi} and let $q(x,y):=g(x,y)/p(x,y)$. The zero locus of $g$ is the union of the zero loci of $p$ and $q$, counted with multiplicity. The zero locus of $p$ is the inverse image of $\Delta_{(2-n)\tau}$ under the projection $\CC\to E$. Since
\begin{align*}
	q(y-2\tau,x+2\tau)
	&\;=\;\frac{f(y-\tau,x+\tau)}{\theta(y-x-2\tau)}\\
	&\;=\;\frac{-\,f(x+\tau,y-\tau)}{-\,e(y-x-2\tau)\theta(x-y+2\tau)}\\
	&\;=\;\frac{q(x,y)}{e(y-x-2\tau)},
\end{align*}
the zero locus of $q$ is stable under the involution $(x,y)\mapsto (y-2\tau,x+2\tau)$ on $\CC^{2}$. Thus $g$ satisfies (b1$'$).

(b2$'$)
Write $g=\psi(f)$ as before. Since $f$ is an anti-symmetric function, the zero locus of $f$ contains the diagonal $\D=\{(x,x)\}$ with odd multiplicity. Suppose $\tau\notin\frac{1}{n}\Lambda$. Since the denominator of the fraction in \cref{eq.def.psi} has zeros along $\Delta_{2\tau}$ with multiplicity one and the zero locus of the numerator does not contain $\Delta_{2\tau}$, the zero locus of $g=\psi(f)$ contains $\Delta_{2\tau}$ with even multiplicity. If $\tau\in\frac{1}{n}\Lambda$, then the theta functions in the numerator and the denominator of \cref{eq.def.psi} cancel each other so $g(x,y)=\psi(f)(x,y)=
f(x+\tau,y-\tau)$ whence $(g)_{0}$ contains $\Delta_{2\tau}$ with odd multiplicity
$\ge 1$. Hence $(g)_{0}-\Delta_{2\tau}$ contains $\Delta_{2\tau}$ with even multiplicity.
\end{proof}

\begin{lemma}\label{lem.rel.vs.d.two}
For all $\tau \in \CC$, $\rel_{n,1}(E,\tau) = D_2'$.
\end{lemma}
\begin{proof}
By \cref{lem.shifted.diag,lem.another.shifted.diag}, $\rel_{n,1}(E,\tau) \subseteq D_2'$. For each $g\in D'_{2}$, define
\begin{equation*}
\varphi(g)(x,y)\;:=\;\frac{\theta(x-y)}{\theta(x-y-n\tau)}g(x-\tau,y+\tau).
\end{equation*}
It suffices to show that $\varphi(g)\in\Alt^{2}\Theta_{n}(\Lambda)$ because having done that the (obvious) fact that $\psi\varphi=\id$ then implies that 
$D'_{2} \subseteq \rel_{n,1}(E,\tau)$. 

Since $g$ satisfies (a$'$), $\varphi(g)$ is holomorphic on $\CC^{2}$ and hence belongs to $\Theta_{n}(\Lambda)^{\otimes 2}$. Condition (b1$'$) implies that $(\varphi(g))_{0}$ is stable under the action $(x,y)\mapsto (y,x)$ on $\CC^{2}$. Since the functions $\varphi(g)(x,y)$ and $\varphi(g)(y,x)$ have the same divisor of zeros, their ratio is a nowhere vanishing holomorphic function on $\CC^{2}$ that is doubly periodic with respect to both $x$ and $y$, and therefore constant. So $\varphi(g)(x,y)=a\varphi(g)(y,x)$ for some non-zero $a\in\CC$. Since $\varphi(g)(x,y)=a\varphi(g)(y,x)=a^{2}\varphi(g)(x,y)$, $a=\pm 1$; hence $\varphi(g)$ is either symmetric or anti-symmetric. Condition (b2$'$) implies that $(\varphi(g))_{0}$ contains $\Delta$ with odd multiplicity, so $\varphi(g)$ is anti-symmetric.
\end{proof}

\begin{proposition}
Assume $\tau \notin \frac{1}{n}\Lambda$. The map $x_i \mapsto \theta_i$ extends to an  isomorphism
\begin{equation}
\label{defn.of.Qn1} 
Q_{n,1}(E,\tau) \, \stackrel{\sim}{\longrightarrow} \,   \frac{T(\Theta_{n}(\L))}{(D'_2)}. 
\end{equation}
\end{proposition}
\begin{proof}
This is an immediate consequence of \cref{lem.rel.vs.d.two}. 
\end{proof}

We could use the right-hand side of \cref{defn.of.Qn1} to define $Q_{n,1}(E,\tau)$ for all $\tau \in \CC$.
That definition would agree with that in \cref{defn.Q.n1}.

Since we are identifying $\Theta_{n}(\Lambda)$ with $H^{0}(E,\cL)$, the isomorphism in \cref{defn.of.Qn1} can be written as 
\begin{equation}
\label{defn.of.Qn1.2} 
Q_{n,1}(E,\tau) \, \cong \, \frac{T(H^{0}(E,\cL))}{(D''_2)}
\end{equation}    
where $D_2''$ is the subspace of $H^{0}(E^2,\cL \boxtimes \cL)$ consisting of the sections satisfying conditions (a$'$), (b1$'$), and (b2$'$).
We could therefore use the right-hand side of \cref{defn.of.Qn1.2} as a definition of $Q_{n,1}(E,\tau)$.

The virtue of using the right-hand side of \cref{defn.of.Qn1.2} as a definition of  $Q_{n,1}(E,\tau)$ 
is that it allows one to define $Q_{n,1}(E,\tau)$ for {\it any} base field 
and any $E$ having a line bundle of degree $n$ \cite[\S4.1]{TvdB96}.
It would be very useful to have a similar ``geometric'' definition of $Q_{n,k}(E,\tau)$ when $k>1$.

\subsubsection{Comparison with Tate-Van den Bergh's construction of $Q_{n,1}(E,\tau)$}
\label{ssect.cf.TVdB} 
Denote by $\tau$ the translation automorphism $x\mapsto x+\tau$ of $E$. In \cite[\S4.1]{TvdB96}, Tate and Van den Bergh considered an isomorphism
\begin{equation*}
\phi:(1,\tau^{-2})^{*}((\tau^{*}\cL\boxtimes\tau^{*}\cL)(-\Delta))\;\longrightarrow\;(\cL\boxtimes\cL)(-\Delta_{(2-n)\tau})
\end{equation*}
where $(1,\tau^{-2})(x,y)=(x,y-2\tau)$, and defined the space of quadratic relations for $Q_{n,1}(E,\tau)$ to be $\phi\big((1,\tau^{-2})^{*}(\Alt^{2}H^{0}(E,\tau^{*}\cL))\big)$.\footnote{The automorphisms $\sigma$ and $\theta$ in \cite[\S4.1]{TvdB96} are our $\tau^{-1}$ and $\tau^{2}$, respectively.} We will now describe the relation between $\phi$ and the isomorphism $\psi:\Alt^2\Theta_{n}(\Lambda)\to\rel_{n,1}(E,\tau)$ in \cref{prop.FO}.

The domain $(1,\tau^{-2})^{*}((\tau^{*}\cL\boxtimes\tau^{*}\cL)(-\Delta))$ of $\phi$  equals  $(\tau^{*}\cL\boxtimes(\tau^{-1})^{*}\cL)(-\Delta_{2\tau})$ and $\phi$ is the composition
\begin{equation}\label{eq.tvdb.comp}
\xymatrix{
(\tau^{*}\cL\boxtimes(\tau^{-1})^{*}\cL)(-\Delta_{2\tau})\ar[r]^-{\varepsilon} &
(\cL\boxtimes\cL)(-\Delta)\ar[r]^-{\delta} &
(\cL\boxtimes\cL)(-\Delta_{(2-n)\tau})	
}
\end{equation}
where $\varepsilon(f)(x,y):=f(x-\tau,y+\tau)$ and  
\begin{equation*}
\delta(f)(x,y)\;:=\;\frac{s(x+\tau)s(y-\tau)}{s(x)s(y)}\, \frac{\theta(x-y+(2-n)\tau)}{\theta(x-y+2\tau)}f(x+\tau,y-\tau)
\end{equation*}
where  $s \in \Theta_n(\Lambda)$ is the function identified in the first paragraph of \cref{ssect.geom.defn}.
The map $\d$ is the global version of the isomorphism $\psi$ in \cref{prop.FO}; the terms involving $s$ in the definition of $\d$ occur 
because we are identifying
$\Theta_{n}(\Lambda)$  and  $H^{0}(\cL)$ via $g \mapsto g/s$.

Thus $H^{0}$ applied to \cref{eq.tvdb.comp} induces isomorphisms
\begin{equation*}
\xymatrix{
(1,\tau^{-2})^{*}(\Alt^{2}H^{0}(E,\tau^{*}\cL))\ar[rr]^-{H^{0}(\varepsilon)} &&
\Alt^{2}H^{0}(\cL)\ar[rr]^-{H^{0}(\delta)} &&
\rel_{n,1}(E,\tau)
}
\end{equation*}
where the last isomorphism $H^{0}(\delta)$ is equal to $\psi$ via the identification $\Theta_{n}(\Lambda)\cong H^{0}(\cL)$.

\subsection{Extending the definition of  $\rel_{n,k}(E,\tau)$ and $Q_{n,k}(E,\tau)$ to all $\tau \in \CC$ when $k\ge 1$}
\label{subsec.rel.tor.pts}
In this subsection we consider three ways of defining $\rel_{n,k}(E,\tau)$ for all $\tau \in \CC$, and show they produce the same space
in ``good'' situations.

\subsubsection{The first method}
\label{ssect.Lij}
If $\tau \in \CC- \frac{1}{n}\Lambda$ and $r_{ij}(\tau) \ne 0$, we define 
$$
L_{ij}(\tau)\; := \;  \CC r_{ij}(\tau).
$$ 
In \Cref{prop.ext.mor.rel} we use a standard result in projective algebraic geometry to define $L_{ij}(\tau)$ for all $\tau \in \CC$;
in \cref{def.Qnk} we then define $\rel_{n,k}(E,\tau)$ to be the linear span of these $L_{ij}(\tau)$'s.  
We do {\it not} define $r_{ij}(\tau)$ for all $\tau$.

\begin{proposition}\label{prop.ext.mor.rel}
Fix $(i,j) \in \ZZ_n^2$ such that $r_{ij}(\tau)$ is not identically zero on $\CC-\frac{1}{n}\Lambda$.
When $\tau \in \CC-\frac{1}{n}\L$, let $L_{ij}(\tau)$ be the 1-dimensional subspace of 
$V \otimes V$ spanned by the element
$r_{ij}(\tau)$  in \cref{the-relns}. The map 
\begin{equation}
  \label{eq:1}
  L_{ij}\colon E-E[n] \; \longrightarrow \; \PP(V\otimes V), \qquad \tau\mapsto L_{ij}(\tau),
\end{equation}
is a morphism of algebraic varieties and extends uniquely to a morphism $E \to \PP(V\otimes V)$ that we continue to denote by $L_{ij}$.
\end{proposition}
\begin{proof}
Since the zeros of the $\theta_\a$'s belong to $\frac{1}{n}\L$, the hypothesis that $\tau$ is not in $\frac{1}{n}\L$ ensures that 
the coefficient of every $x_{j-r}\otimes x_{i+r}$ in $r_{ij}(\tau)$ is a well-defined number. By hypothesis,
 at least one of those coefficients is non-zero so $r_{ij}(\tau) \ne 0$ for all $\tau\in E-E[n]$. 
As remarked in \cref{subsec.def.FO.alg}, the subspace $L_{ij}(\tau)$ depends only the image of $\tau$ in $E-E[n]$.
Hence $L_{ij}$ is a well-defined map from $E-E[n]$. 

Since the map $E\to\PP^{n-1}$ given by $z \; \mapsto \; (\theta_0(z),\ldots,\theta_{n-1}(z))$ is a morphism of algebraic varieties, the ratios $\theta_{\alpha}(z)/\theta_{\beta}(z)$ are rational functions on $E$ and therefore regular functions on $E-E[n]$. Thus, since $\theta_\a(-\tau)
= - e(-n\tau+\frac{\a}{n} \big) \theta_{-\a}(\tau)$, the ratio of any two of the coefficients of $r_{ij}(\tau)$ is a regular function on $E-E[n]$. Hence $L_{ij}$ is a morphism of algebraic varieties.

Since $E$ is a non-singular curve, $L_{ij}$ extends uniquely to a morphism $E\to\PP(V\otimes V)$ by using \cite[Prop.~I.6.8]{Hart} repeatedly.
\end{proof}

If $r_{ij}(\tau)$ is not identically zero on $\CC-\frac{1}{n}\Lambda$, we will abuse notation and define, for all $\tau \in \CC$, 
$$
L_{ij}(\tau) \;:=\;  {L}_{ij}(\text{the image of $\tau$ in $E$}) \;\subseteq\; V^{\otimes 2}.
$$

\begin{definition}\label{def.Qnk}
For all $\tau \in \CC$, we define 
\begin{align*}
\rel_{n,k}(E,\tau) & \;:=\; \text{the subspace  of $V^{\otimes 2}$ spanned by the $L_{ij}(\tau)$'s  },
\\
Q_{n,k}(E,\tau)  & \; :=\; \frac{\, \CC\langle x_0,\ldots,x_{n-1}\rangle \, }{(\rel_{n,k}(E,\tau))}.
\end{align*} 
\end{definition}

When $k=1$, this definition agrees with the definition of $\rel_{n,1}(E,\tau)$ in \cref{defn.rel.n1} (\cref{prop.rel.n.one.vs.n.k}).

\begin{proposition}
\label{prop.polyn.Q2}
For all $\tau \in \CC$, $Q_{2,1}(E,\tau)=\bbC[x_{0},x_{1}]$, a polynomial ring on two variables.
\end{proposition}
\begin{pf}
First we consider the case $\tau\notin \frac{1}{2}\Lambda$. Since $\theta_{0}(0)=0$, $r_{00}=r_{11}=0$. The other relations in \cref{the-relns} are
\begin{align*}
r_{01}&\;=\;\theta_{1}(0)\left(\frac{ x_1x_0}{\theta_1(-\tau)\theta_0(\tau)}  \; + \; \frac{x_0x_1}{\theta_0(-\tau)\theta_1(\tau)}\right),\qquad\text{and}
\\
r_{10}&\;=\;\theta_{1}(0)\left(\frac{ x_0x_1}{\theta_1(-\tau)\theta_0(\tau)}  \; + \; \frac{x_1x_0}{\theta_0(-\tau)\theta_1(\tau)}\right)
\end{align*}
in $\CC\langle x_0,x_1 \rangle$. Since $n=2$, 
$$
\theta_\a(-z) \;=\; - e(-2z+\tfrac{\a}{2}) \, \theta_{-\a}(z).
$$
In particular, $\theta_0(-z)=-e(-2z)\theta_0(z)$ and $\theta_1(-z)=e(-2z) \theta_1(z)$ so 
\begin{equation*}
r_{01}\;=\;-\, \frac{\theta_{1}(0)}{e(-2\tau)\theta_{0}(\tau)\theta_{1}(\tau)}\, (x_{0}x_{1}-x_{1}x_{0})\;=\;-r_{10}.
\end{equation*}

Let $(i,j)=(0,1)$ or $(1,0)$. The morphism $L_{ij}\colon E-E[2]\to\PP(V\otimes V)$ is constant with value $\CC.(x_{0}x_{1}-x_{1}x_{0})$ so it extends to the constant morphism $E\to\PP(V\otimes V)$ sending every point in $E$ to $\CC.(x_{0}x_{1}-x_{1}x_{0})$. Therefore
$\rel_{2,1}(E,\tau)=\bbC.(x_{0}x_{1}-x_{1}x_{0})$ and $Q_{2,1}(E,\tau)=\bbC[x_{0},x_{1}]$.
\end{pf}

\subsubsection{The second method}
\label{ssect.R.tau.tau}
For each $\tau \in \CC-\frac{1}{n}\Lambda$, and each $z \in \CC$, we define the linear operator
$R_\tau(z):V^{\otimes 2} \to V^{\otimes 2}$ by the formula
\begin{equation}\label{eq:odr}
  R_\tau(z)(x_i\otimes x_j)  \; := \; 
  \frac{\theta_0(-z) \cdots \theta_{n-1}(-z)}{\theta_1(0) \cdots \theta_{n-1}(0)}
 \sum_{r\in \ZZ_n}
  \frac{\theta_{j-i+r(k-1)}(-z+\tau)}
  {\theta_{j-i-r}(-z)\theta_{kr}(\tau)} \,
  x_{j-r}\otimes x_{i+r}
\end{equation}
for all $(i,j) \in \ZZ_n^2$. The fact that $\tau \notin \frac{1}{n}\Lambda$ ensures that $\theta_{kr}(\tau) \ne 0$ for all $r \in \ZZ_n$
whence $z \mapsto R_\tau(z)$ is a holomorphic function  $\CC \to \End_\CC(V^{\otimes 2})$. 

If $\tau \in \CC- \frac{1}{n}\Lambda$,  then the term before the $\Sigma$ symbol in \cref{eq:odr} is non-zero at $z=\tau$ so 
$R_\tau(\tau)(x_i \otimes x_j)$ is a non-zero scalar multiple of $r_{ij}(\tau)$ and $\rel_{n,k}(E,\tau)=\text{the image of $R_\tau(\tau)$}$.

The term $\theta_1(0) \cdots \theta_{n-1}(0)$ before the $\Sigma$ sign is a normalization factor which ensures that $R_\tau(0)$ is the 
identity operator on $V^{\otimes 2}$. The importance of this becomes apparent in one of our later papers when we exploit the 
fact that $R_\tau(z)$ is a solution to the quantum Yang-Baxter equation (with spectral parameter). The normalization factor 
plays no role in this paper.

As a function of $\tau$, $R_\tau(\tau)$ is holomorphic on $\CC-\frac{1}{n}\Lambda$ and its singularities at $\frac{1}{n}\Lambda$
are removable:

\begin{lemma}
  The function $\tau \mapsto R_\tau(\tau)$ extends uniquely to a holomorphic function on $\CC$, which we also denote by $R_{\tau}(\tau)$.\footnote{We warn the reader that $R_0(0)$, which is defined to be $\lim_{\tau\to 0}R_{\tau}(\tau)$, 
  does not equal $\lim_{\tau\to 0}R_{\tau}(-\tau)$ (see \cite[\S5]{CKS4}).} 
\end{lemma}
\begin{proof}
If the two theta functions in the denominator of one of the summands in the expression 
\begin{equation}\label{eq.r.tau.tau}
R_{\tau}(\tau)(x_{i}\otimes x_{j})\;=\;
 \frac{\theta_0(-\tau) \cdots \theta_{n-1}(-\tau)}{\theta_1(0) \cdots \theta_{n-1}(0)}
 \sum_{r\in \ZZ_n}
  \frac{\theta_{j-i+r(k-1)}(0)}
  {\theta_{j-i-r}(-\tau)\theta_{kr}(\tau)} \,
  x_{j-r}\otimes x_{i+r}
\end{equation}
both vanish  at $\tau$, then $-\tau \in -\frac{j-i-r}{n}\eta+\frac{1}{n}\ZZ+\ZZ\eta$ and $\tau \in -\frac{kr}{n}\eta+\frac{1}{n}\ZZ+\ZZ\eta$
 so $0\in-\frac{j-i+(k-1)r}{n}\eta+\frac{1}{n}\ZZ+\ZZ\eta$, whence $\theta_{j-i+r(k-1)}(0)=0$. Each summand
 therefore has at most a pole of order 1 at $\tau$ which is canceled out by the order-one zero at $\tau$ that appears in the term before 
 the $\Sigma$ sign.
\end{proof}

\begin{lemma}
For all $\tau \in \CC$,
$$
\CC.R_\tau(\tau)(x_i \otimes x_j) \;=\; 
	\begin{cases} 
	0 & \text{if $\tau = \tfrac{a}{n}+\tfrac{b}{n}\eta$ for some $a,b \in \ZZ$ and $i-j=(k'-1)b$ in $\ZZ_n$,}
	\\
	L_{ij}(\tau) & \text{otherwise.}
	\end{cases}
$$	
\end{lemma}
\begin{pf}
Suppose $R_\tau(\tau)(x_i \otimes x_j) \ne 0$. There is a neighborhood $U \subseteq E$ of $\tau+\Lambda$ on which the function
$z \mapsto \CC .R_z(z)(x_i \otimes x_j)$ is a non-vanishing continuous function $U \to \PP(V \otimes V)$. Since this function agrees with 
the function $z \mapsto L_{ij}(z)$ on $U \cap (E-E[n])$, these two functions agree on $U$. Hence $\CC.R_\tau(\tau)(x_i \otimes x_j) =L_{ij}(\tau)$.

Now we assume that $R_\tau(\tau)(x_i \otimes x_j) = 0$. 

If $\tau \notin \frac{1}{n}\Lambda$, then $r_{ij}(\tau)$ would be non-zero and $R_\tau(\tau)(x_i \otimes x_j)$ would be a non-zero scalar multiple 
of $r_{ij}(\tau)$; but this is not the case, so we conclude that $\tau = \tfrac{a}{n}+\tfrac{b}{n}\eta$ for some $a,b \in \ZZ$. 
Since the term before the $\Sigma$ sign in \cref{eq:odr} has a zero of order $1$ at $z= \tfrac{a}{n}+\tfrac{b}{n}\eta$, 
$\theta_{j-i+r(k-1)}(0)$ must be $0$ whenever $\theta_{j-i-r}(-\tau)\theta_{kr}(\tau)=0$;
i.e., $j-i+r(k-1)=0$ when $j-i-r=b$ and when $kr=-b$ (in $\ZZ_{n}$);  
i.e., $j-i+(j-i-b)(k-1)= j-i-k'b(k-1)= 0$; hence $j-i+(k'-1)b=0$.
\end{pf}

The next proof uses two results that are proved in later sections.

\begin{proposition}
\label{prop.relns.im.R.tau.tau}
For all $\tau \in \CC$, 
$$
\rel_{n,k}(E,\tau) \;=\; \text{the image of $R_\tau(\tau)$}.
$$
\end{proposition}
\begin{pf}
If $\tau \notin \frac{1}{n}\Lambda$, 
then $\CC.R_\tau(\tau)(x_i \otimes x_j) = L_{ij}(\tau)$ for all $i$ and $j$  for which $r_{ij}$ is not identically zero on $\CC-\frac{1}{n}\Lambda$
 so $\im R_\tau(\tau)=\rel_{n,k}(E,\tau)$.
It therefore remains to prove the result when $\tau = \tfrac{a}{n}+\tfrac{b}{n}\eta$ for some $a,b \in \ZZ$. For the rest of the proof we assume that is 
the case. 

If $i-j \ne (k'-1)b$, then $\CC.R_\tau(\tau)(x_i \otimes x_j) = L_{ij}(\tau)$. Hence 
$$
\rel_{n,k}(E,\tau)\;=\;\im R_\tau(\tau)+\sum_{\substack{i,j\in\ZZ_{n}\\  i-j=(k'-1)b}}L_{ij}(\tau).
$$
We will complete the proof by showing that the $L_{ij}(\tau)$'s for which $ i-j=(k'-1)b$ are contained in the sum of the 
$L_{i'j'}(\tau)$'s for which $j'-i'+(k'-1)b\neq 0$.

With that goal in mind, assume  $ i-j=(k'-1)b$. 
By \cref{lem.tw.L},
\begin{align*}
	L_{ij}(\tau)
	& \;=\;(1\otimes S^{-k-1})^{a}(L_{ij}(\tfrac{b}{n}\eta))
	\\
	&\;=\;(1\otimes S^{-k-1})^{a}(1\otimes T^{-k'-1})^{b}(L_{i+b,j+k'b}(0))
	\\
	&\;=\;(1\otimes S^{-k-1})^{a}(1\otimes T^{-k'-1})^{b}(L_{i+b,i+b}(0)).
\end{align*}
By \cref{prop.qnk.poly}\cref{item.prop.qnk.poly.ii},  $L_{i+b,i+b}(0)$ is contained in the 
sum of the $L_{\a\b}(0)$'s for which $\a\neq\b$. Thus $L_{ij}(\tau)$ is contained in 
\begin{equation*}
	\sum_{\a \ne \b} (1\otimes S^{-k-1})^{a}(1\otimes T^{-k'-1})^{b}(L_{\a\b}(0))\;=\; \sum_{\a \ne \b}  L_{\a-b,\b-k'b}(\tau).
\end{equation*}
Set $i':=\a-b$ and $j':=\b-k'b$. Then $\a\neq\b$ implies $j'-i'+(k'-1)b\neq 0$.
\end{pf}

\begin{proposition}\label{prop.rel.n.one.vs.n.k}
When $k=1$, the space $\rel_{n,k}(E,\tau)$ defined in \cref{def.Qnk} is equal to $\rel_{n,1}(E,\tau)$ defined in \cref{defn.rel.n1}.
\end{proposition}
\begin{proof}
In this proof, $\rel_{n,1}(E,\tau)$ denotes the space defined in \cref{defn.rel.n1}. We will show that $\rel_{n,1}(E,\tau)=\rel'_{n,1}(E,\tau)$. Recall that $\rel_{n,1}(E,\tau)$ is the subspace of $\Theta_{n}(\Lambda)^{\otimes 2}$ spanned by $h_{ij}$'s defined in \cref{half.complicated.identity} via the identification $x_i \leftrightarrow \theta_i$. Hence
\begin{equation*}
h_{ij}\;=\;\tfrac{1}{n}\theta(\tfrac{1}{n})\cdots\theta(\tfrac{n-1}{n})\theta_{1}(0)\cdots\theta_{n-1}(0)g(\tau)R_{\tau}(\tau)(x_{i}\otimes x_{j})
\end{equation*}
where
\begin{equation*}
	g(\tau)\;:=\;\frac{\theta(-n\tau)}{\theta_{0}(-\tau)\cdots\theta_{n-1}(-\tau)}.
\end{equation*}
Since both the numerator and denominator of $g(\tau)$ have zeros exactly at $\frac{1}{n}\Lambda$ with multiplicity one, $g$ is a nowhere vanishing holomorphic function on $\CC$. Thus the linear span of $h_{ij}$ is equal to that of $R_{ij}(x_{i}\otimes x_{j})$ for all $\tau$.
\end{proof}

\subsubsection{The third method} 
\label{sect.subspaces}
We write $\Grass(d, W)$ for the Grassmannian of $d$-dimensional subspaces of a finite dimensional vector space $W$.

In \cite{CKS4}, we will show that $Q_{n,k}(E,\tau)$ has the same Hilbert series as the polynomial ring on $n$ variables
when $\tau$ is not a torsion point of $E$. The first step towards this is to determine the dimension of $\rel_{n,k}(E,\tau)$. 
(The results in this paper do not give any information about this, except in some special cases.)
In \cite{CKS4}, we will show that $\dim \rel_{n,k}(E,\tau) =  \binom{n}{2}$ when $\tau \notin E[2n]$.\footnote{\cref{cor.qnk.tors} below 
shows that $\dim \rel_{n,k}(E,\tau)  = \binom{n}{2}$ for all $\tau \in E[n]$.}

Once we know that $\dim \rel_{n,k}(E,\tau)=\binom{n}{2}$ outside a finite set $\cS \subseteq E$, 
the map $\tau \mapsto \rel_{n,k}(E,\tau)$ becomes a morphism 
$E-\cS \to \Grass\big(\binom{n}{2}, V^{\otimes 2}\big)$; that morphism extends in a unique way to a morphism 
$f:E \to \Grass\big(\binom{n}{2}, V^{\otimes 2}\big)$ so we could  use $f(\tau)$ in place of $\rel_{n,k}(E,\tau)$.  
In this subsection we fill in the details of this argument and check that $\rel_{n,k}(E,\tau)$ is contained in $f(\tau)$ (with equality whenever
$\dim \rel_{n,k}(E,\tau)=\binom{n}{2}$). 

Although the next two results are ``standard'' we include proofs for the convenience of the reader.
In them we work over an algebraically closed field $\Bbbk$.

\begin{proposition}
\cite[Prop.~13.4]{saltman}
\label{prop.grass.1}
Let $W$ and $W'$ be finite dimensional $\Bbbk$-vector spaces. Let $d:=\dim W$.
Let $X$ be a variety over $\Bbbk$ and $g:X \to \Hom_\Bbbk(W,W')$ a morphism of varieties.
If $r:=\rank g(x)$ is the same for all $x \in X$, then the maps
\begin{enumerate}
  \item\label{item.grass.im}
 $X \to \Grass(r,W')$, $x \mapsto \im g(x)$, and
  \item\label{item.grass.ker}
 $X \to \Grass(d-r,W)$, $x \mapsto \ker g(x)$,
\end{enumerate}
are morphisms.
\end{proposition}
\begin{pf}
\cref{item.grass.im}
Fix a basis $\{e_1,\ldots,e_d\}$ for $W$.  For each $r$-element subset $I \subseteq \{1,\ldots,d\}$, let 
$$
U_I  \;:= \; \{x \in X \; | \; \{g(x)(e_i) \; |\; i\in I\} \text{ is linearly independent}\}.
$$
The $U_I$'s provide an open cover of $X$. 

Let $p:\Grass(r,W') \stackrel{p}{\longrightarrow} \PP(\bigwedge^r W')$ be the Pl\"ucker embedding, 
$p(\operatorname{span}\{v_1,\ldots,v_r\}):=v_1 \wedge \cdots \wedge v_r$.

The  composition $U_I \longrightarrow \Grass(r,W') \stackrel{p}{\longrightarrow} \PP(\bigwedge^r W')$, 
$$
x  \; \mapsto \; \operatorname{span}\big\{ g(x)(e_i) \; | \; i \in I\} \;=\; \im g(x) \,  \mapsto \, \bigwedge_{i \in I} g(x)(e_i),
$$
is a morphism; the morphisms $U_{I}\to\Grass(r,W')$ agree on their intersections so glue to give a morphism $X \to\Grass(r,W')$.

\cref{item.grass.ker}
The linear map $\Hom(W,W') \to \Hom(W'^*,W^*)$, $T \mapsto T^*$, is a morphism so its composition with $g$; i.e., the map 
$g^*: X \to  \Hom(W'^*,W^*)$, $g^*(x):=g(x)^*$, is a morphism. Since $\ker g(x) = \big(\im g^*(x)\big)^\perp$,
the map $x \mapsto \ker g(x)$ is the composition 
\begin{equation}
\label{eq:grass.1}
x \mapsto g^*(x) \mapsto \im g^*(x) \mapsto \big(\im g^*(x)\big)^\perp.
\end{equation}
The right-most map in \cref{eq:grass.1} is given by the map ${\rm Grass}(r,W') \to {\rm Grass}(d-r,W'^*)$, $W_0 \mapsto W_0^\perp$;
this map is an isomorphism of algebraic varieties (see \cite[(11.8)]{Hassett}, for example) so
the map in \cref{eq:grass.1} is a morphism, as claimed.
\end{pf}

\begin{lemma}
\label{lem.fm.vec.sp}
Let $X$ be a variety over an algebraically closed field $\Bbbk$. 
Let $V$ be a $\Bbbk$-vector space with basis $\{v_1,\ldots,v_n\}$. Fix an integer $m\geq 0$ and let $\l_{ij}$, $1 \le i\le m$, $1\le j \le n$, be regular functions on $X$. For each closed point $x\in X$, define
$$
r_i(x) \; :=\; \sum_{j=1}^n \l_{ij}(x)v_j  
$$
for $1\le i \le m$,  $R(x):= \operatorname{span}\{r_i(x) \; | \; 1\le i \le m\}$, and  $d:=\max\{\dim R(x) \; | \; x\in X\}$.
\begin{enumerate}
\item\label{item.lem.fm.vec.sp.op}
$U := \{ x \in X \; | \; \dim R(x)=d\}$ is a non-empty Zariski-open subset of $X$.
\item\label{item.lem.fm.vec.sp.gr}
The map $f\colon U \to \Grass(d,V)$, $x\mapsto R(x)$,  is a morphism of algebraic varieties.
\item\label{item.lem.fm.vec.sp.ex}
If $X$ is a non-empty Zariski-open subset of a non-singular curve $\overline{X}$, then 
$f$ extends uniquely 
to a morphism $\overline{X}\to\Grass(d,V)$.
\end{enumerate}
\end{lemma}
\begin{proof}
\cref{item.lem.fm.vec.sp.op}
Let $M_{m,n}(\Bbbk)$ denote the space of all $m \times n$ matrices, and let $M(x):=(\l_{ij}(x)) \in M_{m,n}(\Bbbk)$.
Since $R(x)$ is essentially  the image of the map ``left-multiplication by $M(x)$'', the dimension of $R(x)$ is the rank of $M(x)$. 
Since the rank of a matrix is $<s$ if and only if all its $s \times s$ minors vanish, the set of matrices having rank $<s$ is 
a Zariski-closed subset of $M_{m,n}(\Bbbk)$. 
Since the map $X\to M_{m,n}(\Bbbk)$ given by $x\mapsto M(x)$ is a morphism of algebraic varieties, the sets
\begin{align*}
Z_s    \; :=\; & \, \{x \in X \; | \; \rank M(x) <s\}
\\
 \; =\;  & \, \{x \in X \; | \; \dim R(x) <s\}
\end{align*}
are Zariski-closed subsets of $X$. The sets $U_s:=\{x \in X \; | \; \dim R(x) \ge s\}$ are therefore open subsets of $X$.
Since $\dim_\Bbbk(V)<\infty$,   $\max\{\dim R(x) \; | \;   x\in X \}$ exists and  $U=U_{d}$ 
  is a non-empty open subset of $X$. 

\cref{item.lem.fm.vec.sp.gr}
The map $x \mapsto M(x)$ is a morphism $U \to M_{m,n}(\Bbbk)$. Since $R(x)$ ``is'' the image of the map ``multiplication by $M(x)$'', the result follows from
\cref{prop.grass.1}\cref{item.grass.im}.

\cref{item.lem.fm.vec.sp.ex}
See \cite[Prop.~I.6.8]{Hart}.
\end{proof}

\begin{proposition}
\label{prop.comp.two.def}
Let $\cS \subseteq E$ be a finite subset, let  $d:=\max\{\dim\rel_{n,k}(E,\tau) \; | \; \tau \in E-\cS\}$, and let $U=\{\tau \in E-\cS \; | \; 
\dim\rel_{n,k}(E,\tau)=d\}$.
\begin{enumerate}
\item\label{item.prop.comp.two.def.gr} 
The function  $U \to \Grass(d,V^{\otimes 2})$, $\tau \mapsto \rel_{n,k}(E,\tau)$, extends in a unique way 
to a morphism $f\colon E\to\Grass(d,V^{\otimes 2})$.
\item\label{item.prop.comp.two.def.incl}
For all $\tau \in E$, $\rel_{n,k}(E,\tau) \subseteq f(\tau)$.
\item\label{item.prop.comp.two.def.gen}
The set $U$ is a non-empty Zariski-open subset of $E$.
\end{enumerate}
\end{proposition}
\begin{proof}
\cref{item.prop.comp.two.def.gr}
The existence and uniqueness of $f$ follows from \cref{lem.fm.vec.sp} applied to $X=E-\cS \subseteq E= \overline{X}$, the function 
$\tau \mapsto \rel_{n,k}(E,\tau) \subseteq V^{\otimes 2}$,  and the integer $d$.  That lemma also tells us that $U$ is a non-empty Zariski-open subset of
$E-\cS$ and hence of $E$, thus proving \cref{item.prop.comp.two.def.gen}.

\cref{item.prop.comp.two.def.incl}
	It suffices to prove that $L_{ij}(\tau)\subseteq f(\tau)$ for all $\tau\in E$ and $(i,j)$ such that $r_{ij}$ is not identically zero.
	
	Write $W:=V^{\otimes 2}$.
	Let $Y$ be the zero locus in $\PP((\bigwedge^{d}W)\otimes W)$ of the linear map
	\begin{equation*}
		\left({\textstyle\bigwedge^{d}W}\right)\otimes W\to{\textstyle\bigwedge^{d+1}W},\qquad\omega\otimes v\mapsto\omega\wedge v.
	\end{equation*}
	The set 
	$Z:=\{\tau\in E\;|\;L_{ij}(\tau)\subseteq f(\tau)\}$ is the inverse image of $Y$ with respect to the composition
	\begin{equation*}
		\begin{tikzcd}
			E\ar[r,"{(f,\,{L}_{ij})}"] & \Grass(d,W)\times\PP(W)\ar[r,"p\times\id"] & \PP\left({\textstyle\bigwedge^{d}W}\right)\times\PP(W)\ar[r,"\iota"] & \PP\left(\left({\textstyle\bigwedge^{d}W}\right)\otimes W\right)
		\end{tikzcd}
	\end{equation*}
	where $p$ and $\iota$ are the Pl\"ucker and Segre embeddings, respectively. 
	Thus $Z$ is a Zariski-closed subset of $E$. If $\tau \in U\cap(E-\cS)$, then
	\begin{equation*}
		L_{ij}(\tau)\;\subseteq\;\rel_{n,k}(E,\tau)\;=\;f(\tau)
	\end{equation*}
	so $Z\supseteq U\cap(E-\cS)$. Since $U\cap(E-\cS)$ is a Zariski-dense subset of $E-\cS$,  $Z=E$.
\end{proof}

The extension $f$ does not depend on the choice of $\cS$: if  $\cS'$ were another finite subset and $f'$ the associated extension, then $f$ would equal 
$f'$ because $f$ and $f'$ agree on the dense open subset $E-(\cS \cup \cS')$. 

\begin{corollary}\label{cor.R=Rbar}
If $\dim \rel_{n,k}(E,\tau)  = \binom{n}{2}$ for all $\tau \in E$, then the morphism $f$ in   \cref{prop.comp.two.def} is 
$\tau \mapsto \rel_{n,k}(E,\tau)$.
\end{corollary}
\begin{proof} 
This follows from \cref{prop.comp.two.def} with $\cS=\varnothing$ and $d=\tbinom{n}{2}$ since the inclusion 
$\rel_{n,k}(E,\tau)\subseteq f(\tau)$ in \cref{prop.comp.two.def}\cref{item.prop.comp.two.def.incl} must be an equality.
\end{proof}

\subsection{Isomorphisms and anti-isomorphisms}
\label{subsec.first.props}

The next result is stated in \cite[\S 1, Rmk.~3]{FO89}. 
Polishchuk sketches a proof of it at \cite[p.~696]{pl98}; he views the isomorphism in it 
as a ``quantization'' of an isomorphism between certain  moduli spaces of vector bundles on $E$. 

The next two proofs use special cases of the equality 
\begin{equation}
\label{eq:simple.identity}
\frac{\theta_{\a+\b}(z_1+z_2)}{\theta_\a(z_1)\theta_\b(z_2)} \;=\; -\, \frac{\theta_{-\a-\b}(-z_1-z_2)}{\theta_{-\a}(-z_1)\theta_{-\b}(-z_2)}
\end{equation}
(which follows from the fact that $\theta_\a(-z)=-e(-nz+\frac{\a}{n})\theta_{-\a}(z)$).

Recall that $k'$ is the unique integer such that $n >k' \ge 1$ and $kk'=1$ in $\ZZ_n=\ZZ/n\ZZ$.

\begin{proposition}
\label{prop.isom.k.k'}
For all $\tau \in \CC$, there is an isomorphism $\Phi\colon Q_{n,k}(E,\tau)\to Q_{n,k'}(E,\tau)$ given by $\Phi(x_{i})=x_{ki}$.
\end{proposition}

\begin{proof}
Let $\Phi$ be the automorphism of $\CC\langle x_0,\ldots,x_{n-1}\rangle$ defined by $\Phi(x_{i})=x_{k'i}$ for all $i \in\ZZ_n$.
We will show that $\Phi$ sends the relations for  $Q_{n,k}(E,\tau)$ bijectively to the relations for $Q_{n,k'}(E,\tau)$.

Assume $\tau\in\CC-\frac{1}{n}\Lambda$. For all $i,j,r \in \ZZ_n$, let 
$$
c_{ijkr}(\tau) \; =\; \frac{\theta_{j-i+(k-1)r}(0)}{\theta_{j-i-r}(-\tau)\theta_{kr}(\tau)}
\qquad \text{and} \qquad
r_{ijk}(\tau) \; =\; \sum_{r \in \ZZ_n}  c_{ijkr}(\tau)  x_{j-r}x_{i+r}.
$$

Let $i'=kj$, $j'=ki$, and $r'=-k(j-i-r)$. Then
\begin{align*}
c_{ijkr}(\tau) & \; =\; \frac{\theta_{j-i+(k-1)r}(0)}{\theta_{j-i-r}(-\tau)\theta_{kr}(\tau)}
\\
& \;=\;  -\, \frac{\theta_{-(j-i+(k-1)r)}(0)}{\theta_{-(j-i-r)}(\tau) \theta_{-kr}(-\tau)}  \qquad \text{by \cref{eq:simple.identity}}
\\
& \;=\; -\, \frac{\theta_{j'-i'+(k'-1)r'}(0)}{\theta_{j'-i'-r'}(-\tau)\theta_{k'r'}(\tau)}
\\
& \;=\; - \, c_{i'j'k'r'}(\tau).
\end{align*}
Hence
\begin{align*}
\Phi(r_{ijk}(\tau)) & \; =\; \sum_{r \in \ZZ_n} c_{ijkr}(\tau)x_{k(j-r)}x_{k(i+r)}
\\
 & \; =\; -  \sum_{r' \in \ZZ_n}  c_{i'j'k'r'}(\tau)  x_{j'-r'} x_{i'+r'}
 \\
& \;=\; - \, r_{i'j'k'}(\tau).
\end{align*}

Denote by ${L}_{ijk}$ the morphism ${L}_{ij}\colon E\to\PP(V\otimes V)$ for $Q_{n,k}(E,\tau)$. Thus ${L}_{ijk}$ is the unique morphism 
such that 
$$
{L}_{ijk}(\text{the image of $\tau$ in $E$}) \; =\;  \CC. r_{ijk}(\tau)
$$ 
when $\tau\in \CC-\frac{1}{n}\Lambda$. 
The above computation shows that ${L}_{ijk}(\tau)={L}_{i'j'k'}(\tau)$ when $\tau\in  \CC-\frac{1}{n}\Lambda$
whence $L_{ijk}=L_{i'j'k'}$ as morphisms from $E$. 
Therefore $\Phi$ descends to an isomorphism $Q_{n,k}(E,\tau)\to Q_{n,k'}(E,\tau)$.
\end{proof}

\begin{proposition}
\label{prop.anti.isom} 
Let $N \in \operatorname{GL}(V)$ be the map $N(x_\a)=x_{-\a}$. For all $\tau \in \CC$, $N$ extends to algebra isomorphisms
$Q_{n,k}(E,\tau)\to Q_{n,k}(E,-\tau)$ and
 $Q_{n,k}(E,\tau)\to Q_{n,k}(E,\tau)^{\rm op}$. In particular,
   $$
  Q_{n,k}(E,\tau)  \; \cong \;  Q_{n,k}(E,\tau)^{\rm op}  \;=\;  Q_{n,k}(E,-\tau).
  $$
 \end{proposition}
 \begin{proof}
Assume $\tau\in\CC-\frac{1}{n}\Lambda$. By definition, $Q_{n,k}(E,\tau)^{\rm op}$ is $\CC\langle x_0,\ldots,x_{n-1}\rangle$
modulo the relations
$$
 r_{ij}^{\rm op}(\tau) \; :=\;  \sum_{r \in \ZZ_n} \frac{\theta_{j-i+(k-1)r}(0)}{\theta_{j-i-r}(-\tau)\theta_{kr}(\tau)}\, x_{i+r} x_{j-r}, \qquad (i,j) \in \ZZ_n^2.
$$
We have
\begin{align*}
r_{ij}(-\tau) & \; =\;  \sum_{s \in \ZZ_n} \frac{\theta_{j-i+(k-1)s}(0)}{\theta_{j-i-s}(\tau)\theta_{ks}(-\tau)}\, x_{j-s}  x_{i+s} 
\\
& \; =\;  
-\, \sum_{s \in \ZZ_n} 
\frac{\theta_{-j+i-(k-1)s}(0)}   
{ \theta_{-j+i+s}(-\tau) \theta_{-ks}(\tau)} 
x_{j-s}x_{i+s}
\qquad \text{by \cref{eq:simple.identity}}
\\
& \; =\;  
-\, \sum_{r \in \ZZ_n} 
\frac{\theta_{i-j+(k-1)r}(0)}   
{ \theta_{i-j-r}(-\tau) \theta_{kr}(\tau)} 
x_{j+r}x_{i-r}
\qquad \text{by $r:=-s$}
\\
& \;=\; - \, r_{ji}^{\rm op}(\tau).
\end{align*}
Hence $Q_{n,k}(E,-\tau)=Q_{n,k}(E,\tau)^{\rm op}$ for all $\tau \in \CC-\frac{1}{n}\Lambda$.

To show that the map $N:V \to V$, $N(x_\a)=x_{-\a}$, extends to an isomorphism $Q_{n,k}(E,\tau) \to Q_{n,k}(E,-\tau)$ we must show that $\rel_{n,k}(E,-\tau)=\operatorname{span}\{N(r_{ij}(\tau))\}$. This is true because
\begin{align*}
N(r_{ij}(\tau))  & \;=\; \sum_{r \in \ZZ_n} \frac{\theta_{j-i+(k-1)r}(0)}{\theta_{j-i-r}(-\tau)\theta_{kr}(\tau)}\, N(x_{j-r}) N(  x_{i+r}) 
\\
& \;=\;  - \, \sum_{r \in \ZZ_n} \frac{\theta_{-j+i-(k-1)r}(0)}{\theta_{-j+i+r}(\tau)\theta_{-kr}(-\tau)}\, x_{-j+r} x_{-i-r}
\qquad \text{by \cref{eq:simple.identity}}
\\
& \;=\;  - \, \sum_{s \in \ZZ_n} \frac{\theta_{-j+i+(k-1)s}(0)}{\theta_{-j+i-s}(\tau)\theta_{ks}(-\tau)}\, x_{-j-s} x_{-i+s}
\qquad \text{by $s:=-r$}
\\
& \;=\; -\, r_{-i,-j}(-\tau).
\end{align*}
Therefore $N$ extends to an isomorphism $Q_{n,k}(E,\tau) \to Q_{n,k}(E,-\tau)$ for all $\tau \in \CC-\frac{1}{n}\Lambda$.

Let $\sigma$ be the automorphism of $\PP(V \otimes V)$ that sends $\CC.x_\a \otimes x_\b$ to  $\CC.x_{\b} \otimes x_{\a}$. The equality $r_{ij}(-\tau)=r_{ji}^{\rm op}(\tau)$ implies that the morphisms $E\to\PP(V\otimes V)$ given by
$\tau\mapsto {L}_{ij}(-\tau)$ and $\tau\mapsto\sigma({L}_{ji}(\tau))$ 
agree on $E-E[n]$. Since the locus where two morphisms agree is closed,  ${L}_{ij}(-\tau)=\sigma({L}_{ji}(\tau))$ for all $\tau\in E$.
Hence $Q_{n,k}(E,-\tau)=Q_{n,k}(E,\tau)^{\rm op}$ for all $\tau \in E$.

The isomorphism $N$ induces an automorphism $N^{\otimes 2}$ of $\PP(V \otimes V)$ that sends $\CC.x_\a \otimes x_\b$ to 
$\CC.x_{-\a} \otimes x_{-\b}$. The equality  $N^{\otimes 2}(r_{ij}(\tau)) = - r_{-i,-j}(-\tau)$ can be interpreted as saying that
$N^{\otimes 2}({L}_{ij}(\tau))={L}_{-i,-j}(-\tau)$ on $E-E[n]$ so, by the same reasoning as before, this equality holds
for all $\tau \in E$. Hence $N^{\otimes 2}(\rel_{n,k}(\tau))=\rel_{n,k}(-\tau)$.
\end{proof}

\subsubsection{}
\label{ssect.TVdB}
The previous result was proved by Tate and Van den Bergh \cite[Prop.~4.1.1, Rmk.~4.1.2]{TvdB96} when $k=1$. 
They also observe in their Proposition 4.1.1 that   $Q_{n,1}(E,\tau)\cong Q_{n,1}(E,\mu(\tau))$ when $\mu:E \to E$
is an automorphism given by complex multiplication.

\subsection{The Heisenberg group acts as automorphisms of $Q_{n,k}(E,\tau)$}
\label{subsec.Heis.group}

As observed in \cref{lem.Hn.1},  the Heisenberg group generators  act on the basis for $\Theta_n(\L)$ as 
$S\cdot \theta_\a= e\big( \tfrac{\a}{n} \big) \theta_{\a}$, and $T\cdot \theta_\a= \theta_{\a+1}$, and the commutator 
$\epsilon=[S,T]$ acts   as multiplication by 
$$
\omega \;:=\; e\big(\tfrac{1}{n}\big).
$$
We now identify the vector space $V=\operatorname{span}\{x_0,\ldots,x_{n-1}\}$ generating $Q_{n,k}(E,\tau)$ with $\Theta_n(\L)$ by identifying
$x_\a$ with $\theta_\a$. Thus, $V$ also becomes a representation of $H_n$ with the action given by \cref{Hberg.action.on.V} below.
We extend the action of $H_n$ on $V$ to $TV$ in the natural way.

\begin{proposition}\label{prop.qnk.act}
The Heisenberg group $H_n$
acts as degree-preserving $\CC$-algebra automorphisms of $Q_{n,k}(E,\tau)$ by 
\begin{equation}
\label{Hberg.action.on.V}
S \cdot x_i = \omega^i x_i, \qquad  T \cdot x_i = x_{i+1}, \qquad  \epsilon \cdot x_i = \omega x_i.
\end{equation}
\end{proposition}
\begin{proof}
It is easy to show that $S \cdot r_{ij}=\omega^{i+j} r_{ij}$ and $T\cdot r_{ij}=r_{i+1,j+1}$.
Hence $\rel_{n,k}(E,\tau)$ is an $H_n$-subrepresentation of $V \otimes V$ for all $\tau\in E$ and therefore $H_n$ acts as degree-preserving $\CC$-algebra automorphisms of $TV/(\rel_{n,k}(E,\tau))$.
\end{proof}

\subsection{Another set of relations for $Q_{n,k}(E,\tau)$}

One drawback to the presentation of $Q_{n,k}(E,\tau)$ via the relations in \cref{the-relns} is that both $i$ and $j$ appear in  the indices of
the monomials $x_{j-r}x_{i-r}$ and in the indices of the structure constants that are the coefficients of those monomials. 
In particular, if $j-i = j'-i'$, then $r_{ij}$ and $r_{i'j'}$ involve the same monomials 
but it is not immediately clear which coefficients occur before the same monomial; for example, if $j-i = j'-i'=0$ some calculation is required to compare the coefficients of $x_0^2$ in each relation.
There is, however, a different set of relations for $Q_{n,k}(E,\tau)$ with the property that the new relation indexed by 
$(i,j)$ has the following property: 
only $i$ is involved in indices of the structure constants and only $j$ is involved in the indices of the quadratic monomials 
$x_\a x_\b$. 
Ultimately, one sees there are row vectors $A_0,\ldots,A_{n-1}$ in $\CC^n$ and column vectors $B_0,\ldots,B_{n-1}$ 
of quadratic monomials such that the   new relation indexed by $(i,j)$ is the product $A_iB_j$. 

We are grateful to Kevin De Laet for allowing us to include the next result.

\begin{proposition}[De Laet]
\label{prop.new.rel}
Assume $\tau\in\CC-\frac{1}{n}\Lambda$. For each $(i,j) \in \ZZ_n^2$, let 
\begin{equation}\label{eq:rij}
  R_{ij} \; :=\; \sum_{r\in \ZZ_n} e\left(\tfrac{r}{n}\right)  \frac{\theta_{-(k+1)i+(k-1)r}(0)}{\theta_{r+i}(\tau)\theta_{k(r-i)}(\tau)}
\,  x_{j-r}x_{j+r}  
\end{equation}
and 
\begin{equation}\label{eq:r'ij}
R'_{ij} \;:=\; \sum_{r \in \ZZ_n}  e\left(\tfrac{r}{n}\right)  
\frac{\theta_{k-(k+1)i+(k-1)r}(0)}{\theta_{r+i}(\tau)\theta_{k(r-i+1)}(\tau)} \,  x_{j-r}x_{j+r+1}.  
\end{equation}
\begin{enumerate}
\item\label{item.prop.new.rel.act.S}
$(S\otimes S)(R_{ij})=e(\frac{2j}{n})R_{ij}$ and $(S\otimes S)(R'_{ij})=e(\frac{2j+1}{n})R'_{ij}$.
\item\label{item.prop.new.rel.act.T}
$(T\otimes T)(R_{ij})=R_{i,j+1}$ and $(T\otimes T)(R'_{ij})=R'_{i,j+1}$.
\item\label{item.prop.new.rel.odd}
If $n$ is odd, then $\rel_{n,k}(E,\tau)=\operatorname{span}\{R_{ij} \;|\; i,j \in \ZZ_n\}=\operatorname{span}\{R'_{ij} \;|\; i,j \in \ZZ_n\}$.
\item\label{item.prop.new.rel.even}
If $n$ is even, then $\rel_{n,k}(E,\tau) = \operatorname{span}\{R_{ij},\, R'_{ij} \;|\; i,j \in \ZZ_n\}$.
\item\label{item.prop.new.rel.dupe}
If $n$ is even, then $R_{i+\frac{n}{2},j+\frac{n}{2}}= - R_{ij}$ and $R'_{i+\frac{n}{2},j+\frac{n}{2}}= - R'_{ij}$.
\end{enumerate}
\end{proposition}
\begin{proof}
If $v$ and $w$ are non-zero scalar multiples of each other we write $v \equiv w$.

Statements \cref{item.prop.new.rel.act.S} and \cref{item.prop.new.rel.act.T} are immediate.

Since $\theta_\a(-z)=-e\big(-nz+\tfrac{\a}{n}\big)\theta_{-\a}(z)$,
\begin{align*}
r_{i,-i} & \;=\; 
 \sum_{r \in \ZZ_n} \frac{\theta_{-2i+(k-1)r}(0)}{\theta_{-2i-r}(-\tau)\theta_{kr}(\tau)} \,  x_{-i-r}x_{i+r}\\
         &\;=\; 
         \sum_{r \in \ZZ_n} \frac{\theta_{-2i+(k-1)r}(0)}{-e(-n\tau-\frac{2i+r}{n})\theta_{2i+r}(\tau)\theta_{kr}(\tau)}\,  x_{-i-r}x_{i+r}
         \\
         & \; \equiv\;
          \sum_{r \in \ZZ_n} e\left(\tfrac{r}{n}\right) \frac{\theta_{-2i+(k-1)r}(0)}{\theta_{2i+r}(\tau)\theta_{kr}(\tau)} \, x_{-i-r}x_{i+r}
         \\
         & \; =\;
         \sum_{r'\in \ZZ_n} e\left(\tfrac{r'-i}{n}\right)   
          \frac{\theta_{-(k+1)i+(k-1)r'}(0)}{\theta_{r'+i}(\tau)\theta_{k(r'-i)}(\tau)} \, x_{-r'}x_{r'}\qquad\text{(after setting $r'=i+r$)}
         \\
         &\; \equiv\; R_{i0}.
\end{align*}
Using \cref{item.prop.new.rel.act.T} and $T\cdot r_{ij}=r_{i+1,j+1}$, we obtain
$R_{ij}=T^{j}\cdot R_{i0}\equiv T^{j}\cdot r_{i,-i}=r_{j+i,j-i}$.
Therefore
\begin{align*}
	\operatorname{span}\{R_{ij}\;|\;i,j\in\ZZ_{n}\}
	&\;=\;
	\operatorname{span}\{r_{j+i,j-i}\;|\;i,j\in\ZZ_{n}\}\\
	&\;=\;
	\operatorname{span}\{r_{\alpha,\beta}\;|\;\alpha,\beta\in\ZZ_{n},\,\alpha+\beta\in 2\ZZ_{n}\}.
\end{align*}
Similarly,
\begin{align*}
r_{i,1-i} & \; = \; \sum_{r \in \ZZ_n} 
\frac{\theta_{1-2i+(k-1)r}(0)}{\theta_{1-2i-r}(-\tau)\theta_{kr}(\tau)}  \, x_{1-i-r}x_{i+r}
\\
&\; = \; \sum_{r \in \ZZ_n} \frac{\theta_{1-2i+(k-1)r}(0)}{-e\left(-n\tau- \frac{-1+2i+r}{n} \right)\theta_{-1+2i+r}(\tau)\theta_{kr}(\tau)} \,
x_{1-i-r}x_{i+r}
\\
          &\; \equiv \;  \sum_{r \in \ZZ_n}   e\left(\tfrac{r}{n}\right) 
          \frac{\theta_{1-2i+(k-1)r}(0)}{\theta_{-1+2i+r}(\tau)\theta_{kr}(\tau)} \, x_{1-i-r}x_{i+r}
          \\
          &\;=\; \sum_{r'\in \ZZ_n}
           e\left(\tfrac{r'-i+1}{n}\right)
           \frac{\theta_{k-(k+1)i+(k-1)r'}(0)}{\theta_{r'+i}(\tau)\theta_{k(r'-i+1)}(\tau)} \, x_{-r'}x_{r'+1}\qquad\text{(after setting $r'=i+r-1$)}
          \\
          &\;\equiv\; R'_{i0}
\end{align*}
which implies that $R'_{ij}=T^{j}\cdot R'_{i0}\equiv T^{j}\cdot r_{i,1-i}=r_{j+i,j-i+1}$ and
\begin{align*}
	\operatorname{span}\{R'_{ij}\;|\;i,j\in\ZZ_{n}\}
	&\;=\;
	\operatorname{span}\{r_{j+i,j-i+1}\;|\;i,j\in\ZZ_{n}\}\\
	&\;=\;
	\operatorname{span}\{r_{\alpha,\beta}\;|\;\alpha,\beta\in\ZZ_{n},\,\alpha+\beta+1\in 2\ZZ_{n}\}.
\end{align*}

If $n$ is odd, then $2\ZZ_{n}=\ZZ_{n}$ so $\operatorname{span}\{R_{ij}\}=\operatorname{span}\{R'_{ij}\}=\rel_{n,k}(E,\tau)$. 
If $n$ is even, then $\operatorname{span}\{R_{ij},R'_{ij}\}=\rel_{n,k}(E,\tau)$. Hence \cref{item.prop.new.rel.odd} and \cref{item.prop.new.rel.even} hold.

\cref{item.prop.new.rel.dupe}
Assume $n$ is even.
The relation $R_{ij}$ is a linear combination of terms of the form $x_{j-r}x_{j+r}$, $r \in \ZZ_n$, and 
$R_{i+\frac{n}{2},j+\frac{n}{2}}$ is a linear combination of terms of the form $x_{j+\frac{n}{2}-r'}x_{j+\frac{n}{2}+r'}$, $r' \in \ZZ_n$. 
Now $x_{j-r}x_{j+r}=x_{j+\frac{n}{2}-r'}x_{j+\frac{n}{2}+r'}$ if and only if $r'=r+\frac{n}{2}$. 
Let $r'=r+\frac{n}{2}$.
The coefficient of $x_{j-r}x_{j+r}$ in $R_{i+\frac{n}{2},j+\frac{n}{2}}$ is 
\begin{equation*}
e\left(\tfrac{r'}{n}\right)  
\frac{\theta_{-(k+1)(i+\frac{n}{2})+(k-1)r'}(0)}{\theta_{r'+i+\frac{n}{2}}(\tau)\theta_{k(r'-i-\frac{n}{2})}(\tau)}
 \;=\; 
- \,  e\left(\tfrac{r}{n}\right)  \frac{\theta_{-(k+1)i+(k-1)r}(0)}{\theta_{r+i}(\tau)\theta_{k(r-i)}(\tau)}
\end{equation*}
which is equal to the coefficient of $x_{j-r}x_{j+r}$ in $-R_{ij}$. Thus $R_{i+\frac{n}{2},j+\frac{n}{2}} = - R_{ij}$ as claimed. 
A similar argument shows that $R'_{i+\frac{n}{2},j+\frac{n}{2}} = - R'_{ij}$.
\end{proof}

\section{Twisting $Q_{n,k}(E,\tau)$}
\label{sec.twist.FO.alg}

\subsection{Twists}
\label{subsec.twists}

Given a degree-preserving automorphism $\phi\colon A\to A$ of a $\ZZ$-graded algebra over a field $\Bbbk$, the the \textsf{twist}, 
$A^{\phi}$, is the  graded vector space $A$ endowed with the associative multiplication 
\begin{equation*}
	a*b\;=\; \phi^m(a)b
\end{equation*}
when $b\in A_m$.
There is an equivalence $\Gr(A)\equiv \Gr(A^\phi)$ between their categories of graded left modules
\cite[Cor.~8.5]{ATV2}.  

Suppose $A=TV/\fa$ is the tensor algebra of a vector space $V$ modulo a graded ideal $\fa$ in $TV$. The restriction of $\phi$ to $V$ extends to a degree-preserving automorphism of $TV$ that we also denote by $\phi$. Since  $\phi$ descends to $A$, $\phi(\fa)=\fa$.

The next result gives a presentation of  $A^{\phi}$.

\begin{lemma}
Let ${\phi'}\colon TV\to TV$ be the linear map 
$\id_{V}\otimes\phi\otimes\cdots\otimes\phi^{m-1}$ on each $V^{\otimes m}$. 
The identity map $I:V\to V$ extends to a graded algebra isomorphism
\begin{equation*}
	\frac{TV}{\phi'(\fa)} \; \longrightarrow  \; \left(\frac{TV}{\fa}\right)^{\!\phi}.
\end{equation*} 
\end{lemma}
\begin{proof}
	Since $(TV/\fa)^{\phi}$ is generated by $V$ as a $\Bbbk$-algebra, the identity $V\to V$ extends to a graded algebra homomorphism $\rho\colon TV\to(TV/\fa)^{\phi}$. We show that $\ker(\rho)=\phi'(\fa)$.
	
	Let $f\in V^{\otimes m}$ and write $f=\sum_{\sfi}c_{\sfi}x_{i_{1}}\cdots x_{i_{m}}$ where $c_{\sfi}\in\Bbbk$ for each $\sfi=(i_{1},\ldots,i_{m})$. The image of $f$ by $\rho$ is
	\begin{equation*}
		g\;:=\;\sum_{\sfi}c_{\sfi}\cdot x_{i_{1}}*\cdots* x_{i_{m}}\;\in\;(TV/\fa)^{\phi}.
	\end{equation*}
	Thus $\rho(f)=0$ if and only if $g\in \fa$, that is, if and only if
	\begin{equation*}
		\sum_{\sfi}c_{\sfi}\phi^{m-1}(x_{i_{1}})\phi^{m-2}(x_{i_{2}})\cdots\phi(x_{i_{m-1}})x_{i_{m}}\;\in\; \fa.
	\end{equation*}
	Since $\fa$ is stable under $\phi$, this is equivalent to the statement that $\fa$ contains 
	\begin{align*}
		\phi^{-(m-1)}\left(\sum_{\sfi}c_{\sfi}\phi^{m-1}(x_{i_{1}})\phi^{m-2}(x_{i_{2}})\cdots\phi(x_{i_{m-1}})x_{i_{m}}\right)
		&\;=\;\sum_{\sfi}c_{\sfi}x_{i_{1}}\phi(x_{i_{2}})\cdots\phi^{m-2}(x_{i_{m-1}})\phi^{m-1}(x_{i_{m}})\\
		&\;=\;(I \otimes\phi\otimes\cdots\otimes\phi^{m-1})^{-1}(f).
	\end{align*}
	 Therefore $\ker(\rho)=\phi'(\fa)$.
\end{proof}

Consider, for example, a degree-preserving automorphism, $\phi$, of the polynomial ring $\CC[x_0,\ldots,x_{n-1}]$ with its standard grading. If $a$ and $b$ are homogeneous elements of degree 1, then
$$
a*\phi(b) \;=\; \phi(a)\phi(b) \;=\; \phi(b)\phi(a) \;=\; b*\phi(a)
$$
so 
$\CC[x_0,\ldots,x_{n-1}]^\phi$ is the free algebra $\CC \langle x_0,\ldots,x_{n-1} \rangle$ modulo the ideal generated by the elements
$x_i \otimes \phi(x_j) \,-\,  x_j \otimes \phi(x_i)$ for $0 \le i< j \le n-1$.

\subsection{The twists of $Q_{n,k}(E,\tau)$ induced from translations by $n$-torsion points}
\label{ssect.nk.twist.polyring}

In this subsection, we prove that for each $\zeta\in E[n]$, $Q_{n,k}(E,\tau+\zeta)$ is a twist of $Q_{n,k}(E,\tau)$ with respect to an
automorphism that is in the image of the map $H_{n} \to \Aut(Q_{n,k}(E,\tau))$ (see \cref{prop.qnk.act}).

For a degree-preserving automorphism $\phi\colon Q_{n,k}(E,\tau)\to Q_{n,k}(E,\tau)$, the automorphism
$1\otimes\phi\colon V\otimes V\to V\otimes V$
descends to an automorphism $1 \otimes \phi\colon\PP(V\otimes V)\to\PP(V\otimes V)$.

\begin{lemma}\label{lem.tw.L}
For all $\tau \in \CC$, define $L_{ij}(\tau)$ and $\rel_{n,k}(E,\tau)$ as in \cref{def.Qnk}.
	\begin{enumerate}
		\item\label{item.lem.tw.L.S} 
		$L_{ij}\big(\tau+\frac{1}{n}\big)=(1 \otimes S^{-k-1}) (L_{ij}(\tau))$ and
		$Q_{n,k}(E,\tau+\tfrac{1}{n}) \, =\, Q_{n,k}(E,\tau)^{S^{-k-1}}$.
		\item\label{item.lem.tw.L.T} $L_{ij}\big(\tau+\frac{1}{n}\eta\big)=(1 \otimes T^{-k'-1}) (L_{i+1,j+k'}(\tau))$ and
		$Q_{n,k}(E,\tau+\tfrac{1}{n}\eta) \, =\, Q_{n,k}(E,\tau)^{T^{-k'-1}}$.
	\end{enumerate}
\end{lemma}

\begin{proof}
	First we assume $\tau\in\bbC-\frac{1}{n}\L$. In this case,  $L_{ij}(\tau)$ is spanned by $r_{ij}(\tau)$ 
	unless $r_{ij}$ is identically zero.
	
	Since
	\begin{align*}
		r_{ij}(\tau+\tfrac{1}{n})
		& \; = \; \sum_{r\in\ZZ_n}
		\frac{\theta_{j-i+(k-1)r}(0)}{e(-\frac{j-i-r}{n})\theta_{j-i-r}(-\tau)e(\frac{kr}{n})\theta_{kr}(\tau)} \, x_{j-r}\otimes x_{i+r}\\
		& \; = \; \sum_{r\in\ZZ_n}
		e\Big(\tfrac{j-i-(k+1)r}{n}\Big)  \frac{\theta_{j-i+(k-1)r}(0)}{\theta_{j-i-r}(-\tau)\theta_{kr}(\tau)} \, x_{j-r}\otimes x_{i+r}\\
		& \; = \;  e\left(\tfrac{ki+j}{n}\right)
		\sum_{r\in\ZZ_n} e\Big(-\tfrac{(k+1)(i+r)}{n}\Big)\frac{\theta_{j-i+(k-1)r}(0)}{\theta_{j-i-r}(-\tau)\theta_{kr}(\tau)}
		\, x_{j-r}\otimes x_{i+r}\\
		& \; = \;  e(\tfrac{ki+j}{n}) (1 \otimes S^{-k-1}) \left( \sum_{r\in\ZZ_n}\frac{\theta_{j-i+(k-1)r}(0)}{\theta_{j-i-r}(-\tau)\theta_{kr}(\tau)}
		\, x_{j-r} \otimes x_{i+r}\right)\\
		& \; = \;  e(\tfrac{ki+j}{n})(1 \otimes S^{-k-1}) \left( r_{ij}(\tau) \right),
	\end{align*}
	statement \cref{item.lem.tw.L.S} holds for all $\tau \in \bbC-\frac{1}{n}\L$.
	The first step towards proving \cref{item.lem.tw.L.T} is the  calculation
	\begin{align*}
		r_{ij}(\tau+\tfrac{1}{n}\eta)
		& \; =\; \sum_{r\in\ZZ_n}\frac{\theta_{j-i+(k-1)r}(0)}{e(-\tau-\frac{1}{n}\eta+\frac{1}{2n}-\frac{n-1}{2n}\eta)\theta_{j-i-r-1}(-\tau)e(-\tau-\frac{1}{2n}+\frac{n-1}{2n}\eta)\theta_{kr+1}(\tau)} \, x_{j-r}  \otimes x_{i+r}\\
		&\; =\;  e(2\tau+\tfrac{1}{n}\eta)\sum_{r\in\ZZ_n}\frac{\theta_{j-i+(k-1)r}(0)}{\theta_{j-i-r-1}(-\tau)\theta_{kr+1}(\tau)}
		\, x_{j-r}\otimes	x_{i+r}.
	\end{align*}
	Given $(i,j,r)$, there is a unique solution $(i',j',r')$ to the system of equations
	\begin{equation*}
		\begin{cases}
			j-i-r-1=j'-i'-r',\\
			kr+1=kr',\\
			j-r=j'-r',
		\end{cases}
	\end{equation*}
	namely $(i',j',r')=(i+1,j+k',r+k')$. Hence
	\begin{align*}
		\frac{\theta_{j-i+(k-1)r}(0)}{\theta_{j-i-r-1}(-\tau)\theta_{kr+1}(\tau)}\, x_{j-r} \otimes x_{i+r}
		\;=\; \frac{\theta_{j'-i'+(k-1)r'}(0)}{\theta_{j'-i'-r'}(-\tau)\theta_{kr'}(\tau)}\, x_{j'-r'} \otimes x_{i'+r'-k'-1}.
	\end{align*}
	Therefore
	\begin{align*}
		r_{ij}(\tau+\tfrac{1}{n}\eta)\; = \;e\big(2\tau+\tfrac{1}{n}\eta\big) ( 1 \otimes T^{-k'-1}) r_{i+1,j+k'}(\tau).
	\end{align*}
	Hence \cref{item.lem.tw.L.T} holds for all $\tau \in \bbC-\frac{1}{n}\L$.
	
	The argument in the proof of \cref{prop.anti.isom} then shows that \cref{item.lem.tw.L.S} 
	and  \cref{item.lem.tw.L.T} hold for all $\tau\in\CC$.
\end{proof}

Let $\psi\colon H_{n}\to\frac{1}{n}\Lambda$ be the group homomorphism defined  by
\begin{equation}
\label{eq:homom.psi:H_n.to.Aut}
	\psi(S)\;:=\; -\tfrac{1}{n},\qquad\psi(T)\; := \; -  \tfrac{k}{n}\eta,\qquad\psi(\epsilon) \; :=\; 0.
\end{equation}
It induces an isomorphism $H_{n}/\epsilon H_{n} \to E[n]=\frac{1}{n}\Lambda/\Lambda$.

\begin{theorem}\label{thm.qnk.tw}
	Assume $\tau\in E$. For all $\sigma\in H_{n}$,
	\begin{equation*}
		Q_{n,k}(E,\tau+\psi(\sigma))\; =\; Q_{n,k}(E,\tau)^{\sigma^{k+1}}.
	\end{equation*}
	If $a,b\in\bbZ$, then $Q_{n,k}(E,\tau+\tfrac{a}{n}+\frac{b}{n}\eta)$ is the twist of $Q_{n,k}(E,\tau)$ by the automorphism 
	\begin{equation*}
		T^{-(k'+1)b}S^{-(k+1)a}  : x_{i} \, \mapsto \, e(-\tfrac{(k+1)ai}{n})x_{i-(k'+1)b}.
	\end{equation*}
\end{theorem}

\begin{proof}
	Let $\sigma=T^{b}S^{a}$ in $H_{n}/\epsilon H_{n}$. By \cref{lem.tw.L},
	\begin{align*}
		Q_{n,k}(E,\tau+\psi(\sigma))
		& \; = \; Q_{n,k}(E,\tau-\tfrac{a}{n}-\tfrac{bk}{n}\eta)
		\\
		& \;=\; Q_{n,k}(E,\tau)^{T^{-bk(-k'-1)}S^{-a(-k-1)}}
		\\
		& \; = \; Q_{n,k}(E,\tau)^{(T^{b}S^{a})^{k+1}}
		\\
		& \; = \; Q_{n,k}(E,\tau)^{\sigma^{k+1}}.
	\end{align*}
	(Here we can use either $T^{b}S^{a}$ or $S^{a}T^{b}$ because the twist by $\epsilon$ does not change the algebra.)
	The second statement in the proposition is obtained from the first with $\sigma=T^{-bk'}S^{-a}$.
\end{proof}

\subsubsection{More isomorphisms}
\label{ssect.more.isoms}
Note that $k+1$ is a unit in $\ZZ_n$ if and only if $k'+1$ is since $k'+1=k'(k+1)$.

Assume $k+1$ is not a unit in $\ZZ_n$. It follows from the second sentence in \cref{thm.qnk.tw} that 
if $a,b \in \ZZ$ are such that $(k+1)a=(k'+1)b=0$ in $\ZZ_n$, then
\begin{equation*}
Q_{n,k}(E,\tau+\tfrac{a}{n}+\tfrac{b}{n}\eta)\;=\;Q_{n,k}(E,\tau).
\end{equation*}
In \cref{prop.qnk.poly} we will show that $Q_{n,k}(E,0)$ is a polynomial ring on $n$ variables for all $(n,k)$.  Thus, if $a,b \in \ZZ$ are such that $(k+1)a=(k'+1)b=0$ in $\ZZ_n$, then $Q_{n,k}(E,\frac{a}{n}+\frac{b}{n}\eta)$ is a polynomial ring on $n$ variables.  For example, $Q_{35,4}(E,\frac{1}{5} + \frac{2}{5}\eta)$ and $Q_{35,6}(E,\frac{3}{7} + \frac{1}{7}\eta)$ are polynomial rings on 35 variables.

We will see in \cref{prop.n.minus.one} that $Q_{n,n-1}(E,\tau)=\bbC[x_{0},\ldots,x_{n-1}]$ for all $\tau$. In that case $k+1=0$ in $\ZZ_n$ so adding an $n$-torsion point to $\tau$ does not change the relations.  However, twisting $\bbC[x_{0},\ldots,x_{n-1}]$ by $S$ (or $T$) \emph{does} change the relations.

\section{$Q_{n,k}(E,\tau)$ for some special $k$'s and $\tau$'s}
\label{sec.FO.alg.tors.pts}

In this section, we use the definition of $\rel_{n,k}(E,\tau)$ as the linear span of the lines $L_{ij}(\tau) \subseteq V^{\otimes 2}$.

In \cref{prop.qnk.poly} we prove the assertion in \cite[\S1.2, Rmk.~1]{FO89} and  \cite[\S 3]{Od-survey}
 that $Q_{n,k}(E,0)$ is a polynomial ring on $n$ variables. 
 It follows from this  and \cref{thm.qnk.tw} that $Q_{n,k}(E,\tau)$ is a twist of that polynomial ring when $\tau\in E[n]$.
In particular, $\dim \rel_{n,k}(E,\tau) = \binom{n}{2}$ when $\tau \in E[n]$.

\subsection{$Q_{n,k}(E,0)$ is a polynomial ring}

\begin{proposition}\label{prop.qnk.poly}\leavevmode
	\begin{enumerate}
		\item\label{item.prop.qnk.poly.ij} If $i\neq j$, then $L_{ij}(0)=\CC.[x_{i},x_{j}]$.
		\item\label{item.prop.qnk.poly.ii} If $r_{ii}(\tau)$ is not identically zero on $E-E[n]$, then
		\begin{equation*}
			L_{ii}(0) \; =\; \CC.\sum_{r=1}^{\lceil\frac{n}{2}\rceil-1}\frac{\theta_{(k-1)r}(0)}{\theta_{-r}(0)\theta_{kr}(0)}[x_{i-r},x_{i+r}].
		\end{equation*}
		\item\label{item.prop.qnk.poly.zero} $Q_{n,k}(E,0)=\CC[x_{0},\ldots,x_{n-1}]$.
	\end{enumerate}
\end{proposition}

Note that
\begin{equation*}
  \lceil\tfrac{n}{2}\rceil-1=
  \begin{cases}
    \frac{n-1}{2} & \text{if $n$ is odd,}\\
    \frac{n}{2}-1 & \text{if $n$ is even.}
  \end{cases}
\end{equation*}

\begin{proof}
  When taking limits in this proof, we give $E$, $V\otimes V$, and $\PP(V\otimes V)$ the analytic topologies.
  
  \cref{item.prop.qnk.poly.ij} Assume $i\neq j$. We first show that
  \begin{equation}\label{eq.rel.mult.theta}
    \lim_{\tau\to 0}\theta_{0}(\tau)r_{ij}(\tau) \; = \; -[x_{i},x_{j}]
  \end{equation}
  in $V\otimes V$.
  
  Let $\tau \in \CC-\frac{1}{n}\L$. If $\a \in \ZZ_n$, then $\theta_{\alpha}(0)=0$ if and only if $\alpha=0$. Among the terms
  \begin{equation*}
    \theta_{0}(\tau)\frac{\theta_{j-i+(k-1)r}(0)}{\theta_{j-i-r}(-\tau)\theta_{kr}(\tau)} \, x_{j-r}x_{i+r}
  \end{equation*}
  appearing in $\theta_{0}(\tau)r_{ij}(\tau)$, we only have to look at those with $r$ satisfying $\theta_{j-i-r}(0)=0$ or $\theta_{kr}(0)=0$, or equivalently, with $r=0$ or $r=j-i$, since all other terms approach zero as $\tau\to 0$. Therefore the left-hand side of \cref{eq.rel.mult.theta} is equal to
  \begin{align*}
    &\lim_{\tau\to 0}\theta_{0}(\tau)\bigg(\frac{\theta_{j-i}(0)}{\theta_{j-i}(-\tau)\theta_{0}(\tau)}x_{j}x_{i}+\frac{\theta_{k(j-i)}(0)}{\theta_{0}(-\tau)\theta_{k(j-i)}(\tau)}\, x_{i}x_{j}\bigg)\\
    &\;=\;\lim_{\tau\to 0}\bigg(\frac{\theta_{j-i}(0)}{\theta_{j-i}(-\tau)}x_{j}x_{i}+\frac{\theta_{0}(\tau)}{-e(-n\tau)\theta_{0}(\tau)}\cdot\frac{\theta_{k(j-i)}(0)}{\theta_{k(j-i)}(\tau)}\, x_{i}x_{j}\bigg)\\
    &\;=\;-[x_{i},x_{j}].
  \end{align*}
  Here we used $\theta_{\alpha}(-z)=-e\big(-nz+\frac{\alpha}{n}\big)\theta_{-\alpha}(z)$.
  
  Since $[x_{i},x_{j}]\neq 0$ in $V\otimes V$ and $\theta_{0}(\tau)\neq 0$ on a punctured open neighborhood of $0$, we can rephrase \cref{eq.rel.mult.theta} as $L_{ij}(\tau)\to\CC.[x_{i},x_{j}]$ in $\bbP(V\otimes V)$ as $\tau\to 0$ in $E$. On the other hand, the morphism ${L}_{ij}\colon E\to\bbP(V\otimes V)$ in \cref{prop.ext.mor.rel} is continuous with respect to the analytic topologies so $L_{ij}(\tau)\to L_{ij}(0)$ as $\tau\to 0$. The uniqueness of the limit implies the desired conclusion.
  
  \cref{item.prop.qnk.poly.ii} Assume $r_{ii}(\tau)$ is not identically zero. In a similar way to \cref{item.prop.qnk.poly.ij}, it suffices to prove
  \begin{equation*}
    \lim_{\tau\to 0}r_{ii}(\tau)\;=\;\sum_{r=1}^{\lceil\frac{n}{2}\rceil-1}\frac{\theta_{(k-1)r}(0)}{\theta_{-r}(0)\theta_{kr}(0)}\, [x_{i-r},x_{i+r}]
  \end{equation*}
  in $V\otimes V$. By definition, 
  \begin{equation*}
    r_{ii}(\tau)\;=\;\sum_{r\in\bbZ_{n}}\frac{\theta_{(k-1)r}(0)}{\theta_{-r}(-\tau)\theta_{kr}(\tau)}\, x_{i-r}x_{i+r}.
  \end{equation*}
  Since $\theta_{0}(0)=0$, the $r=0$ summand in $r_{ii}(\tau)$ is zero on a punctured open neighborhood of $0$.  When $r\neq 0$, the limit as $\tau \to 0$ of that summand is obtained by substituting $\tau=0$.
  
  Assume $n$ is even. Since $k$ is coprime to $n$, $k$ is odd and $(k-1)\frac{n}{2}=0$ in $\ZZ_n$; the $r=\frac{n}{2}$ summand is therefore zero.
  
  Therefore, in general, $\lim_{\tau\to 0}r_{ii}(\tau)$ is equal to
  \begin{align*}
    &\sum_{r=1}^{\lceil\frac{n}{2}\rceil-1}\bigg(\frac{\theta_{(k-1)r}(0)}{\theta_{-r}(0)\theta_{kr}(0)}x_{i-r}x_{i+r}+\frac{\theta_{(k-1)(-r)}(0)}{\theta_{-(-r)}(0)\theta_{k(-r)}(0)}x_{i-(-r)}x_{i+(-r)}\bigg)\\
    &\;=\;\sum_{r=1}^{\lceil\frac{n}{2}\rceil-1}\bigg(\frac{\theta_{(k-1)r}(0)}{\theta_{-r}(0)\theta_{kr}(0)}x_{i-r}x_{i+r}+\frac{-e(-\frac{(k-1)r}{n})\theta_{(k-1)r}(0)}{(-e(\frac{r}{n}))\theta_{-r}(0)(-e(-\frac{kr}{n}))\theta_{kr}(0)}x_{i+r}x_{i-r}\bigg)\\
    &\;=\;\sum_{r=1}^{\lceil\frac{n}{2}\rceil-1}\frac{\theta_{(k-1)r}(0)}{\theta_{-r}(0)\theta_{kr}(0)}[x_{i-r},x_{i+r}].
  \end{align*}
  
  \cref{item.prop.qnk.poly.zero} This is immediate from \cref{item.prop.qnk.poly.ij} and \cref{item.prop.qnk.poly.ii}.
\end{proof}

\subsection{$\rel_{n,k}(E,\tau)$ and  $Q_{n,k}(E,\tau)$ when $\tau \in E[n]$}

\begin{corollary}\label{cor.qnk.tors}
  If $\zeta\in E[n]$, then $Q_{n,k}(E,\zeta)$ is the twist of the polynomial ring $\bbC[x_{0},\ldots,x_{n-1}]$ by the automorphism $\sigma^{k+1}$ where $\sigma$ is an arbitrary element of $\psi^{-1}(\zeta)\subseteq H_{n}$ and $\psi$ is the homomorphism, in \cref{eq:homom.psi:H_n.to.Aut}.
\end{corollary}
\begin{proof} This is a consequence of \cref{thm.qnk.tw} and \cref{prop.qnk.poly}.
\end{proof}

\begin{corollary}\label{cor.qnk.tors.2}
For all $k$ and all $\tau \in \frac{1}{n}\Lambda$, $\dim \rel_{n,k}(E,\tau) = \binom{n}{2}$.
\end{corollary}

\subsection{$Q_{n,n-1}(E,\tau)$ is a polynomial ring for all $\tau$}
In \Cref{prop.n.minus.one} we apply \cref{prop.qnk.poly,pr:samerk} to prove the assertions in \cite[\S1.2, Rmk.~1]{FO89} and \cite[\S 3]{Od-survey} 
that $Q_{n,n-1}(E,\tau)$ is a polynomial ring in $n$ variables for all $\tau$.

\begin{proposition}\label{pr:samerk}
For all $\tau \in \CC$, $\rel_{n,k}(E,\tau)$ and $\rel_{n,n-k}(E,\tau)$ have the same dimension.
\end{proposition}
\begin{proof}
This is true when $\tau \in \frac{1}{n}\Lambda$ (\Cref{cor.qnk.tors.2}) so we assume that $\tau \in \CC-\frac{1}{n}\Lambda$. 
Now \Cref{prop.new.rel} applies: the relation spaces are the spans of the $R_{ij}$ and $R'_{ij}$ described in that result.

Assume $n$ is odd. For a fixed $j\in \ZZ_n$, the coefficients  in \Cref{eq:rij} are the matrix entries for the linear operator
$T_j$ on $\operatorname{span}\{x_a x_b \; | \;  a+b=2j\in \ZZ_n\}\subseteq V\otimes V$, with respect to the basis $\{ x_{j-i}x_{j+i} \; | \; i\in \ZZ_n\}$, given by the formula
  \begin{equation*}
T_j(x_{j-i}x_{j+i}) \; :=\;  
    \sum_{r\in \ZZ_n} e\left(\tfrac{r}{n}\right)  \frac{\theta_{-(k+1)i+(k-1)r}(0)}{\theta_{r+i}(\tau)\theta_{k(r-i)}(\tau)}x_{j-r}x_{j+r}.
  \end{equation*}
The dimension of $\rel_{n,k}(E,\tau)$ is
  \begin{equation*}
    \sum_{j\in \ZZ_n}\rank T_j, 
  \end{equation*}
  so we will be done once we show that switching between $k$ and $n-k$ does not alter the ranks of the operators $T_j$. To see this, observe that once the $e\left(\tfrac{r}{n}\right)$ factors (which only scale the rows of the matrix) have been removed, the left-over matrix with respective $(r,i)$-entries
  \begin{equation*}
    \frac{\theta_{-(k+1)i+(k-1)r}(0)}{\theta_{r+i}(\tau)\theta_{k(r-i)}(\tau)}
  \end{equation*}
  is simply transposed by the passage from $k$ to $-k$.

  The argument is similar for even $n$, the only difference being that for the coefficients
  \begin{equation*}
    C_{r,i} \; := \; \frac{\theta_{k-(k+1)i+(k-1)r}(0)}{\theta_{r+i}(\tau)\theta_{k(r-i+1)}(\tau)}
  \end{equation*}
  in \Cref{eq:r'ij} (again, after eliminating the exponential factors) the transformation $k \leftrightarrow  -k\in \ZZ_n$ translates to
$C_{r,i}\leftrightarrow C_{i-1,r+1}$. 
  Once more, this does not affect the rank of the matrix with entries $C_{r,i}$.
\end{proof} 

\begin{proposition}\label{prop.n.minus.one}
For all $\tau \in \CC$, $Q_{n,n-1}(E,\tau)=\CC[x_0,\ldots,x_{n-1}]$. 
\end{proposition}
\begin{proof}
By \Cref{prop.FO,cor.qnk.tors}, $\dim \rel_{n,1}(E,\tau)= \dim \Alt^2 V$ for all $\tau \in \CC$. 
By \Cref{pr:samerk}, the same holds for $\rel_{n,n-1}(E,\tau)$. Thus, to prove the proposition is suffices to show that
  \begin{equation}\label{eq:isalt}
    \rel_{n,n-1}(E,\tau) \subseteq \Alt^2 V.
  \end{equation}
We will now do this.
  
  If $\tau \in \frac{1}{n}\Lambda$, then \cref{cor.qnk.tors} implies that $\rel_{n,n-1}(E,\tau)=\Alt^2 V$, so we assume that $\tau\in\CC-\frac{1}{n}\Lambda$ for the rest of the proof.

  Suppose $n$ is odd. The relations $R_{ij}$ in \cref{prop.new.rel} are
  $$
  R_{ij} \;=\;  \sum_{r\in \ZZ_n}  e\left(\tfrac{r}{n}\right) \frac{\theta_{-2r}(0)}{\theta_{i+r}(\tau)\theta_{i-r}(\tau)} \, x_{j-r}x_{j+r}.
  $$
  Since $\theta_{0}(0)=0$, the coefficient of $x_{j}^2$ in $R_{ij}$ is equal to $0$. The coefficient of $x_{j+r}x_{j-r}$ is
  \begin{align*}
    e\left(-\tfrac{r}{n}\right)   \frac{\theta_{2r}(0)}{\theta_{i-r}(\tau)\theta_{i+r}(\tau)}
    \; =\; - \, e\left(\tfrac{r}{n}\right)  \frac{\theta_{-2r}(0)}{\theta_{i+r}(\tau)\theta_{i-r}(\tau)}
  \end{align*}
  which is the negative of the coefficient of $x_{j-r}x_{j+r}$. 
  Hence  $R_{ij} \in \Alt^2 V$.

  Suppose $n$ is even.
  As in the odd case, the coefficient of $x_{j}^2$ in $R_{ij}$ is zero and so is the coefficient of
  $x^2_{\frac{n}{2}+j}$.
  The ``same'' computation shows that  $R_{ij} \in \Alt^2 V$. The coefficient of $x_{j-r}x_{j+r+1}$ in $R'_{ij}$ is
  $$
  e\left(\tfrac{r}{n}\right)   \frac{\theta_{-2r-1}(0)}{\theta_{i+r}(\tau)\theta_{i-r-1}(\tau)}
  $$
  and the coefficient of $x_{j+r+1}x_{j-r} = x_{j-(-r-1)}x_{j+(-r-1)+1}$ is
  \begin{align*}
    e\left(\tfrac{-r-1}{n}\right)    \frac{\theta_{-2(-r-1)-1}(0)}{\theta_{i+(-r-1)}(\tau)\theta_{i-(-r-1)-1}(\tau)}
    & \; =\;
      e\left(\tfrac{-r-1}{n}\right)    \frac{\theta_{2r+1}(0)}{\theta_{i-r-1}(\tau)\theta_{i+r}(\tau)}
    \\
    & \; =\;
      -\,e\left(\tfrac{-r-1}{n}\right)  \frac{e\left(\frac{2r+1}{n}\right)\theta_{-2r-1}(0)}{\theta_{i-r-1}(\tau)\theta_{i+r}(\tau)}\\
    & \; =\;
      -\,e\left(\tfrac{r}{n}\right) \frac{\theta_{-2r-1}(0)}{\theta_{i-r-1}(\tau)\theta_{i+r}(\tau)}\, .
  \end{align*}
  Hence  $R'_{ij} \in \Alt^2 V$. This concludes the proof of \Cref{eq:isalt} and therefore that of the proposition. 
\end{proof}

\subsection{The relations $\rel_{n,1}(E,\tau)$ and  the structure of $Q_{n,1}(E,\tau)$ when $\tau \in E[2]$}
Since $Q_{n,k}(E,0)$ is a polynomial ring for all $(n,k,E)$  one might expect that $Q_{n,k}(E,\tau)$
is only moderately non-commutative when $\tau$ is a 2-torsion point on $E$. Kevin De Laet proved a decisive result in this
direction when $k=1$: if $n$ is an odd prime and $\tau \in \frac{1}{2}\Lambda$, then
$Q_{n,1}(E,\tau)$ is a Clifford algebra  \cite{KDL14}. 
The first step towards that result is part \cref{item.KDL.odd} of the following observation.

\begin{proposition}
\label{lem.KDL}\leavevmode
\begin{enumerate}
  \item\label{item.KDL.odd}
  If $n$ is odd and $\tau \in \frac{1}{2}\Lambda-\Lambda$, then 
$\rel_{n,1}(E,\tau) \subseteq \operatorname{span}\{x_\a x_\b+x_\b x_\a \; |\; \a,\b \in \ZZ_n\}$.
  \item\label{item.KDL.even}
If $n$ is even and $\tau \in \frac{1}{2}\Lambda$, then $Q_{n,1}(E,\tau)$ is a polynomial ring.
\end{enumerate}
\end{proposition}
\begin{proof}
\cref{item.KDL.odd}
The hypothesis ensures that $\tau \notin \frac{1}{n}\Lambda$. Hence, by \cref{ssect.k=1.relns}, 
$Q_{n,1}(E,\tau)$ is defined by the relations 
\begin{equation}
\label{eq:k=1.relns}
\sum_{r \in \ZZ_n}   \frac{    x_{j-r}x_{i+r}  } {  \theta_{j-i-r}(-\tau) \theta_r(\tau)  } \;=\; 0, \qquad  i \ne j.
\end{equation}

Let $\l \in \L$ be such that $-\tau=\tau+\lambda$.

Fix $\a,\b \in \ZZ_n$.
The word $x_\a x_\b$ appears in the left-hand side of  \cref{eq:k=1.relns} if and only if 
there is an $r \in \ZZ_n$ such that $j-r=\a$ and $i+r=\b$, i.e., if and only if
$j-\a=\b-i$; i.e., if and only if $\a+\b=i+j$. Thus  $x_\a x_\b$ appears in the left-hand side of  
\cref{eq:k=1.relns} if and only if $x_\b x_\a $ does.

For the rest of the proof assume $\a+\b=i+j$. 
Let $r,r' \in \ZZ_n$ be such that $x_\a x_\b = x_{j-r}x_{i+r}$ and $x_\b x_\a = x_{j-r'}x_{i+r'}$; 
then $r=j-\a$ and  $r'=j-\b$, so $r+r'=j-i$. 
To prove the lemma it suffices to show that the coefficients of 
$x_{j-r} x_{i+r}$ and $x_{j-r'}x_{i+r'}$  in \cref{eq:k=1.relns} are the same. 

The reciprocals of those  coefficients are 
$\theta_{j-i-r}(-\tau) \theta_{r}(\tau)$ and $\theta_{j-i-r'}(-\tau) \theta_{r'}(\tau)$, respectively. 
But  $\theta_{j-i-r'}(-\tau) \theta_{r'}(\tau)=\theta_{r}(-\tau) \theta_{j-i-r}(\tau)$, so the coefficients are the same if and only if 
$$
\theta_{j-i-r}(\tau+\lambda) \theta_{r}(\tau) 
= 
\theta_{r}(\tau+\lambda) \theta_{j-i-r}(\tau)
$$
i.e., if and only if 
$$
\frac{ \theta_{j-i-r}(\tau+\lambda) } 
        { \theta_{r}(\tau+\lambda)  } 
=
\frac{ \theta_{j-i-r}(\tau) } 
        { \theta_{r}(\tau)  } .
$$
These are equal: since $\theta_{j-i-r}$ and $\theta_{r}$ belong to $\Theta_n(\Lambda)$,
$$
\frac{ \theta_{j-i-r} } { \theta_{r}} 
$$
is a well-defined (meromorphic) function on $E$. 

\cref{item.KDL.even}
There are integers $a$ and $b$ such that $\tau=\frac{a}{n}+\frac{b}{n}\eta$ and $2a=2b=0$ in $\ZZ_n$. In particular, $(k+1)a=(k'+1)b=0$ so,
as  noted in \cref{ssect.more.isoms}, $Q_{n,1}(E,\tau)$ is a polynomial ring.
\end{proof}

\appendix

\section{Quasi-periodic functions}
\label{sect.appx}

A function $f$ satisfying the hypotheses of the following lemma is called a {\sf theta function of order $c\eta_1-a\eta_2$} 
with respect to $\Lambda$. Thus a theta function of order $r$ has exactly $r$ zeros (counted with multiplicity) in
every fundamental parallelogram  for $\Lambda$.

\begin{lemma} 
\label{lem.theta.fns}
Assume $\Lambda =\ZZ\eta_1 + \ZZ\eta_2$ is a lattice in $\CC$ such that $\Im(\eta_2/\eta_1)>0$, 
and suppose $f$ is a non-constant holomorphic function on $\CC$.
If there are constants $a,b,c,d \in \CC$ such that 
\begin{align*}
f(z+\eta_1) & \;=\; e^{-2\pi i (az+b)}f(z) \qquad \hbox{and} \qquad
\\
f(z+\eta_2) & \;=\; e^{-2\pi i (cz+d)}f(z), 
\end{align*}
then
\begin{enumerate}
  \item 
 $c \eta_1 - a\eta_2 \in \ZZ_{\geq 0}$, and
  \item 
  $f$ has $c \eta_1 - a\eta_2$ zeros (counted with multiplicity) in every fundamental parallelogram for $\Lambda$, and
  \item 
the sum of those zeros is ${{1}\over{2}}(c \eta_1^2 - a\eta_2^2 ) + (c-a)\eta_1\eta_2 + b\eta_2 - d\eta_1 $ modulo $\Lambda$. 
\end{enumerate}
\end{lemma}
\begin{proof} 
Since $f$ is holomorphic, and not identically zero, it has finitely many zeros in every compact region of $\CC$. 
Hence we can, and do, choose a fundamental parallelogram for $\Lambda$ such that no zeros of $f$ lie on its boundary. 
Because $\Im(\eta_2/\eta_1)>0$, the vertices of such a parallelogram can be labeled $A,B,C,D$ in a counterclockwise direction with 
$A=r$, $B=r+\eta_1$, $C=r+\eta_1+\eta_2$, and $D=r+\eta_2$.  

The number of zeros of $f$ in the parallelogram $ABCD$ is ${{1}\over {2\pi i}} \int_{ABCD}\frac{f'(z)}{f(z)}  \, dz.$
It follows from the translation properties of $f$ that
$$
{{f^\prime(z+\eta_1)}\over{f(z+\eta_1)}} \;=\; {{f^\prime(z)}\over{f(z)}} - 2\pi i a
$$
 and 
$$
{{f^\prime(z+\eta_2)}\over{f(z+\eta_2)}} \;=\; {{f^\prime(z)}\over{f(z)}} - 2\pi i c.
$$
 Hence
\begin{align*}
\int_{AB} {{f^\prime(z)}\over{f(z)}}\, dz +  \int_{CD} {{f^\prime(z)}\over{f(z)}}\, dz  
& \;=\; 
\int_r^{r+\eta_1} \left( {{f^\prime(z)}\over{f(z)}} - {{f^\prime(z+\eta_2)}\over{f(z+\eta_2)}} \right) dz
\\
& \;=\;  2\pi i c \eta_1 \phantom{\Bigg)}
\end{align*}
and
\begin{align*}
\int_{AD} {{f^\prime(z)}\over{f(z)}}\, dz +  \int_{CB} {{f^\prime(z)}\over{f}(z)}\, dz   
& \;=\; 
\int_r^{r+\eta_2} \left( {{f^\prime(z)}\over{f(z)}} - {{f^\prime(z+\eta_1)}\over{f(z+\eta_1)}} \right) dz
\\
& \;=\; 2\pi i a \eta_2.     \phantom{\Bigg)}
\end{align*} 
The number of zeros of $f$ in the parallelogram $ABCD$ is therefore $c\eta_1- a\eta_2$.

The sum of these zeros is ${{1}\over{2\pi i}} \int_{ABCD}z{{f^\prime(z)} \over {f(z)}}dz.$ Now
\begin{align*}
\int_{DA} z {{f^\prime(z)}\over{f(z)}}\, dz + \int_{BC} z {{f^\prime(z)}\over{f(z)}}\, dz 
& \; = \;   \int_r^{r+\eta_2} \left( - z {{f^\prime(z)}\over{f(z)}}  +    (z+\eta_1) {{f^\prime(z+\eta_1)}\over{f(z+\eta_1)}} \right)dz 
\\
& \; = \;   \int_r^{r+\eta_2} \left( - z {{f^\prime(z)}\over{f(z)}}  +    (z+\eta_1) \left( {{f^\prime(z)}\over{f(z)}} -2\pi i a \right) \right)dz 
\cr
& \; =\;   \Big[  \eta_1 \log f(z) -2\pi i a\eta_1 z  -\pi i az^2  \Big] ^{r+\eta_2}_r  \phantom{\bigg)}
\cr
& \; =\;   \eta_1 \log\left( \tfrac{f(r+\eta_2)}{f(r)} \right) -2\pi i a\eta_1 \eta_2  -\pi i a (2r \eta_2 +\eta_2^2)  
\cr
& \; =\;  - 2\pi i (cr+d) \eta_1  -\pi i a (2\eta_1\eta_2 + 2r \eta_2 +\eta_2^2) 
\end{align*}
and, similarly, 
\begin{align*}
\int_{AB} z {{f^\prime(z)}\over{f(z)}}\, dz + \int_{CD} z {{f^\prime(z)}\over{f(z)}}\, dz 
&  \;= \;   2\pi i (ar+b) \eta_2  + \pi i c (2\eta_1\eta_2 + 2r \eta_1 +\eta_1^2). 
\end{align*}
Hence the sum of the zeros is ${{1}\over{2}}(c \eta_1^2 - a\eta_2^2 ) + (c-a)\eta_1\eta_2 + b\eta_2 - d\eta_1$ modulo $\Lambda$.
\end{proof}



\begin{thebibliography}{ATVdB91}

\bibitem[AS64]{handbook}
M.~Abramowitz and I.~A. Stegun, \emph{Handbook of mathematical functions with
  formulas, graphs, and mathematical tables}, National Bureau of Standards
  Applied Mathematics Series, vol.~55, For sale by the Superintendent of
  Documents, U.S. Government Printing Office, Washington, D.C., 1964.
  \MR{0167642}

\bibitem[ATVdB90]{ATV1}
M.~Artin, J.~Tate, and M.~Van~den Bergh, \emph{Some algebras associated to
  automorphisms of elliptic curves}, The {G}rothendieck {F}estschrift, {V}ol.\
  {I}, Progr. Math., vol.~86, Birkh\"auser Boston, Boston, MA, 1990,
  pp.~33--85. \MR{1086882 (92e:14002)}

\bibitem[ATVdB91]{ATV2}
\bysame, \emph{Modules over regular algebras of dimension {$3$}}, Invent. Math.
  \textbf{106} (1991), no.~2, 335--388. \MR{1128218 (93e:16055)}

\bibitem[AVdB90]{AV90}
M.~Artin and M.~Van~den Bergh, \emph{Twisted homogeneous coordinate rings}, J.
  Algebra \textbf{133} (1990), no.~2, 249--271. \MR{1067406 (91k:14003)}

\bibitem[CKS19]{CKS2}
A.~{Chirvasitu}, R.~{Kanda}, and S.~P. {Smith}, \emph{{The characteristic
  variety for Feigin and Odesskii's elliptic algebras}}, arXiv:1903.11798v4.

\bibitem[CKS20]{CKS4}
\bysame, \emph{{Elliptic R-matrices and Feigin and Odesskii's elliptic
  algebras}}, arXiv:2006.12283v1.

\bibitem[{De }14]{KDL14}
K.~{De Laet}, \emph{{Character series and Sklyanin algebras at points of order
  2}}, arXiv:1412.7001v2.

\bibitem[Fis10]{Fisher}
T.~Fisher, \emph{Pfaffian presentations of elliptic normal curves}, Trans.
  Amer. Math. Soc. \textbf{362} (2010), no.~5, 2525--2540. \MR{2584609}

\bibitem[FO89]{FO-Kiev}
B.~L. Feigin and A.~V. Odesskii, \emph{Sklyanin algebras associated with an
  elliptic curve}, Preprint deposited with Institute of Theoretical Physics of
  the Academy of Sciences of the Ukrainian SSR (1989), 33 pages.

\bibitem[FO98]{FO98}
\bysame, \emph{Vector bundles on an elliptic curve and {S}klyanin algebras},
  Topics in quantum groups and finite-type invariants, Amer. Math. Soc. Transl.
  Ser. 2, vol. 185, Amer. Math. Soc., Providence, RI, 1998, pp.~65--84.
  \MR{1736164 (2001f:14063)}

\bibitem[FO01]{FO2001}
\bysame, \emph{Functional realization of some elliptic {H}amiltonian structures
  and bosonization of the corresponding quantum algebras}, Integrable
  structures of exactly solvable two-dimensional models of quantum field theory
  ({K}iev, 2000), NATO Sci. Ser. II Math. Phys. Chem., vol.~35, Kluwer Acad.
  Publ., Dordrecht, 2001, pp.~109--122. \MR{1873567}

\bibitem[GH78]{GH78}
P.~Griffiths and J.~Harris, \emph{Principles of algebraic geometry},
  Wiley-Interscience [John Wiley \& Sons], New York, 1978, Pure and Applied
  Mathematics. \MR{507725}

\bibitem[Har77]{Hart}
R.~Hartshorne, \emph{Algebraic geometry}, Springer-Verlag, New York-Heidelberg,
  1977, Graduate Texts in Mathematics, No. 52. \MR{0463157 (57 \#3116)}

\bibitem[Has07]{Hassett}
B.~Hassett, \emph{Introduction to algebraic geometry}, Cambridge University
  Press, Cambridge, 2007. \MR{2324354}

\bibitem[HP18]{HP1}
Z.~Hua and A.~Polishchuk, \emph{Shifted {P}oisson structures and moduli spaces
  of complexes}, Adv. Math. \textbf{338} (2018), 991--1037. \MR{3861721}

\bibitem[LS93]{LS93}
T.~Levasseur and S.~P. Smith, \emph{Modules over the {$4$}-dimensional
  {S}klyanin algebra}, Bull. Soc. Math. France \textbf{121} (1993), no.~1,
  35--90. \MR{1207244 (94f:16054)}

\bibitem[Mum07]{Mum07}
D.~Mumford, \emph{Tata lectures on theta. {I}}, Modern Birkh\"auser Classics,
  Birkh\"auser Boston, Inc., Boston, MA, 2007, With the collaboration of C.
  Musili, M. Nori, E. Previato and M. Stillman, Reprint of the 1983 edition.
  \MR{2352717}

\bibitem[Ode92]{OF92}
A.~V. Odesski, \emph{Rational degeneration of elliptic quadratic algebras},
  Infinite analysis, {P}art {A}, {B} ({K}yoto, 1991), Adv. Ser. Math. Phys.,
  vol.~16, World Sci. Publ., River Edge, NJ, 1992, pp.~773--779.

\bibitem[Ode02]{Od-survey}
A.~V. Odesskii, \emph{Elliptic algebras}, Uspekhi Mat. Nauk \textbf{57} (2002),
  no.~6(348), 87--122. \MR{1991863}

\bibitem[OF89]{FO89}
A.~V. Odesskii and B.~L. Feigin, \emph{Sklyanin elliptic algebras},
  Funktsional. Anal. i Prilozhen. \textbf{23} (1989), no.~3, 45--54, 96.
  \MR{1026987 (91e:16037)}

\bibitem[OF93]{OF93}
\bysame, \emph{Constructions of elliptic {S}klyanin algebras and of quantum
  {$R$}-matrices}, Funktsional. Anal. i Prilozhen. \textbf{27} (1993), no.~1,
  37--45. \MR{1225909 (94m:17019)}

\bibitem[OF95]{OF95}
A.~V. Odesskii and B.~L. Feigin, \emph{Sklyanin's elliptic algebras. {T}he case
  of a point of finite order}, Funktsional. Anal. i Prilozhen. \textbf{29}
  (1995), no.~2, 9--21, 95. \MR{1340300}

\bibitem[OR08]{OR08}
A.~Odesskii and V.~Rubtsov, \emph{Integrable systems associated with elliptic
  algebras}, Quantum groups, IRMA Lect. Math. Theor. Phys., vol.~12, Eur. Math.
  Soc., Z{\"u}rich, 2008, pp.~81--105.

\bibitem[ORTP11a]{ORTP11a}
G.~Ortenzi, V.~Rubtsov, and S.~R. Tagne~Pelap, \emph{Integer solutions of
  integral inequalities and {$H$}-invariant {J}acobian {P}oisson structures},
  Adv. Math. Phys. (2011), Art. ID 252186, 18.

\bibitem[ORTP11b]{ORTP11b}
\bysame, \emph{On the {H}eisenberg invariance and the elliptic {P}oisson
  tensors}, Lett. Math. Phys. \textbf{96} (2011), no.~1-3, 263--284.

\bibitem[Pol98]{pl98}
A.~Polishchuk, \emph{Poisson structures and birational morphisms associated
  with bundles on elliptic curves}, Internat. Math. Res. Notices (1998),
  no.~13, 683--703. \MR{1636545}

\bibitem[Sal99]{saltman}
D.~J. Saltman, \emph{Lectures on division algebras}, CBMS Regional Conference
  Series in Mathematics, vol.~94, Published by American Mathematical Society,
  Providence, RI; on behalf of Conference Board of the Mathematical Sciences,
  Washington, DC, 1999. \MR{1692654}

\bibitem[Skl82]{Skl82}
E.~K. Sklyanin, \emph{Some algebraic structures connected with the
  {Y}ang-{B}axter equation}, Funktsional. Anal. i Prilozhen. \textbf{16}
  (1982), no.~4, 27--34, 96. \MR{684124 (84c:82004)}

\bibitem[SS92]{SS92}
S.~P. Smith and J.~T. Stafford, \emph{Regularity of the four-dimensional
  {S}klyanin algebra}, Compositio Math. \textbf{83} (1992), no.~3, 259--289.
  \MR{1175941 (93h:16037)}

\bibitem[ST94]{ST94}
S.~P. Smith and J.~T. Tate, \emph{The center of the {$3$}-dimensional and
  {$4$}-dimensional {S}klyanin algebras}, Proceedings of {C}onference on
  {A}lgebraic {G}eometry and {R}ing {T}heory in honor of {M}ichael {A}rtin,
  {P}art {I} ({A}ntwerp, 1992), vol.~8, 1994, pp.~19--63. \MR{1273835}

\bibitem[TVdB96]{TvdB96}
J.~T. Tate and M.~Van~den Bergh, \emph{Homological properties of {S}klyanin
  algebras}, Invent. Math. \textbf{124} (1996), no.~1-3, 619--647. \MR{1369430
  (98c:16057)}

\end{thebibliography}
\bibliographystyle{customamsalpha}

\def\cprime{$'$}
\providecommand{\bysame}{\leavevmode\hbox to3em{\hrulefill}\thinspace}
\providecommand{\MR}{\relax\ifhmode\unskip\space\fi MR }
\providecommand{\MRhref}[2]{%
  \href{http://www.ams.org/mathscinet-getitem?mr=#1}{#2}
}
\providecommand{\href}[2]{#2}

\end{document}